\documentclass[11pt]{article}
\usepackage[a4paper,margin=2.6cm]{geometry}

\RequirePackage{amsthm,amsmath,amsfonts,amssymb}
\RequirePackage[numbers]{natbib}
\usepackage{enumitem}
\usepackage{array}

\RequirePackage[colorlinks,citecolor=blue,urlcolor=blue,linkcolor=blue]{hyperref}

\providecommand{\U}[1]{\protect\rule{.1in}{.1in}}

\usepackage{xcolor}

\theoremstyle{plain}
\newtheorem{prop}{Proposition}[section]

\theoremstyle{plain}
\newtheorem{cor}[prop]{Corollary}

\theoremstyle{plain}

\theoremstyle{plain}

\theoremstyle{plain}
\newtheorem{rmk}[prop]{Remark}

\theoremstyle{plain}
\newtheorem{lem}[prop]{Lemma}

\theoremstyle{plain}

\theoremstyle{plain}
\newtheorem{theo}[prop]{Theorem}

\theoremstyle{plain}
\newtheorem{examp}[prop]{Example}

\newcommand{\tr}{\mbox{\rm Tr}}

\newcommand{\EE}{\mathbb{E}}

\newcommand{\NN}{\mathbb{N}}
\newcommand{\PP}{\mathbb{P}}
\newcommand{\QQ}{\mathbb{Q}}
\newcommand{\RR}{\mathbb{R}}

\newcommand{\WW}{\mathbb{W}}

\newcommand{\Ba}{ {\cal B }}

\newcommand{\Da}{ {\cal D }}

\newcommand{\Ka}{ {\cal K }}

\newcommand{\Ga}{ {\cal G }}

\newcommand{\Pa}{ {\cal P }}
\newcommand{\Za}{ {\cal Z }}

\newcommand{\Wa}{ {\cal W }}

\newcommand{\cqfd}{\hfill\blbx \\}
\def\blbx{\hbox{\vrule height 5pt width 5pt depth 0pt}\medskip}

\def \PP{\mathbb{P}}
\def \RR{\mathbb{R}}

\def \EE{\mathbb{E}}

\def \WW{\mathbb{W}}

\usepackage[textwidth=2.5cm, textsize=scriptsize]{todonotes}

\newcommand{\vertiii}[1]{{\left\vert\kern-0.25ex\left\vert\kern-0.25ex\left\vert #1
    \right\vert\kern-0.25ex\right\vert\kern-0.25ex\right\vert}}

\begin{document}

\title{On the Kantorovich contraction of Markov semigroups}
\author{Pierre Del Moral\thanks{Centre de Recherche Inria Bordeaux
Sud-Ouest, \texttt{pierre.del-moral@inria.fr}}
\and
Mathieu Gerber\thanks{School of Mathematics, University of Bristol,
\texttt{mathieu.gerber@bristol.ac.uk}}}
\date{}
\maketitle

\begin{abstract}
We develop an operator-theoretic framework for the contraction
properties of Markov semigroups with respect to a general class of
Kantorovich semi-distances, which includes Wasserstein distances
and weighted total variation norms. The framework is based on
Kantorovich contraction coefficients, which extend Dobrushin's
ergodic coefficient from the total variation distance to general
transport costs. Our main result shows that a geometric drift
condition, combined with a local contraction condition on the
sub-level sets of the Lyapunov function, implies that the Markov
operator is a strict contraction with respect to an explicit
modification of the underlying Kantorovich cost. This yields, in a
unified way, the existence and uniqueness of invariant measures and
exponential decay estimates in Wasserstein distances, for both
discrete time and continuous time Markov semigroups, as well as for
time-varying semigroups. We illustrate the scope of the framework
with (i) Markov transitions on models with boundary states,
including bounded domains with entrance boundaries, (ii) operator
products of a Markov kernel and its adjoint, including two-block
Gibbs samplers, (iii) iterated random functions, including
Euler-type schemes, (iv) diffusion models, including overdamped
Langevin diffusions with potentials that are convex at infinity,
jump-diffusions on compact manifolds, and Jacobi and Riccati matrix
diffusions, and (v) time-varying and nonlinear semigroups, including
Feynman-Kac distribution flows and Sinkhorn bridges. In several of
these settings we obtain new contraction estimates, while in others
we recover known results through short and elementary arguments.

\end{abstract}



 \setcounter{tocdepth}{2}
 \tableofcontents

\section{Introduction}

\subsection{Context}

Understanding the long-time behavior of Markov processes, namely
whether they forget their initial condition, whether they converge
towards an equilibrium or a limiting evolution, at what rate, and in
which metric, is a fundamental problem in probability theory,
stochastic analysis, statistics and computational
science~\cite{douc,meyn-tweedie-2}. This problem arises naturally in
the study of stochastic ordinary and partial differential
equations~\cite{goldys}, Markov chain Monte Carlo algorithms,
nonlinear filtering, stochastic optimization, and many other areas in
which Markovian dynamics play a central role.

The stability analysis of Markov semigroups has been the subject of
extensive research over the past decades, which has led to several
powerful and complementary methodologies. Among the most influential
are spectral methods, probabilistic coupling techniques, and
functional inequalities, including Poincar\'e, Nash and logarithmic
Sobolev inequalities~\cite{bakry,dplm-03,otto}. A further classical
line of research, rooted in the Foster--Lyapunov and Harris-type
theory of Markov
chains~\cite{meyn-tweedie-92,meyn-tweedie-93,meyn-tweedie,meyn-tweedie-2},
combines drift conditions with minorization or contraction conditions
on suitable sets; see also~\cite{douc,fort-02,hairer-mattingly,hairer-mattingly-scheutzow,Quin}.
Each of these approaches is tailored to a specific structural setting,
within which it often provides deep and elegant results by exploiting
the properties of the underlying dynamics. Deploying them outside
their natural domain of applicability is, however, usually
challenging.

Spectral methods are particularly well suited to reversible Markov
semigroups: by exploiting the spectral properties of the
infinitesimal generator in an appropriate Hilbert space, they provide
precise convergence rates through spectral gap
estimates~\cite{bakry,diaconis,saloff-zuniga-3}. Their applicability
is, however, closely tied to reversibility or symmetry with respect to
an invariant measure, assumptions which fail for many non-reversible
models arising in applications. Functional inequalities constitute
another major class of techniques. In the diffusion setting,
logarithmic Sobolev and related inequalities often follow from
Bakry--\'Emery curvature-dimension conditions~\cite{bakry} or from
convexity assumptions on the underlying potential, and they yield
strong convergence estimates and concentration inequalities.
Nevertheless, verifying the required curvature or convexity
conditions can be difficult, or impossible, for degenerate
diffusions, constrained dynamics, jump-diffusion processes,
non-convex energy landscapes, or models evolving on domains with
boundaries. Coupling methods provide a probabilistic alternative,
based on the construction of two copies of the process that approach
one another over time. Depending on the model, synchronous,
reflection, maximal or more sophisticated couplings can lead to sharp
contraction estimates in Wasserstein or total variation
distances~\cite{eberle-cras,eberle-ptrf,eberle,eberle-majka,hairer-mattingly-scheutzow}.
The construction of efficient couplings is, however, often highly
problem dependent: for instance, the analysis of Langevin diffusions
with non-convex potentials, of jump-diffusion processes, or of
constrained stochastic systems frequently requires delicate
reflection or hybrid coupling arguments, which can be difficult to
adapt to new settings.

These difficulties are even more acute for time-varying models. The
absence of stationarity and of invariant probability measures
significantly limits the applicability of spectral techniques, while
the lack of homogeneity complicates both coupling constructions and
the use of functional inequalities.

\subsection{Main contributions}

In this article, rather than exploiting reversibility, curvature or
specially designed couplings, we develop a general theory of
Kantorovich contraction coefficients (or, borrowing the terminology
introduced in~\cite{dob-96}, of Kantorovich norms) to establish the
quantitative stability of Markov dynamics with respect to
Kantorovich-type semi-distances, including Wasserstein distances and
weighted total variation norms as particular cases. These
coefficients extend Dobrushin's ergodic coefficient~\cite{dob-56}
from the total variation distance to general transport costs, in the
spirit of the Kantorovich norms of Markov transitions introduced by
Dobrushin~\cite{dob-70,dob-96} on the basis of the Kantorovich
transportation distance~\cite{kanto-42}; see also the
review~\cite{rachev-84}. In particular, our framework allows one to
compare two different costs, and it applies to operators acting
between different state spaces equipped with different Lyapunov
functions (see Sections~\ref{sec:Kanto} and~\ref{sub:two_op}).

A central insight of the present work is that global exponential
stability in Kantorovich semi-distances emerges from two elementary
structural ingredients: a Lyapunov recurrence mechanism and a local
contraction principle. More precisely, our analysis relies only on a geometric drift
condition and on a local contraction condition, namely
\begin{align*}
&P(V)\leq \epsilon\, V+c
&&\mbox{\rm (drift condition)},\\
&V(x)+V(y)\leq r\ \Longrightarrow\
D_{\kappa}(\delta_xP,\delta_yP)\leq (1-\alpha(r))\,\kappa(x,y)
&&\mbox{\rm (local contraction)},
\end{align*}
for all $x,y$ and all $r\geq r_0$. Here, $V$ is a Lyapunov function,
$\epsilon\in]0,1[$, $c\geq 0$ and $r_0>0$ are constants, $\kappa$ is a
cost function on pairs of states, and $\alpha(r)\in]0,1[$ are local
contraction rates. The notation is that of
Sections~\ref{definitions-sec} and~\ref{sec:Kanto}: $P(V)(x)$ is the
expected value of $V$ after one step of the Markov chain started at
$x$, $\delta_xP$ is the distribution of this step, and $D_\kappa$ is
the Kantorovich semi-distance associated with the cost function $\kappa$
(Section~\ref{sub:kanto-dist}). These two conditions are stated
precisely as conditions~\ref{H1} and~\ref{H2} in
Section~\ref{reg-cond-sec-intro}. The local contraction condition
captures the local geometry of the dynamics, while the drift
condition guarantees that the trajectories repeatedly return to the
regions where this contraction is effective. When $\kappa$ is the
discrete metric, the local contraction condition is implied by a
local minorization condition.

Our main result (Theorem~\ref{th-intro-basic}) shows that, under
these two conditions, $P$ is a strict contraction with respect to the
Kantorovich semi-distance associated with an explicit modification of
the cost $\kappa$; in other words, the corresponding Kantorovich norm
of $P$, defined in Section~\ref{sub:contM}, is strictly smaller than
one. This modified cost is given by
$$
\kappa_{\iota,\rho}(x,y)=\kappa(x,y)^{\iota}\,
\big(1+\rho\,(V(x)+V(y))\big)^{1-\iota},\quad x\neq y,
$$
for suitable parameters $\iota\in[1/2,1[$ and $\rho>0$ which depend
only on the constants in the two conditions. When $\kappa$ is the
discrete metric, this result yields a strict contraction in a
weighted total variation norm, in the spirit of the Harris-type
theorem of~\cite{hairer-mattingly}. A quantitative version of this
result, with explicit constants, is provided in
Section~\ref{sub:quantitative} (Theorem~\ref{th1-intro}), with
proofs in Appendix~\ref{sec:quantitative}. Once the two conditions have been
verified, this result provides a general-purpose mechanism for
deriving: the existence and uniqueness of an invariant measure
(Proposition~\ref{cor:invariant}); exponential decay estimates of the form
$$D_\phi(\mu_1P_n,\mu_2P_n)\leq c\,\lambda^n\,\|\mu_1-\mu_2\|_V,
$$ for
general cost functions $\phi$ and a weighted total variation norm
$\|\cdot\|_V$ (Theorems~\ref{th1-intro-cor} and~\ref{thm:expo});
exponential contraction estimates in Wasserstein distances
(Section~\ref{sub:Wasser}); and their counterparts for continuous
time Markov semigroups (Section~\ref{sub:continuous}) and for
time-varying semigroups (Section~\ref{time-varying}).

The present work extends the weighted total variation (or $V$-norm)
contraction approach developed in~\cite{ado-24,dhj} (see
also~\cite{dm-penev-2017}) to general Kantorovich semi-distances. In
the $V$-norm setting, contraction coefficients coincide with operator
norms on weighted spaces of bounded functions, and the local
contraction condition reduces to a local minorization condition,
which can be handled directly by means of Hahn-Jordan
decompositions. For general Kantorovich semi-distances these tools
are no longer available, and the analysis relies on three new
ingredients: (i) the representation of Kantorovich norms in terms of
Dirac measures (Lemma~\ref{lemma:ced}), whose proof for lower
semi-continuous costs requires an approximation argument to
circumvent measurability issues for optimal couplings; (ii) the interpolated cost $\kappa_{\iota,\rho}$
displayed above, which combines the cost $\kappa$ of the local
contraction condition with a weighted discrete metric associated with
the Lyapunov function, and allows us to combine the drift and local
contraction conditions (Theorem~\ref{th-intro-basic}); and (iii) comparison
principles between Kantorovich norms associated with different costs
(Section~\ref{comparisons-principles-sec}), which allow us to
transfer contraction estimates to Wasserstein-type costs.

As illustrated in Section~\ref{sec:applications}, this allows us to
treat, within the same analytical framework, discrete time and
continuous time models which are usually studied with different
techniques: Markov chains on state spaces with boundary behavior,
forward-backward products of a Markov kernel and its adjoint arising
in Gibbs-type sampling algorithms, iterated random functions, and
diffusion processes such as overdamped Langevin dynamics,
jump-diffusions on compact manifolds, and Jacobi and Riccati matrix
diffusions. Beyond these time-homogeneous settings, the methodology
extends naturally to time-varying semigroups, including nonlinear
Feynman--Kac evolutions and Sinkhorn bridges, which arise
respectively in nonlinear filtering and in entropic optimal
transport.

The generality of the approach leads to new contraction results for
several classes of models whose structure lies outside the natural
scope of the classical approaches discussed above (see
Section~\ref{sec:applications}). The framework also provides a
methodological contribution: for several existing exponential
stability results, it yields proofs that are considerably shorter and
more elementary than the original arguments. A notable illustration
is the case of overdamped Langevin diffusions with a potential
function that is strongly convex outside a ball, for which we recover
exponential Wasserstein contraction estimates through a short,
self-contained argument based solely on the drift and local
contraction conditions described above
(Section~\ref{langevin-sec}).

Our approach also has its own limitations. While the underlying
Kantorovich cost can often be constructed explicitly once an
appropriate Lyapunov function is available, obtaining sharp estimates
of the local contraction (or minorization) rates $\alpha(r)$ may be
challenging. Consequently, although our method provides explicit
exponential contraction rates under readily verifiable assumptions,
these rates are not expected to be optimal in general. The objective
of this work is therefore not to replace the existing methodologies,
but to provide a general framework which can be deployed outside
their natural domains of applicability, and which allows one to
recover many of their stability conclusions through simple and
common arguments.

\subsection{Organization\label{sub:org}}

The article is organized as follows. Section~\ref{definitions-sec}
below introduces the set-up and the notation used throughout the
article. Section~\ref{sec:Kanto} introduces the general class of
Kantorovich norms studied in this article and derives some of their
basic properties. Section~\ref{statement-sec-intro} states the two
main assumptions and the main contraction result of the article,
in both qualitative and quantitative forms.
This result is then used in Section~\ref{sec:select_results} to
derive various stability results, including the existence and
uniqueness of invariant measures, exponential decay estimates in
Wasserstein distances, and results for continuous time Markov
semigroups.

Section~\ref{sec:applications} is devoted to applications. It is
intentionally developed at some length, in order to illustrate the
generality of the framework, and it has two main purposes. First, we
establish new contraction results for classes of Markov models whose
structure lies outside the natural scope of several classical
approaches, including Markov models with non-standard state spaces
and boundary behavior, jump-diffusions, stochastic Riccati
diffusions, forward-backward operator products and two-block Gibbs
samplers, as well as time-varying and nonlinear evolutions such as
Feynman--Kac distribution flows and Sinkhorn bridges. Second, we
show that the abstract theory allows us to recover and unify several
existing stability results, including contraction estimates for
iterated random functions, diffusion semigroups and overdamped
Langevin dynamics.

Appendix~\ref{sec:quantitative} contains the proofs of the
quantitative contraction estimates stated in
Section~\ref{sub:quantitative}, while Appendix~\ref{sec:proof}
collects some additional proofs.

\subsection{Set-up and basic notation}\label{definitions-sec}

Throughout the article, $(S,\psi_S)$ denotes a complete and
separable metric space, $\Ba(S)$ the set of measurable real-valued
functions on $S$, and $\Pa(S)$ the set of probability measures on
$S$. We let $\Vert f\Vert:=\sup_{x\in S}|f(x)|$
denote the uniform norm of a function
$f\in\Ba(S)$, and
$\mbox{\rm osc}(f):=\sup_{x,y\in S}|f(x)-f(y)|$
its oscillation. For all
$\mu_1,\mu_2\in\Pa(S)$, we let $\Pi(\mu_1,\mu_2)$ denote the convex
set of couplings of $\mu_1$ and $\mu_2$, that is, of probability
measures $\pi\in\Pa(S^2)$ with first and second marginals $\mu_1$
and $\mu_2$, respectively.

\subsubsection{Lyapunov functions and weighted norms}
Throughout the article, $V\in\Ba(S)$ is a lower semi-continuous
(abbreviated l.s.c.) function bounded below by a positive constant.
We denote by $\Ba_V(S)$ the subspace of functions $f\in\Ba(S)$ such
that $\Vert f/V\Vert<\infty$, and we set
$$
\Pa_{V}(S):=\left\{\mu\in \Pa(S)~:~\mu(V)<\infty\right\}
\quad \mbox{\rm with}\quad \mu(V):=\int V(x)~\mu(dx).
$$
The weighted total variation norm (a.k.a.~$V$-norm) is defined, for
any $\mu_1,\mu_2\in \Pa_{V}(S)$, by
$$
\Vert \mu_1-\mu_2\Vert_V:=\sup{\left\{ \vert(\mu_1-\mu_2)(f)\vert~:~f\in\Ba_V(S),~\Vert f/V\Vert \leq 1 \right\}}=\vert \mu_1-\mu_2\vert(V),
$$
where $\vert\nu \vert:=\nu_++\nu_-$ stands for the total variation
measure associated with the Hahn-Jordan decomposition
$\nu=\nu_+-\nu_-$ of the signed measure $\nu=\mu_1-\mu_2$. We also
denote by $$\Vert\mu_1-\mu_2\Vert_{\rm tv}:=\sup_{A}|\mu_1(A)-\mu_2(A)|$$
the total variation distance between $\mu_1,\mu_2\in\Pa(S)$, where
the supremum is taken over all measurable subsets $A\subset S$.

\subsubsection{Markov operators}
Throughout the article, $P$ is a Markov operator on $S$. Its right
action $f\mapsto P(f)$ on functions and its left action
$\mu\mapsto\mu P$ on probability measures are defined by
\begin{equation}\label{P-V-f}
P(f)(x):=\int P(x,dy)~f(y)\quad \mbox{\rm and}\quad
(\mu P)(dy):=\int \mu(dx)~P(x,dy).
\end{equation}
We denote by $(P_n)_{n\geq 0}$ the discrete-time semigroup
associated with $P$, defined by $P_0:=I$, the identity operator, and
$P_n:=P_{n-1}P=PP_{n-1}$ for $n\geq 1$, where the products are
understood as compositions of the actions \eqref{P-V-f}.

More generally, let $W$ be another l.s.c.~function bounded below by
a positive constant. We say that a Markov operator $K$ is a
$(V,W)$-operator if
\begin{equation*}
\vertiii{ K}_{V,W}:=
\Vert K(W)/V\Vert=\sup{\left\{ \Vert K(f)/V\Vert~:~f\in\Ba_W(S),~\Vert f/W\Vert\leq 1\right\}}<\infty,
\end{equation*}
and a $V$-operator if it is a $(V,V)$-operator. If $K$ is a
$(V,W)$-operator and $L$ is a $(W,V)$-operator, then $KL$ is a
$V$-operator, since
$$
\vertiii{KL}_{V,V}\leq \vertiii{ K}_{V,W}~\vertiii{ L}_{W,V}<\infty.
$$
\emph{Throughout the article, we assume that $P$ is a $V$-operator.}
In particular, we have $\mu P\in\Pa_V(S)$ for all $\mu\in\Pa_V(S)$.

\subsubsection{Functions growing or vanishing at infinity}
We denote by $\Ba_{\infty}(S)$ the convex cone of locally bounded
functions $f$, bounded below by a positive constant, which grow at
infinity, in the sense that their sub-level sets $\{f\leq r\}$, with
$r\geq\inf_S f$, are non-empty compact sets. This cone is stable
under sums and products. We also let
\begin{equation}\label{vanish-0-set}
\Ba_0(S):= \{f\,:\,1/f\in \Ba_{\infty}(S)\}
\end{equation}
denote the convex cone of bounded positive functions $f$ which
are locally bounded below and vanish at infinity, in the sense that,
for all $0<\epsilon<\Vert f \Vert$, the set
$\{f\geq \epsilon\}$ is a non-empty compact set; this cone is
also stable under sums and products.

\subsubsection{Further notation}
For an indicator function $f=1_{A}$ of a measurable subset
$A\subset S$ and a measure $\nu$, we sometimes write $\nu(A)$
instead of $\nu(1_A)$. A measure $\nu_1$ is absolutely continuous
with respect to a measure $\nu_2$, written $\nu_1\ll \nu_2$, if
$\nu_2(A)=0$ implies $\nu_1(A)=0$ for all measurable subsets
$A\subset S$, in which case we denote by $d\nu_1/d\nu_2$ the
Radon-Nikodym derivative of $\nu_1$ with respect to $\nu_2$. We
write $\nu_1\leq \nu_2$ when $\nu_1(A)\leq \nu_2(A)$ for all
measurable subset $A\subset S$. For $a,b\in\RR$, we set
$a_+:=\max(a,0)$, $a_-:=\max(-a,0)$, $a\vee b:=\max(a,b)$ and
$a\wedge b:=\min(a,b)$.

\section{Kantorovich norms\label{sec:Kanto}}

This section introduces the main objects used throughout the article;
the corresponding notation is summarized in Table~\ref{tab:notation}
at the end of Section~\ref{sub:contM}.

\subsection{Semi-distances and $V$-semi-distances\label{sub:semi_dist}}

A function $\phi$ defined on $S^2$ is called a semi-distance if
(i) $\phi(x,y)\geq 0$ for all $x,y\in S$, (ii) $\phi(x,y)=0$ if and
only if $x=y$, and (iii) $\phi$ is l.s.c.~(and therefore
measurable). If, in addition, $\|\phi/\varpi_{V}\|<\infty$, where
$$
\varpi_V(x,y):=V(x)+V(y),
$$
then $\phi$ is called a $V$-semi-distance. Here and below, ratios of
semi-distances are understood with the convention $0/0:=0$ on the
diagonal, so that, for instance,
$\|\phi/\varpi_V\|=\sup_{x\neq y}\phi(x,y)/\varpi_V(x,y)$.
Semi-distances are stable
under sums and products, and $h\circ \phi$ is a semi-distance
whenever $\phi$ is a semi-distance and $h$ is a continuous and
strictly increasing function on $[0,\infty[$ such that $h(0)=0$.

Two basic examples of $V$-semi-distances used throughout the article
are the discrete metric $\varphi_0(x,y):=1_{x\neq y}$ and, since $V$
is l.s.c., the weighted discrete metric
\begin{align}\label{refvarpiV}
\varphi_V(x,y):=\varphi_0(x,y)~\varpi_V(x,y).
\end{align}
Moreover, for all semi-distances (resp.~$V$-semi-distances)
$\phi$ and $\kappa$, and any $\iota\in [0,1]$, the functions
\begin{align}\label{eq:kappa_iota}
\phi_V(x,y):=\phi(x,y)+\varphi_V(x,y)
\quad\mbox{\rm and}
\quad\kappa_{\iota,V}(x,y):=\kappa(x,y)^{\iota}~\varphi_V(x,y)^{1-\iota}
\end{align}
are also semi-distances (resp.~$V$-semi-distances).

The following rescaled versions of these functions play a central
role in the article. For all $\rho\geq 0$, we let $V_{\rho}$ and
$\varpi_{\rho}$ be the l.s.c.~functions defined by
\begin{equation}
V_{\rho}(x):=\frac{1}{2}+\rho~ V(x)\quad\mbox{\rm and}\quad
\varpi_{\rho}(x,y):= 1+\rho~(V(x)+V(y)),\label{vero}
\end{equation}
and, for all semi-distances $\phi$ and $\kappa$ and any
$\iota\in [0,1]$, we define the semi-distances
\begin{equation}\label{def-phi-rho}
\begin{split}
\varphi_{\rho}(x,y)&:=\varphi_0(x,y)~\varpi_{\rho}(x,y),\\
\phi_{\rho}(x,y)&:=\phi(x,y)+\varphi_{\rho}(x,y)\quad \mbox{\rm and}\quad
\kappa_{\iota,\rho}(x,y):=\kappa(x,y)^{\iota}~\varphi_{\rho}(x,y)^{1-\iota}.
\end{split}
\end{equation}
The functions $(\varpi_{\rho},\varphi_{\rho},\phi_{\rho},
\kappa_{\iota,\rho})$ coincide with the functions
$(\varpi_{V_{\rho}},\varphi_{V_{\rho}},\phi_{V_{\rho}},
\kappa_{\iota,V_{\rho}})$ defined in \eqref{refvarpiV} and
\eqref{eq:kappa_iota}, with $V$ replaced by the rescaled function
$V_{\rho}$. Note also that $\varphi_\rho=\varphi_0$ for $\rho=0$,
and that if $\phi$ (resp.~$\kappa$) is a $V$-semi-distance then so is
$\phi_\rho$ (resp.~$\kappa_{\iota,\rho}$).

\subsection{Kantorovich semi-distances}\label{sub:kanto-dist}

For a given semi-distance $\phi$, the Kantorovich semi-distance
between two probability measures $\mu_1,\mu_2\in\mathcal{P}(S)$ is
defined by
\begin{equation}\label{ref-cost}
D_{\phi}(\mu_1,\mu_2):=\inf_{\pi\in \Pi(\mu_1,\mu_2)}{\pi(\phi)}
\quad \mbox{\rm with}\quad \pi(\phi):=\int\phi(x,y)~\pi(d(x,y)).
\end{equation}
Since $S$ is complete and separable and $\phi$ is l.s.c.~and
non-negative, the infimum in \eqref{ref-cost} is attained by an
optimal coupling (see~\cite[Theorem 4.1]{Villani}). Since
$\phi(x,y)=0$ if and only if $x=y$, this implies that
$D_{\phi}(\mu_1,\mu_2)=0$ if and only if $\mu_1=\mu_2$. Noting that
$D_{\phi}(\delta_{x},\delta_{y})=\phi(x,y)$, the quantity
$\phi(x,y)$ can be interpreted as the cost of transporting $x$ to
$y$. In the optimal transport literature, the couplings in
$\Pi(\mu_1,\mu_2)$ are also called transport plans or Kantorovich
plans.

Kantorovich semi-distances include several classical distances
between probability measures. The Kantorovich distance (a.k.a.~the
Wasserstein $1$-distance) is $D_{\psi_S}$ and, for $p\geq 1$, the
Wasserstein $p$-distance is
$$\WW_{\psi_S,p}(\mu_1,\mu_2):=D_{\psi_S^p}(\mu_1,\mu_2)^{1/p}.
$$ The
total variation distance is $D_{\varphi_0}(\mu_1,\mu_2)=
\|\mu_1-\mu_2\|_{\rm tv}$, while the weighted total variation norm
satisfies
\begin{eqnarray}
D_{\varphi_V}(\mu_1,\mu_2)=\Vert \mu_1-\mu_2\Vert_V,\label{kr2}
\end{eqnarray}
We also have the dual formulation
\begin{equation}\label{kr2-dual}
\Vert \mu_1-\mu_2\Vert_V=\sup{\left\{ \vert(\mu_1-\mu_2)(f)\vert~:~f\in\Ba(S),~\mbox{\rm osc}_V(f) \leq 1 \right\}},
\end{equation}
where the $\varphi_V$-oscillation of a function $f\in \Ba(S)$ is the
Lipschitz constant
$$
\mbox{\rm osc}_V(f):=\sup{\left\{|f(x_1)-f(x_2)|/\varphi_V(x_1,x_2)~:~(x_1,x_2)\in S^2,~x_1\not=x_2 \right\}}.
$$
The formulae \eqref{kr2} and \eqref{kr2-dual} are consequences of the
Kantorovich-Rubinstein theorem~\cite{edwards,kellerer,kellerer2};
see~\cite[Theorem 19.1.7]{douc} and~\cite[Proposition
8.2.16]{dm-penev-2017} for a proof.

The quantity $D_{\phi}(\mu_1,\mu_2)$ may be infinite. This is why we
will work with $V$-semi-distances $\phi$, for which, for all
$\mu_1,\mu_2\in\Pa_V(S)$,
$$
D_{\phi}(\mu_1,\mu_2)\leq
(\mu_1\otimes\mu_2)(\phi)\leq \|\phi/\varpi_V\|\big(\mu_1(V)+\mu_2(V)\big)<\infty,
$$
where $(\mu_1\otimes\mu_2)(d(x,y)):=\mu_1(dx)\mu_2(dy)$ denotes the
product measure.

\subsection{Contraction of Markov transport maps\label{sub:contM}}

Applying the Markov operator $P$ to two probability measures $\mu_1$
and $\mu_2$, we expect the measures $\mu_1 P_n$ and $\mu_2 P_n$ to
become closer as $n$ increases. To quantify this property, for two
$V$-semi-distances $\phi$ and $\psi$, we define the
$(\psi,\phi)$-Kantorovich norm of $P$ by
\begin{equation}\label{def-betaP2}
\beta_{\psi,\phi}(P):=\sup{
\frac{D_{\phi}(\mu_1 P,\mu_2 P)}{D_{\psi}(\mu_1,\mu_2)}}\in[0,\infty],
\end{equation}
where the supremum is taken over all pairs
$(\mu_1,\mu_2)\in \Pa_V(S)^2$ such that $D_{\psi}(\mu_1,\mu_2)>0$.
Since $P$ is a $V$-operator, $\mu_1P$ and $\mu_2P$ belong to
$\Pa_V(S)$, so that $D_{\phi}(\mu_1 P,\mu_2 P)<\infty$; we stress,
however, that the coefficient $\beta_{\psi,\phi}(P)$ may be
infinite (for instance, $\beta_{\psi,\varphi_0}(I)=\infty$ when
 $\psi(x,y)=|x-y|$ 
and $V(x)=1+|x|$ on $S=\RR$).

When $\psi=\phi$, we have, by iteration,
\begin{align*}
D_{\phi}(\mu_1 P_n,\mu_2 P_n)\leq  \beta_{\phi,\phi}(P)^n  D_{\phi}(\mu_1,\mu_2),\quad\forall n\geq 0.
\end{align*}

The following key result shows that the supremum in
\eqref{def-betaP2} can be restricted to Dirac measures. To state it,
for all $V$-semi-distances $\phi$ and $\psi$, we let
$\ell_{\psi,\phi}(P):S^2\rightarrow [0,\infty]$ be the function
defined by
\begin{equation}\label{eq:ell}
\ell_{\psi,\phi}(P)(x,y):=\frac{D_{\phi}(\delta_{x}P,\delta_{y}P)}{D_{\psi}(\delta_{x},\delta_{y})},\quad x,y\in S,~x\neq y.
\end{equation}
\begin{lem}\label{lemma:ced}
Let $\phi$ and $\psi$ be two $V$-semi-distances. Then,
\begin{align*}
\beta_{\psi,\phi}(P)=\sup_{(x,y)\in S^2~:~x\neq y}\ell_{\psi,\phi}(P)(x,y).
\end{align*}
\end{lem}
\begin{proof}
See Appendix \ref{ced-proof}.
\end{proof}

In particular, although its definition involves the set
$\Pa_V(S)$, the coefficient $\beta_{\psi,\phi}(P)$ does not depend
on the function $V$.

In the total variation case $\psi=\phi=\varphi_0$, the
representation of the contraction coefficient in terms of Dirac
measures goes back to Dobrushin~\cite{dob-56}: his ergodic
coefficient $$\alpha_{\rm D}(P)=1-\sup|P(x,A)-P(y,A)|,$$ where the supremum is
taken over all states $x,y$ and all measurable sets $A$, satisfies
$\alpha_{\rm D}(P)=1-N(P)$, where $N(P)$ is the norm of $P$ acting on finite
signed measures with zero total mass, that is,
$N(P)=\beta_{\varphi_0}(P)$; see~\cite[(1.5)--(1.5$''$)]{dob-56}.

When $\phi$ is a metric on $S$,
Dobrushin \cite[Section 14]{dob-96}
defines the Kantorovich norm of $P$ as
$$
\sup_{x\neq y}
\frac{D_\phi(\delta_xP,\delta_yP)}{\phi(x,y)}
$$
and shows that it controls the
contraction of the Kantorovich
distance~\cite{kanto-42} between any
two probability measures
\cite[Proposition 14.3]{dob-96}.
For the discrete metric, this distance
reduces to the total variation distance
\cite[(14.28)]{dob-96}, and the
corresponding ergodic theorem for
inhomogeneous Markov chains is given in
\cite[Theorem 14.5]{dob-96}
and~\cite{dob-70}.
Lemma~\ref{lemma:ced} thus extends
\cite[Proposition 14.3]{dob-96},
which deals with the case where
$\psi=\phi$ is a metric on $S$.
We note, however, that the proof of \cite[Proposition
14.3]{dob-96} is somewhat incomplete, as it relies on the
integration of functions which are not proved to be measurable. We
also note that, by \eqref{kr2-dual} and \cite[Theorem
8.2.20]{dm-penev-2017}, the result of Lemma \ref{lemma:ced} is known
in the particular case where $\psi=\phi=\varphi_V$.

Two choices of the reference semi-distance $\psi$ in
\eqref{def-betaP2} play a particular role in the article. The first
one is $\psi=\phi$, in which case
\begin{equation*}
\beta_{\phi}(P):=\beta_{\phi,\phi}(P)
\end{equation*}
is the $\phi$-Dobrushin contraction coefficient (a.k.a.~Kantorovich
norm~\cite{dob-96}) of $P$; below, we also write $\ell_{\phi}(P)$
instead of $\ell_{\phi,\phi}(P)$. In particular, by \eqref{kr2},
$\beta_{\varphi_V}(P)$ is the $V$-norm contraction coefficient of
$P$. The second one is $\psi=\varphi_V$, in which case, by
\eqref{kr2},
\begin{align*}
\beta_{\varphi_V,\phi}(P)=\sup
\frac{D_{\phi}(\mu_1 P,\mu_2 P)}{\|\mu_1-\mu_2\|_V},
\end{align*}
where the supremum is taken over all pairs
$(\mu_1,\mu_2)\in \Pa_V(S)^2$ such that $\|\mu_1-\mu_2\|_V>0$.

For the convenience of the reader, the main notation introduced in
Section~\ref{definitions-sec} and in the present section is
summarized in Table~\ref{tab:notation}, together with references to
the corresponding definitions.

\begin{table}[htbp]
\small
\renewcommand{\arraystretch}{1.2}
\begin{tabular}{>{\raggedright\arraybackslash}p{0.27\textwidth}%
>{\raggedright\arraybackslash}p{0.47\textwidth}%
>{\raggedright\arraybackslash}p{0.16\textwidth}}
\hline
\textbf{Notation} & \textbf{Meaning} & \textbf{Defined in}\\
\hline
$\Ba_\infty(S)$, $\Ba_0(S)$ & functions growing, resp.~vanishing,
at infinity & \eqref{vanish-0-set}\\
$\|\mu_1-\mu_2\|_V$ & weighted total variation norm
& Sect.~\ref{definitions-sec}\\
$\varpi_V(x,y)$ & $V(x)+V(y)$ & Sect.~\ref{sub:semi_dist}\\
$\varphi_0$, $\varphi_V$ & discrete and weighted discrete metrics
& \eqref{refvarpiV}\\
$\phi_V$, $\kappa_{\iota,V}$ & $\phi+\varphi_V$ and
$\kappa^\iota\varphi_V^{1-\iota}$ & \eqref{eq:kappa_iota}\\
$V_\rho$, $\varpi_\rho$ & $\frac12+\rho V$ and $1+\rho\,\varpi_V$
& \eqref{vero}\\
$\varphi_\rho$, $\phi_\rho$, $\kappa_{\iota,\rho}$ &
$\varphi_V$, $\phi_V$, $\kappa_{\iota,V}$ with $V$ replaced by
$V_\rho$ & \eqref{def-phi-rho}\\
$D_\phi$ & Kantorovich semi-distance & \eqref{ref-cost}\\
$\WW_{\psi_S,p}$ & Wasserstein $p$-distance
& Sect.~\ref{sub:kanto-dist}\\
$\beta_{\psi,\phi}(P)$, $\beta_\phi(P)$ & Kantorovich norms of $P$
& \eqref{def-betaP2}\\
$\ell_{\psi,\phi}(P)$ & contraction ratio on Dirac measures
& \eqref{eq:ell}\\
\ref{H1}, \ref{H2} & drift and local contraction conditions
& Sect.~\ref{reg-cond-sec-intro}\\
\hline
\end{tabular}
\caption{Main notation.}
\label{tab:notation}
\end{table}

\subsection{Some comparison principles}\label{comparisons-principles-sec}

In this subsection, we collect several comparison principles for $(\psi,\phi)$-Kantorovich norms. They are used only in the proofs of some of the results of the article, and readers who are not interested in the proofs may skip this subsection.

Let $(\phi,\varphi,\psi)$ be three $V$-semi-distances. If  $P=K L$ for  some Markov $V$-operators $K$ and $L$, we have 
\begin{align}\label{eq:prodPK}
\beta_{\phi,\psi}(P)\leq \beta_{\phi,\varphi}(K)~\beta_{\varphi,\psi}(L).
\end{align}
Next, we note that
$$
  \Vert \psi/\varphi\Vert<\infty\implies \beta_{\phi,\psi}(P)\leq \Vert \psi/\varphi\Vert~\beta_{\phi,\varphi}(P)
\,\, 
$$
while
\begin{align}\label{eq:Use_BB}
\Vert\varphi/\phi\Vert<\infty\implies 
\beta_{\phi,\psi}(P)\leq~\Vert\varphi/\phi\Vert~\beta_{\varphi,\psi}(P).
\end{align}
Similarly, when $\Vert \psi/\phi\Vert<\infty$  we have both
$$
\beta_{\phi,\psi}(P)\leq \Vert \psi/\phi\Vert~\beta_{\phi}(P)
\quad \mbox{\rm and}\quad
\beta_{\phi,\psi}(P)\leq \Vert \psi/\phi\Vert~\beta_{\psi}(P).
$$
In particular, by applying the latter two results with $\phi=\varphi_V$, we have the implication
$$
 \|\psi/\varphi_V\|<\infty\Longrightarrow
\beta_{\varphi_V,\psi}(P)\leq \|\psi/\varphi_V\|~ \beta_{\varphi_V}(P)\quad
\mbox{\rm and}\quad 
\beta_{\varphi_V,\psi}(P)\leq \|\psi/\varphi_V\|~\beta_{\psi}(P).
$$
In the same vein, one can check that 
$$
\Vert \varphi/\psi\Vert<\infty\Longrightarrow
\beta_{\varphi}(P)\leq \Vert \varphi/\psi\Vert~\beta_{\varphi,\psi}(P)
\quad\mbox{\rm and}\quad
\beta_{\psi}(P)\leq \Vert \varphi/\psi\Vert~\beta_{\varphi,\psi}(P).
$$

The following technical lemma provides  some general scaling properties and   transfer principles of the $\phi$-Kantorovich norm.

\begin{lem}\label{klem}
Let $\varphi$ be a $V$-semi-distance function. Then,
\begin{itemize}
\item For all constants $a> 0$  and  $\iota\in]0,1]$   we have  
\begin{equation}\label{SJ}
\beta_{a\varphi}(P)=\beta_{\varphi}(P)\quad \mbox{and}\quad
\beta_{\varphi^{\iota}}(P)\leq 
\beta_{\varphi}(P)^{\iota}.
\end{equation}
\item For all $V$-semi-distance function $\psi$ and constants $0<a\leq b$ we have
\begin{eqnarray}
a \varphi\leq\psi\leq b \varphi&\Longrightarrow &\beta_{\varphi}(P)\leq (b/a)~ \beta_{\psi}(P).\label{compCost}
\end{eqnarray}
\item For all $V$-semi-distance functions $\varphi\leq \psi $ such that
$\beta_{\psi}(P)<1$ and any constants $\iota\geq 0$ and $\delta\in ]0,1[$  we have
\begin{equation}\label{pre-Tex}
\beta_{\psi}(P)<\frac{\delta}{1+\iota}\Longrightarrow
\beta_{\iota\varphi +\psi}(P)\leq \delta.
\end{equation}
In particular, for $\iota \leq  1-\beta_{\psi}(P)$ we have 
\begin{align}\label{Tex}
\beta_{\iota\varphi+\psi}(P)\leq 1-\left(1-\beta_{\psi}(P)\right)^2.
\end{align}
\end{itemize}
\end{lem}
\begin{proof}
See Appendix \ref{klem-proof}.  
\end{proof}

\section{Strict contraction of Kantorovich norms\label{statement-sec-intro}}

\subsection{Regularity conditions}\label{reg-cond-sec-intro}

All the results presented in this work rest  on two standing hypotheses on the $V$-operator $P$: a Lyapunov drift condition and a local contraction condition, which we now introduce.

\begin{enumerate}[label=(H{{\arabic*}})]
\item\label{H1} There exist constants $\epsilon\in ]0,1[$ and $c\geq 0$ such that $P(V)\leq \epsilon V+c$.
\end{enumerate}

\begin{enumerate}[label=(H{{\arabic*}})]
\setcounter{enumi}{1}
\item\label{H2} There exists a $V$-semi-distance $\kappa$ such that, for some parameter $r_0>0$ and map $\alpha:[r_0,\infty[\,\mapsto\, \alpha(r)\in ]0,1[$,
\begin{align*}
\varpi_V(x,y)\leq r\implies D_\kappa\left(\delta_x P,\delta_y P\right)\leq (1-\alpha(r))~\kappa(x,y),\quad\forall r\geq r_0.
\end{align*}
\end{enumerate}

Condition~\ref{H1} is the standard geometric drift condition. Since
$P_n(V)\leq \epsilon^n V+c/(1-\epsilon)$ for all $n\geq 0$, it
ensures that the operators $P_n$ are $V$-operators with
$\sup_{n\geq 0}\vertiii{P_n}_{V,V}<\infty$. We stress that $V$ is
not assumed to be a Lyapunov function in the usual sense: we do not
require that $V\in \Ba_{\infty}(S)$. Condition~\ref{H2} requires a
Kantorovich-type contraction on each set $\{\varpi_V\leq r\}$,
$r\geq r_0$, with a contraction rate $\alpha(r)$ which may
deteriorate as $r\rightarrow\infty$. Speaking informally, the drift
condition \ref{H1} drives the trajectories back into a region where,
under \ref{H2}, the local contraction is effective; see
Section~\ref{sec:select_results} for a discussion of related
conditions in the literature.

\begin{rmk}\label{rmk:single-r}
The proofs only use condition~\ref{H2}
for a single value $r>2c/(1-\epsilon)$,
together with the inequality
$\beta_\kappa(P)\leq 1$, which follows
from \ref{H2} and is automatically
satisfied when $\kappa=\varphi_0$.
\end{rmk}

In practice, condition~\ref{H1} can often be verified using the class $\Ba_0(S)$ of functions vanishing at infinity, defined in \eqref{vanish-0-set}, via the following simple sufficient condition:
\begin{equation}\label{vanish-0}
\text{ $\exists$ a function $W\in \Ba_{\infty}(S)$ such that } P(W)/W\leq \Theta \text{ for some $\Theta\in \Ba_0(S)$}.
\end{equation}
Indeed, if \eqref{vanish-0} holds then, 
for all $0<\epsilon<\Vert\Theta \Vert\wedge 1$, we have
\begin{align}\label{eq:Lap}
P(W)\leq \epsilon~1_{\{\Theta< \epsilon\}}~W+1_{\{\Theta\geq \epsilon\}}~\Theta~W\leq \epsilon~W+c_{\epsilon}\quad\mbox{\rm with}\quad 
c_{\epsilon}:=\Vert \Theta\Vert~\sup_{\Theta(x)\geq \epsilon}W(x) 
\end{align}
implying that \ref{H1} holds with $V=W$ (and also with $V=1+W$).

When $\kappa$ is the discrete metric $\varphi_0(x,y)=1_{x\neq y}$, condition~\ref{H2} is implied by the following standard local minorization condition:
\begin{equation}\label{loc-min}
\text{$\forall r\geq r_0$, $\exists \alpha(r)\in]0,1[$ and $\exists \nu_r\in\mathcal{P}(S)$ such that }V(x)\leq r\implies \delta_{x}P\geq \alpha(r)\,\nu_r.
\end{equation}
When $V$ is bounded, \eqref{loc-min} becomes a minorization condition on the \emph{whole} state space,
\begin{equation}\label{global-min}
\forall x\in S\qquad
\delta_xP\geq \jmath\,\nu
\quad\mbox{\rm for some $\jmath>0$ and some probability measure $\nu\in\mathcal{P}(S)$,}
\end{equation}
which is well known to yield a contraction in total variation norm (see~\cite{dob-56}):
\begin{equation}\label{ref-tv-contract}
\begin{split}
(\ref{global-min})\quad\Longrightarrow \quad\beta_{\varphi_0}(P)&=\sup_{(\mu_1,\mu_2)\in \Pa(S)^2}
\frac{\|\mu_1 P-\mu_2 P\|_{\rm tv}}{\|\mu_1-\mu_2\|_{\rm tv}}\\
&=\sup_{(x_1,x_2)\in S^2}
\|\delta_{x_1} P-\delta_{x_2} P\|_{\rm tv}\leq 1-\jmath.
\end{split}
\end{equation}

In contrast, when $V\in \Ba_{\infty}(S)$, condition~\ref{H2} becomes a genuinely \emph{local} contraction condition, and \eqref{loc-min} a local minorization condition on the compact sub-level sets of $V$. The next lemma shows that, in this regime, \ref{H2} holds for $\kappa=\varphi_0$ under a mild density-type assumption on $P$.

\begin{lem}\label{lemma:minor}
Assume that $V\in \Ba_{\infty}(S)$ and that $P(x,dy)\geq p(x,y)\eta(dy)$ for some reference measure $\eta$ on $S$ and continuous function $p:S^2\rightarrow ]0,\infty[$. Then condition~\ref{H2} holds with $\kappa=\varphi_0$.
\end{lem}
\begin{proof}
See Appendix \ref{lemminor-proof}.
\end{proof}

We end this subsection with the following technical lemma which, speaking informally, will allow us to prove some results assuming without loss of generality that $V\geq 1/2$.
\begin{lem}\label{lemma:barV}
Assume that \ref{H1} holds for some parameters $\epsilon\in]0,1[$
and $c>0$ (which entails no loss of generality, since \ref{H1} also
holds with any larger constant $c$), and that \ref{H2} is met for
some $V$-semi-distance $\kappa$, parameter $r_0>0$ and map
$\alpha:[r_0,\infty[\rightarrow]0,1[$. Set $a:=\epsilon/(2c)$ and
$\overline{V}:=\frac{1}{2}+a\,V$. Then $\overline{V}$ is an
l.s.c.~function such that $\overline{V}\geq 1/2$, and:
\begin{enumerate}[label=\textnormal{(\roman*)}]
\item a function $\phi$ on $S^2$ is a $V$-semi-distance if and only
if it is a $\overline{V}$-semi-distance;
\item $P(\overline{V})\leq \epsilon\,\overline{V}+\frac{1}{2}$;
\item \ref{H2} holds with $V$ replaced by $\overline{V}$, with the
same $V$-semi-distance $\kappa$, and with $(r_0,\alpha)$ replaced by
$(\bar r_0,\bar\alpha)$, where
\begin{equation}\label{eq:bar-alpha}
\bar r_0:=1+a\,r_0
\quad\mbox{\rm and}\quad
\bar\alpha(r):=\alpha\big((r-1)/a\big),
\quad r\geq \bar r_0.
\end{equation}
\end{enumerate}
In particular, $(\bar r_0,\bar\alpha)$ depend only on
$(r_0,\alpha,\epsilon,c)$.
\end{lem}

\begin{proof}
Like $V$, the function $\overline{V}$ is l.s.c.~and bounded below by
a positive constant. Since $\varpi_{\bar V}=1+a\,\varpi_V$ and
$\varpi_V\geq 2\inf_S V>0$, the functions $\varpi_{\bar V}$ and
$\varpi_V$ are bounded above and below by positive multiples of each
other, which proves (i). By \ref{H1}, we have
$$
P(\overline{V})=\frac{1}{2}+a\,P(V)\leq \frac{1}{2}+a\,\epsilon\, V+a\,c
=\epsilon\,\overline{V}+\frac{1}{2},
$$
since $a\,c=\epsilon/2$, which proves (ii). Finally, let
$r\geq \bar r_0$ and $x,y\in S$ be such that
$\varpi_{\bar V}(x,y)\leq r$. Then $\varpi_V(x,y)\leq (r-1)/a$, with
$(r-1)/a\geq r_0$, and thus, by \ref{H2},
$$
D_\kappa(\delta_xP,\delta_yP)\leq (1-\bar\alpha(r))\,\kappa(x,y),
$$
which proves (iii).
\end{proof}

\subsection{A general local-to-global contraction theorem}\label{general-local-to-global-sec}

The central idea behind our main contraction theorem, Theorem \ref{th-intro-basic}, is to combine a local distance-contraction property with a Lyapunov function, so as to turn the local contraction into a global one. The mechanism is to rescale the underlying Lyapunov function and the associated semi-distances by a parameter $\rho\geq 0$, chosen small enough that the resulting rescaled semi-distance is contracted \emph{everywhere}, not only on the sub-level sets where~\ref{H2} holds directly.

Recall that, for all parameters $\rho\geq 0$ and $\iota\in[0,1]$, and any $V$-semi-distances  $\phi$ and $\kappa$, the $V$-semi-distances  $\varphi_{\rho}$, $\phi_\rho$   and  $\kappa_{\iota,\rho}$ are  defined in \eqref{vero} and \eqref{def-phi-rho}. Then, formally, our objective is to estimate the Kantorovich norms  $\beta_{\varphi_{\rho}}(P)$,  $\beta_{\kappa_{\iota,\rho}}(P)$, and $\beta_{\phi_\rho}(P)$.  Theorem \ref{th-intro-basic} below shows that, under the two regularity conditions~\ref{H1}--\ref{H2} stated in Section~\ref{reg-cond-sec-intro}, the scaling parameter $\rho$ can always be chosen so that these coefficients are strictly less than one.

\begin{theo}\label{th-intro-basic}

Assume that \ref{H1} holds for   some parameters $(\epsilon,c)$.

\begin{enumerate}[label=\textnormal{(\roman*)}]
\item Assume that \ref{H2} holds for some $V$-semi-distance $\kappa$
and some map $\alpha$. Then there exists a constant $\rho>0$,
depending only on $(\alpha,\epsilon,c)$, such that, for every
$\iota\in[1/2,1[$,
\begin{align}\label{re-upsilon.v2}
\beta_{\kappa_{\iota,\rho}}(P)<1 .
\end{align}

\item Assume that \ref{H2} holds with $\kappa=\varphi_0$ and some map $\alpha$. Then, there
exists a constant $\rho>0$, depending only on the parameters $(\alpha,\epsilon,c)$,
such that
\begin{align}\label{re-3.v2}
\beta_{\varphi_{\rho}}(P)<1 .
\end{align}
More generally, under the same assumption, for every $V$-semi-distance
$\phi$ s.t. $\|\phi/\varphi_V\|\leq 1$, there exist constants  $\rho>0$ and $\bar{\delta}>0$, depending only
on the parameters $(\alpha,\epsilon,c)$, such that
\begin{align}\label{re-3-cor.v2}
\beta_{\delta\phi+\varphi_{\rho}}(P)<1,\quad\forall\delta\in[0,\bar{\delta}].
\end{align}
\end{enumerate}

\end{theo}

\begin{proof}
Theorem~\ref{th-intro-basic} is an immediate consequence of the quantitative local-to-global contraction results of Theorem~\ref{th1-intro} (stated in Section~\ref{sub:quantitative}) and Corollary~\ref{cor:strict}, which are proved in Appendix~\ref{sec:quantitative}.
\end{proof}

The result stated in \eqref{re-3.v2} for the weighted discrete
metric $\varphi_{\rho}$ is a slight variation of a contraction
estimate presented in~\cite{ado-24} (itself based on the weighted
norm contraction analysis presented in~\cite{dhj,dm-penev-2017}; see
also Lemma 2.5 in~\cite{ado-24}). It is also closely related to
Theorem 1.3 in~\cite{hairer-mattingly}, which establishes a
contraction in the weighted total variation distance
$\int(1+\beta V)\,d|\mu_1-\mu_2|$ for some $\beta>0$; in the
notation of the present article, this inequality can be written as
$\beta_{\varphi_\rho}(P)<1$ with $\rho=\beta/2$.  To the best of our knowledge the other contraction results presented in Theorem \ref{th-intro-basic} are new.

As pointed out in~\cite{eberle-majka}, in practical applications it is usually unclear how to choose the semi-distance function $\phi$ to ensure that $\beta_\phi(P)<1$. To answer this question, more or less involved and carefully designed  semi-distance functions combined with appropriate couplings depending on the model at hand have been proposed in the literature, see for instance the refined analysis developed in the series of articles~\cite{nawaf-20,eberle,eberle-majka,hairer-mattingly-scheutzow,huang-21}.  It is   interesting to remark  that the semi-distance function in the strict contraction result \eqref{re-upsilon.v2} provided in Theorem \ref{th-intro-basic} is similar to the one in Theorem 2.6 of the article \cite{nawaf-20} (dedicated to Hamiltonian Monte Carlo) and to the one in Theorem 4.8 of \cite{hairer-mattingly-scheutzow}. On the other hand, the semi-distance appearing in \eqref{re-3-cor.v2}  is similar to the one in Theorem 2.3 of \cite{eberle-majka} and in Theorem 2.3 of \cite{huang-21}.

\subsection{Quantitative contraction estimates}\label{sub:quantitative}

In this subsection, we state a quantitative version of Theorem~\ref{th-intro-basic}, with explicit constants.

To state this result, we need to introduce some additional notation.
When \ref{H1} and   \ref{H2} are  met, we set $r_{\epsilon}=1/(1-\epsilon)$ and, for all $r> r_0\vee r_{\epsilon}$, we define the parameters
\begin{equation*}
\delta_{\epsilon}(r):=\frac{1-\epsilon}{2+\epsilon}~\left(1-\frac{r_{\epsilon}}{r}\right)\in]0,1/2]
\quad  \mbox{and}\quad
\rho_{\epsilon}(r):=\frac{1}{1+\epsilon}~\frac{\alpha(r)}{2r}~\in~ ]0,1/2].
\end{equation*}

\begin{theo}\label{th1-intro}
Assume that \ref{H1} is met with $c=1/2$. Then,
\begin{itemize}
\item if \ref{H2} is  met for some $V$-semi-distance  $\kappa$, for all $r>r_{\epsilon}\vee r_0$  and $\iota\in [1/2,1[$   we have  the strict contraction estimate    
\begin{equation}\label{re-upsilon}      
\beta_{\kappa_{\iota,\rho_\epsilon(r)}}(P)\leq 
\Big(1-\alpha(r)\min\big\{\delta_{\epsilon}(r)/2,\alpha(r)\big\}\Big)^{1-\iota}<1.
\end{equation}

\item if \ref{H2}  is  met for the discrete metric $\kappa=\varphi_0$, for all $r>r_{\epsilon}\vee r_0$  we have  the strict contraction estimate 
\begin{equation}\label{re-3}
 \beta_{\varphi_{\rho_{\epsilon}(r)}}(P)\leq 
1-\delta_{\epsilon}(r)~\frac{\alpha(r)}{2}<1.
\end{equation}
Moreover,  if $V\geq 1/2$ then, for all $V$-semi-distance  $\phi$  such that  $\|\phi/\varphi_V\|\leq \rho_\epsilon(r)\big(1-\beta_{\varphi_{\rho_{\epsilon}(r)}}(P)\big)$, we have  the strict contraction estimate
\begin{align}\label{re-3-cor}
 \beta_{\phi_{\rho_{\epsilon}(r)}}(P)\leq 1-\left(1-\beta_{\varphi_{\rho_{\epsilon}(r)}}(P)\right)^2<1.
\end{align}
\end{itemize}
\end{theo}

\begin{proof}
See Appendix~\ref{p-th1-intro}.
\end{proof}

The assumption that \ref{H1} holds with $c=1/2$ entails no loss of
generality: for a general constant $c>0$, Theorem~\ref{th1-intro}
can be applied to the function $\bar V=\frac12+\frac{\epsilon}{2c}V$
of Lemma~\ref{lemma:barV}, and the corresponding constants are given
in Corollary~\ref{cor:strict} in Appendix~\ref{sec:quantitative}. In
particular, Theorem~\ref{th-intro-basic} follows from
Theorem~\ref{th1-intro} and Corollary~\ref{cor:strict}. We also note
that, for $c=1/2$, the threshold $r_\epsilon=1/(1-\epsilon)$
coincides with the value $2c/(1-\epsilon)$ discussed in the remark
following conditions \ref{H1} and \ref{H2}.

\subsection{Two-operator contraction\label{sub:two_op}}

In some applications developed later (Section~\ref{sub:GIbbs}), $P$
factorizes as $P=KL$ for some Markov operators $K$ and $L$, where $K$
maps functions on a second state space to functions on $S$ and $L$
acts in the reverse direction. In this context, we assume that there
exists an l.s.c.~function $W$ on the second state space, bounded
below by a positive constant, such that
\begin{equation}\label{Lyap-KL}
K(W)\leq \epsilon_0~ V+ c_0\quad \mbox{\rm and}\quad
L(V)\leq \epsilon_0~ W+c_0
\end{equation}
for some constants $c_0\geq 0$ and $\epsilon_0\in ]0,1[$. Observe that $K$ is a $(V,W)$-operator and $L$ is a $(W,V)$-operator under \eqref{Lyap-KL}. Therefore, under this condition,  $P$ is indeed a Markov $V$-operator and \ref{H1} holds with $\epsilon=\epsilon_0^2$ and $c=c_0(1+\epsilon_0)$. 

In addition, we assume that there exist a constant $r_0>0$ and a map $\alpha:r\in [r_0,\infty[\mapsto \alpha(r)\in ]0,1[$ such that
\begin{equation}\label{loc-contracKL}
\varpi_{V}(x_1,x_2)\leq r \implies \Vert \delta_{x_1} K-\delta_{x_2} K\Vert_{\rm tv}\leq 1-\alpha(r),\quad\forall r\geq r_0
\end{equation}
in which case, since $\Vert \delta_{x_1} KL-\delta_{x_2} KL\Vert_{\rm tv}\leq \Vert \delta_{x_1} K-\delta_{x_2} K\Vert_{\rm tv}$, condition \ref{H2} holds for $P=KL$ with the discrete metric $\kappa=\varphi_0$. Hence, under the two conditions \eqref{Lyap-KL} and \eqref{loc-contracKL}, the results stated in part (ii) of Theorem \ref{th-intro-basic} hold.

\section{Selected results\label{sec:select_results}} 

In this section we present some important implications of Theorem \ref{th-intro-basic}.  We remind the reader that the regularity conditions \ref{H1} and \ref{H2}, assumed to hold in all the results presented below, are stated in Section \ref{reg-cond-sec-intro}.
  
\subsection{Existence of an invariant distribution}

A first merit of the strict contraction results given in the previous section  (Theorem \ref{th-intro-basic}) is that they easily show that the $V$-operator $P$ has a unique invariant distribution $\pi=\pi P\in\mathcal{P}_V(S)$.  Before stating the next result we recall that $(S,\psi_S)$ is assumed to be a complete separable metric space.

\begin{prop}\label{cor:invariant}
Assume that \ref{H1} holds and that \ref{H2} is met either for the $V$-semi-distance   $\kappa=\varphi_0$ or for some $V$-semi-distance  $\kappa$ such that $\|\psi_S^p/\kappa\|<\infty$ for some $p\geq 1$. Then,  $P$ admits a unique invariant distribution $\pi\in\mathcal{P}_V(S)$.
\end{prop}

\begin{proof}
The fact that the result stated in  \eqref{re-3.v2}   implies that $P$ has a single invariant distribution is known  (see~Theorem 8.2.21 in \cite{dm-penev-2017}), and thus below we only prove the second part of the proposition. To simplify notation, we let   $\psi=\psi_S$ in what follows.

Let $\mathcal{P}_{\psi^p}(S)$ denote the subset of probability measures $\mu\in \mathcal{P}(S)$ such that $\mu(\psi^p(\cdot,y))<\infty$ for some and hence all given $y\in S$. By Lemma~\ref{lemma:barV}, we can without loss of generality assume that $V\geq 1/2$. Then, by the first assertion of Theorem~\ref{th-intro-basic}, there exist constants  $\iota\in [1/2,1[$ and $\rho>0$ such that $\lambda:=\beta_{\kappa_{\iota,\rho}}(P)<1$, where the semi-distance $\kappa_{\iota,\rho}$ is defined in \eqref{vero}-\eqref{def-phi-rho}. By \eqref{eq:prodPK}, this implies that
\begin{align}\label{eq:inv-contr}
D_{\kappa_{\iota,\rho}}\big(\mu_1  P_n,\mu_2 P_n\big)\leq \lambda^n D_{\kappa_{\iota,\rho}}(\mu_1,\mu_2),\quad\forall n\geq 0,\quad\forall \mu_1,\mu_2\in\mathcal{P}_{V}(S).
\end{align}

Since $V\geq 1/2$, the l.s.c.~functions $\varpi_\rho$ (defined in \eqref{vero}) and $\varpi_V$ (defined in \eqref{refvarpiV}) are related through the following formula:
\begin{equation}\label{VVr}
\begin{array}{l}
\rho\,\varpi_V\leq\varpi_{\rho}
\leq (1+\rho)\,\varpi_V,\quad\forall \rho\geq 0
\end{array}
\end{equation}
and thus, under the assumptions imposed on the semi-distance function $\kappa$,  we have
\begin{align*}
\ \psi^p\leq  \Big\|\frac{\psi^p}{\kappa}\Big\|~\kappa\leq  \Big\|\frac{\psi^p}{\kappa}\Big\|\,\Big\|\frac{\kappa}{\varpi_V}\Big\|~\varpi_V\leq \frac{1}{\rho}\Big\|\frac{\psi^p}{\kappa}\Big\|\,\Big\|\frac{\kappa}{\varpi_V}\Big\|\varpi_{\rho}
\end{align*} 
where the last inequality holds by \eqref{VVr}. Writing $\psi^p=(\psi^p)^{\iota}(\psi^p)^{1-\iota}$, this shows that there exists a constant $c>0$ such that $c~\psi^p\leq  \kappa_{\iota,\rho}$. In particular, $\psi^p\leq c'\,\varpi_V$ for some constant $c'<\infty$, so that $\mathcal{P}_V(S)\subseteq\mathcal{P}_{\psi^p}(S)$, and by \eqref{eq:inv-contr} we have, for all $\mu\in\mathcal{P}_V(S)$,
\begin{align}\label{eq:cauchy}
\WW_{\psi,p}\left(\mu  P_n,\mu  P_{n-1}\right)^p=D_{\psi^p}\left(\mu  P_n,\mu  P_{n-1}\right)\leq c^{-1} \lambda^{n-1} D_{\kappa_{\iota,\rho}} (\mu P,\mu ),\quad\forall n\geq 1.
\end{align}

\emph{Existence.} Let $\mu\in\mathcal{P}_V(S)$. By \ref{H1}, we have $$m_\mu:=\sup_{n\geq 0}\mu P_n(V)\leq \mu(V)+c_1/(1-\epsilon)<\infty,$$ where $c_1$ denotes the constant $c$ appearing in \ref{H1}. By \eqref{eq:cauchy}, the sequence $(\mu P_n)_{n\geq 1}$ is Cauchy in $(\mathcal{P}_{\psi^p}(S), \WW_{\psi,p})$ and thus, since this space  is complete (see e.g.~Theorem 20.1.8 in \cite{douc}), it converges in $\WW_{\psi,p}$, and therefore weakly, to some  $\pi\in \mathcal{P}_{\psi^p}(S)$. Since $V$ is l.s.c.~and bounded below, we have $$\pi(V)\leq\liminf_{n\rightarrow\infty}\mu P_n(V)\leq m_\mu,$$ and thus $\pi\in\mathcal{P}_V(S)$. Next, for all integers $m>n\geq 1$, using \eqref{eq:inv-contr} and the fact that $\mu P_m=(\mu P_{m-n})P_n$, we have
$$
D_{\kappa_{\iota,\rho}}(\mu P_n,\mu P_m)\leq \lambda^n\,D_{\kappa_{\iota,\rho}}(\mu,\mu P_{m-n})\leq \lambda^n\,\|\kappa_{\iota,\rho}/\varpi_V\|\,\big(\mu(V)+m_\mu\big).
$$
Since $\kappa_{\iota,\rho}$ is l.s.c.~and non-negative, the map $\nu\mapsto D_{\kappa_{\iota,\rho}}(\mu P_n,\nu)$ is l.s.c.~with respect to weak convergence (this follows from the tightness of the sets of couplings with weakly converging marginals and from the lower semi-continuity of $\pi\mapsto\pi(\kappa_{\iota,\rho})$; see Lemmas 4.3 and 4.4 in~\cite{Villani}). Letting $m\rightarrow\infty$, we obtain that $$D_{\kappa_{\iota,\rho}}(\mu P_n,\pi)\leq \lambda^n\|\kappa_{\iota,\rho}/\varpi_V\|(\mu(V)+m_\mu)\rightarrow 0$$ as $n\rightarrow\infty$. Consequently,
$$
\WW_{\psi,p}\left(\mu P_{n+1},\pi P\right)^p\leq c^{-1}D_{\kappa_{\iota,\rho}}\left((\mu P_{n})P,\pi P\right)\leq c^{-1}\lambda\, D_{\kappa_{\iota,\rho}}(\mu P_n,\pi),
$$
where the right-hand side tends to zero as $n\rightarrow\infty$.
Since we also have $\WW_{\psi,p}(\mu P_{n+1},\pi)\rightarrow 0$, we conclude that $\pi P=\pi$.

\emph{Uniqueness.} Assume that $\pi'\in \mathcal{P}_{V}(S)$ is such that $\pi'P=\pi'$. Then, by \eqref{eq:inv-contr},
$$
D_{\kappa_{\iota,\rho}}\left(\pi,\pi'\right)= D_{\kappa_{\iota,\rho}}\left(\pi P,\pi'P\right)\leq \lambda\, D_{\kappa_{\iota,\rho}}\left(\pi,\pi'\right)
$$
and thus, since $\lambda<1$ and $D_{\kappa_{\iota,\rho}}(\pi,\pi')<\infty$, we must   have $D_{\kappa_{\iota,\rho}}\left(\pi,\pi'\right)=0$. Since $\kappa_{\iota,\rho}$ is a $V$-semi-distance, it follows that $\pi=\pi'$. This ends the proof of the proposition.
\end{proof}

It is well-known that $P$ admits a unique invariant probability measure $\pi=\pi P$ when  the drift condition \ref{H1} is met and when the local contraction condition \ref{H2} holds for the discrete metric $\varphi_0$.  However, as recalled in the present paper, the weighted-norm contraction coefficients technique  developed in~\cite{ado-24,dm-penev-2017} provides a very direct and short proof of this result.  Proposition \ref{cor:invariant} shows that $P$ also admits a single invariant distribution under weaker assumptions on the   semi-distance $\kappa$ for which   condition \ref{H2} holds. The result of Proposition \ref{cor:invariant} for the case where $\kappa\neq\varphi_0$ is almost identical to part (i) of Theorem 40.4.5 in \cite{douc} (see also Proposition 2.9 in \cite{Quin}), but the operator theoretic approach used in this article has the advantage of providing a very direct and concise proof.

\subsection{Exponential decay of $\phi$-Kantorovich norms}\label{expo-decays-dob-sec}

The goal of this subsection is to establish exponential decay of $\beta_\phi(P_n)$ for a larger class of $V$-semi-distances than those for which the strict contraction estimates in Theorem \ref{th-intro-basic} hold. Using the comparison principles of
Section~\ref{comparisons-principles-sec}, and in particular
Lemma~\ref{klem}, we obtain the following result, expressed in terms
of the semi-distances $(\varphi_{V},\phi_{V},\kappa_{\iota,V})$
introduced in (\ref{refvarpiV}) and (\ref{eq:kappa_iota}).

 \begin{theo}\label{th1-intro-cor}
 Assume that \ref{H1} holds. Then,
\begin{itemize}
\item  if \ref{H2} is  met for some $V$-semi-distance   $\kappa$,  for all  $\iota\in[1/2,1[$ there exist constants  $c\geq 0$ and $\lambda\in ]0,1[$   such that
\begin{equation}\label{res-H0-psi}
\beta_{\kappa_{\iota,V}}(P_n)\leq c\,\lambda^n,\quad\forall n\geq 0.
\end{equation}
\item  if \ref{H2} is met for the $V$-semi-distance   $\kappa=\varphi_0$, 
 there exist constants
  $c\geq 0$ and $\lambda\in ]0,1[$ such that 
\begin{equation}\label{res-H0}
\beta_{\varphi_V}(P_n)\leq c\,\lambda^n,\quad\forall n\geq 0.
\end{equation}
Moreover, for all $V$-semi-distance $\phi$,  there exist constants  $c\geq 0$ and $\lambda \in ]0,1[$  such that  
\begin{equation}\label{res-H1}
\beta_{\phi_V}(P_n)\leq c \,\lambda^n,\quad\forall n\geq 0.
\end{equation}
   \end{itemize}
\end{theo}

\begin{proof}

To prove the theorem we first note that, by Lemma \ref{lemma:barV}, in the calculations  below we can without loss of generality assume  that $V\geq 1/2$ and that \ref{H1} holds with $c=1/2$.

To prove \eqref{res-H0-psi}, let $\iota\in[1/2,1]$. Then, since $V\geq 1/2$, by    using \eqref{VVr} we obtain that, for all $\rho\in]0,1[$,
\begin{align*}
\rho^{1-\iota}~\kappa_{\iota,V}=\rho^{1-\iota}~\varpi_{V}^{1-\iota}\kappa^\iota\leq \kappa_{\iota,\rho}=
\varpi_{\rho}^{1-\iota}\kappa^{\iota}\leq (1+\rho)^{1-\iota}~\varpi_{V}^{1-\iota}\kappa^\iota= (1+\rho)^{1-\iota}~\kappa_{\iota,V},
\end{align*}
and thus, by using the result \eqref{compCost} presented in Lemma \ref{klem} together with \eqref{eq:prodPK}, 
\begin{align*}
\beta_{\kappa_{\iota,V}}(P_{n})\leq (1+1/\rho)^{1-\iota}~
\beta_{\kappa_{\iota,\rho}}(P)^n,\quad\forall n\geq 1,\quad\forall \rho\in]0,1[.
\end{align*}
Then,  \eqref{res-H0-psi} follows from result \eqref{re-upsilon.v2} stated in Theorem \ref{th-intro-basic}. Next, by using \eqref{VVr} (which holds since $V\geq 1/2$) and \eqref{compCost}, we have 
\begin{align*}
\beta_{\varphi_V}(P_{n}) \leq (1+1/\rho) ~\beta_{\varphi_{\rho}}(P)^n,\quad\forall n\geq 1,\quad\forall \rho>0
\end{align*}
and \eqref{res-H0} follows from result \eqref{re-3.v2} stated in Theorem \ref{th-intro-basic}. Finally, to show \eqref{res-H1}, we first note that we can assume without loss of generality that $\|\phi/\varphi_V\|\leq 1$. Indeed, if $a:=\|\phi/\varphi_V\|>1$ then $\tilde\phi:=\phi/a$ satisfies $$\tilde\phi+\varphi_V\leq \phi+\varphi_V\leq a(\tilde\phi+\varphi_V),$$
so that, by \eqref{compCost}, \eqref{res-H1} holds for $\phi$ as soon as it holds for $\tilde\phi$. Next, remark that for all constants $\rho>0$ and $\iota\in ]0,\rho]$ we have, using again   \eqref{VVr},
$$
 \iota ~\big(\varphi_V+\phi)\leq
\varphi_{\rho}+  \iota \phi\leq ((1+\rho)\vee\iota)\big(\varphi_V+\phi)
$$
and thus, by \eqref{compCost} and \eqref{eq:prodPK}, for all constants $\rho>0$ and $\iota\in ]0,\rho]$ we have
\begin{align*}
\beta_{\phi+\varphi_{V}}(P_{n}) \leq  \frac{(1+\rho)\vee\iota}{\iota} 
~\beta_{\iota \phi+\varphi_{\rho}}(P)^n,\quad\forall n\geq 1.
\end{align*}
Then, \eqref{res-H1} follows from result \eqref{re-3-cor.v2} stated in Theorem \ref{th-intro-basic}, applied with $\delta=\iota\in]0,\rho\wedge\bar\delta]$. This ends the proof of the theorem.
\end{proof}

The second part of Theorem \ref{th1-intro-cor} establishes the exponential contraction estimate   $$\|\mu_1 P_n-\mu_2 P_n\|_V\leq c~ \lambda^n~\|\mu_1-\mu_2\|_V$$ under the assumption that  \ref{H1} is met and that \ref{H2} holds for the discrete metric $\varphi_0$. This result is well known, and the present paper shows that the weighted-norm contraction coefficient technique developed in~\cite{ado-24,dm-penev-2017} provides a particularly direct and concise proof. To the best of our knowledge, the other assertions of  Theorem \ref{th1-intro-cor} are new.

\begin{rmk}\label{rmk:V-unif}
Assume that \ref{H1} holds and that \ref{H2} holds with
$\kappa=\varphi_0$, and let $\pi\in\Pa_V(S)$ be the invariant
measure of $P$ given by Proposition~\ref{cor:invariant}. Then the
estimate \eqref{res-H0} is equivalent to the $V$-uniform ergodicity
of $P$ in the sense of~\cite[Chapter 16]{meyn-tweedie-2}, that is
$$
\|\delta_xP_n-\pi\|_V\leq c'\,\lambda^n\,V(x),
\quad\forall x\in S,~\forall n\geq 0,
$$
for some constant $c'<\infty$. Indeed, applying \eqref{res-H0} with
$(\mu_1,\mu_2)=(\delta_x,\pi)$, and using
$\pi P_n=\pi$ and $\inf_SV>0$, we obtain the above
estimate with $c'=c\,(1+\pi(V)/\inf_SV)$. Conversely, for all
$\mu_1,\mu_2\in\Pa_V(S)$ and any function $f$ such that
$|f|\leq V$, since $\mu_1-\mu_2$ has zero total mass, we have
$$(\mu_1P_n-\mu_2P_n)(f)=(\mu_1-\mu_2)(g_n)
$$ 
with the function
$$
g_n:=P_n(f)-\pi(f)\Longrightarrow
|g_n(x)|\leq\|\delta_xP_n-\pi\|_V\leq c'\lambda^nV(x)
\Longrightarrow\beta_{\varphi_V}(P_n)\leq c'\lambda^n.
$$ $V$-uniform ergodicity is
established in~\cite{goldys} for general Markov processes on Polish
spaces, and applied to the stochastic Burgers and 2D Navier--Stokes
equations. The application of the present framework to such infinite
dimensional models, for which the sub-level sets of $V$ are not
compact, is left for future work.
\end{rmk}

\subsection{Exponential decay of    $(\varphi_V,\phi)$-Kantorovich norms}\label{expo-decays-lip-sec}
 
By using  the contraction estimates provided in the previous theorem, as well as the trivial observation that, for all $V$-semi-distances $\varphi$ and $\phi$, we have
\begin{align*}
\varphi\leq\phi\implies D_{\varphi}(\mu_1,\mu_2)\leq D_{\phi}(\mu_1,\mu_2), 
\end{align*}
we  can prove an exponential decay of the  $(\varphi_V,\phi)$-Kantorovich norm $\beta_{\varphi_V,\phi}(P_n)$ for a large  class of $V$-semi-distances.

\begin{theo}\label{thm:expo}

Assume that   \ref{H1} is met.
\begin{enumerate}[label=\textnormal{(\roman*)}]
\item If \ref{H2} holds for some $V$-semi-distance  $\kappa$ then, for   any $\iota\in [1/2,1]$, there exist constants $c>0$ and $\lambda\in ]0,1[$ such that
\begin{align*}
\beta_{\varphi_V,\kappa_{\iota,V}}(P_n)\leq c~  \lambda^n,\quad\forall n\geq 0,
\end{align*}
with the $V$-semi-distance $\kappa_{\iota,V}=\kappa^\iota\varphi_V^{1-\iota}$.
\item If \ref{H2} holds with $\kappa=\varphi_0$ then, for all $V$-semi-distance $\phi$, there exist constants $c>0$ and $\lambda\in ]0,1[$ such that
\begin{align*}
\beta_{\varphi_V,\phi}(P_n)\leq c~  \lambda^n,\quad\forall n\geq 0.
\end{align*}
\end{enumerate}
\end{theo}
\begin{proof}
 Since for all   $V$-semi-distance   $\kappa$ we have $$ \beta_{\varphi_V,\kappa}(P_n)\leq \|\kappa/\varpi_V\|~\beta_{\kappa}(P_n) $$ for all $n\geq 1$, the result of part (i) when $\iota<1$ follows from the first part of Theorem \ref{th1-intro-cor}.
 
To complete the proof of part (i) it remains to consider the case $\iota=1$. 
To this aim, let    $\kappa$ be a $V$-semi-distance such that \ref{H2} holds, and let   $c_\kappa=\|\kappa/\varpi_V\|<\infty$  so that
\begin{align}\label{eq:varp}
{\kappa(x,y)}/{c_\kappa}\leq \varpi_V(x,y).
\end{align}
To proceed further let $\tilde{\kappa}:= (1/c_\kappa)\kappa$.  We remark that $\tilde{\kappa}$ is a $V$-semi-distance that also satisfies \ref{H2}, and we let   $\kappa'  :=(\varpi_{V}  \tilde{\kappa})^{1/2}$. Then, by using the first result stated in Theorem \ref{th1-intro-cor}, it follows that there exist    constants $c_0>0$ and $\lambda\in ]0,1[$ such that $$\beta_{\kappa'}(P_{n} )\leq c_0\lambda^n\quad\mbox{\rm for all $n\geq 0$.}$$ On the other hand, using \eqref{eq:varp} we have $\tilde{\kappa}\leq \kappa'\leq \varphi_{V}$, where the last inequality uses the fact that $\kappa(x,x)=0$ for all $x\in S$. Consequently, for all   $n\geq 0$ and $\mu_1,\mu_2\in\mathcal{P}_V(S)$ we have 
\begin{align*}
D_{\tilde{\kappa}}(\mu_1 P_n,\mu_2 P_n)\leq D_{\kappa'}(\mu_1 P_n,\mu_2 P_n)&\leq \beta_{\kappa'}(P_n) D_{\kappa'}(\mu_1,\mu_2)\leq c_0~\lambda^n~D_{\varphi_{V}}(\mu_1,\mu_2) 
\end{align*}
and thus, recalling that $\tilde{\kappa}= (1/c_\kappa)\kappa$, for all $n\geq 0$ and $l\mu_1,\mu_2\in\mathcal{P}_V(S)$ we check that
\begin{align}\label{eq:in_varphi}
D_{\kappa}\big(\mu_1 P_n,\mu_2 P_n\big)\leq  c_0~c_\kappa~ \lambda^n~ D_{\varphi_{V}}(\mu_1,\mu_2).
\end{align}
This ends the proof of part (i).

To prove part (ii), let $\phi$ be a $V$-semi-distance. Since $\phi(x,x)=0$ for all $x\in S$, we have $\phi\leq \|\phi/\varpi_V\|\,\varphi_V$ and thus $$\phi\leq \phi_V=\phi+\varphi_V\leq (1+\|\phi/\varpi_V\|)\,\varphi_V.$$ Therefore, for all $n\geq 0$ and $\mu_1,\mu_2\in\mathcal{P}_V(S)$,
\begin{eqnarray*}
D_{\phi}(\mu_1P_n,\mu_2P_n)&\leq& D_{\phi_V}(\mu_1P_n,\mu_2P_n)\leq \beta_{\phi_V}(P_n)\,D_{\phi_V}(\mu_1,\mu_2)\\
&\leq& (1+\|\phi/\varpi_V\|)\,\beta_{\phi_V}(P_n)\,\|\mu_1-\mu_2\|_V,
\end{eqnarray*}
and part (ii) follows from \eqref{res-H1} in Theorem~\ref{th1-intro-cor}.
\end{proof}

It is important to remark that, unlike  in the  first assertion of Theorem \ref{th1-intro-cor}, in part (i) of Theorem \ref{thm:expo} the case $\iota=1$ is allowed. In this scenario, part (i) of Theorem \ref{thm:expo} states that
\begin{align}\label{eq:bv}
\text{$\exists$ constants $c>0$ and $\lambda\in ]0,1[$ such that $\beta_{\varphi_V,\kappa}(P_n)\leq c \lambda^n$ for all $n\geq 0$}.
\end{align}

We also note that, under assumptions  which are essentially equivalent to assuming  that  \ref{H1} holds and that \ref{H2} is met for some semi-distance $\kappa$,  an exponential convergence result  of the form $$D_{\kappa}(\mu_1 P_n, \mu_2 P_n)\leq c \,\lambda^n\big(\mu_1(V)+\mu_2(V)\big)$$  is obtained for instance in \cite{Quin} (see in particular Corollary 2.1 and Theorem 2.6) and in Theorem 40.4.5 of the book \cite{douc}.  Theorem \ref{thm:expo}    improves upon this result  by establishing  an  exponential contraction estimate of the form  $$D_{\kappa}(\mu_1 P_n, \mu_2 P_n)\leq c\,\lambda^n\|\mu_1-\mu_2\|_V.$$

\subsection{Exponential contraction in Wasserstein distances\label{sub:Wasser}}

In the terminology of Joulin and Ollivier~\cite{joulin},  we assume below that the following local positive Ricci curvature condition holds:  There exist    constants $p\geq 1$ and  $r_0>0$, and a map $\alpha:[r_0,\infty[\,\mapsto\, \alpha(r)\in ]0,1[$, such that
\begin{align}\label{eq:ricci}
\varpi_V(x,y)\leq r\implies \mbox{\rm curv}_{V}(x,y):=
1-\frac{\WW_{\psi_S,p}(\delta_x P,\delta_yP)^p}{\psi^p_S(x,y)}\geq \alpha(r),\quad\forall r\geq r_0,\quad\forall x\neq y.
\end{align}

To proceed further,  we focus on the situation where $(S,\|\cdot\|)$ is a finite dimensional normed  vector space,  and we let  $\psi_S(x,y)=\|x-y\|$. In this context,  we first assume that \ref{H1} holds with the  polynomial-type Lyapunov function  $V(x)=\frac{1}{2}+\Vert x\Vert^p$ for all $x\in S$. In this case, as proved in Appendix \ref{cond-intro-exp-proof}, we have the estimate
\begin{equation}\label{cond-intro-p-norms}
 \varpi_V(x,y)\geq \gamma_p~  \psi_S(x,y)^p
 \quad \mbox{\rm with}\quad \gamma_p:=2^{-(p-1)_+} 
\end{equation}
showing that the function $\psi_S^p$ is a $V$-semi-distance. 

In this scenario, assuming that the local positive Ricci curvature condition \eqref{eq:ricci} holds is equivalent to assuming that condition \ref{H2} is met with $\kappa=\psi_S^p$. Therefore, by Theorem~\ref{th-intro-basic},  for some $\rho>0$ and any $\iota\in[1/2,1[$, we have $\beta_{\kappa_{\iota,\rho}}(P)<1$, 
with the  $V$-semi-distance $\kappa_{\iota,\rho}$ defined in \eqref{def-phi-rho}. In words,   local positive curvatures yield, through Theorem~\ref{th-intro-basic}, a global contraction estimate for the Lyapunov-adapted Kantorovich semi-distance $\kappa_{\iota,\rho}$. In addition, by part (i) of Theorem \ref{thm:expo},  there exist  constants $c'_p>0$ and $\lambda_p\in ]0,1[$ such that, for all  $\mu_1,\mu_2\in\mathcal{P}_V(S)$,
\begin{align*}
\WW_{\psi_S,p}\left(\mu_1P_n,\mu_2P_n\right)^p=D_{\psi_S^p}\left(\mu_1P_n,\mu_2P_n\right) \leq c'_p~\lambda_p^n~\|\mu_1-\mu_2\|_V,\quad\forall n\geq 0.
\end{align*}

The  local curvature assumption  \`a la Joulin-Ollivier therefore fits naturally into our framework, and our results  provide a mechanism for turning this  local   contraction assumption into  an exponential contraction result for the $p$-Wasserstein distance. We however point out  that  a  global positive Ricci   curvature assumption with $p=1$   leads directly to one-step $1$-Wasserstein contraction estimates, as shown in~\cite{joulin}. Such $1$-Wasserstein contraction properties play an important role in Markov chain Monte Carlo, where they provide a fundamental ingredient for deriving quantitative estimates and concentration inequalities for empirical averages, as well as perturbation bounds~\cite{joulin,rudolf-18}.  Similarly, a recent work by Monmarch\'e~\cite{monmarche-AHL} has shown that quantitative $2$-Wasserstein contraction estimates can be used, under suitable assumptions, to establish functional inequalities such as Poincar\'e inequalities for invariant measures. Our framework complements these geometric approaches by showing that Wasserstein stability properties, in the slightly weaker form of exponential decay estimates with a multiplicative constant and with the weighted total variation norm $\|\mu_1-\mu_2\|_V$ on the right-hand side, can be obtained under substantially weaker assumptions, namely local contraction combined with a Lyapunov drift condition.
Specifically,  while curvature-based approaches exploit a globally contractive geometry, our framework provides a mechanism for creating global Wasserstein stability from local contraction and recurrence properties, in settings where no natural globally contractive metric is available.

To conclude this subsection, we assume now that \ref{H1}  holds with $V$ an exponential-type Lyapunov function, that is with $V(x)=\exp{(\delta\Vert x\Vert)}$ for all $x\in S$ and some $\delta>0$. In this case, we show in Appendix \ref{cond-intro-exp-proof} that
\begin{equation}\label{cond-intro-exp}
\begin{array}{l} 
 \displaystyle\varpi_V(x,y)\geq \varpi_e(x,y)\geq
c_p\,\psi_S(x,y)^p,\quad  \mbox{\rm with}\quad c_p:= \frac{\delta^{p}}{2^{p-1}p!},\quad\forall p\in\NN-\{0\}
 \end{array}\end{equation}
and where the function $\varpi_e$ is defined by
\begin{align*}
\varpi_e(x,y)=2\left(\exp{\left(\frac{\delta \|x-y\|}{2}\right)}-1\right),\quad x,y\in S.
\end{align*}
By \eqref{cond-intro-exp}, the    function  $\varpi_e$ is a $V$-semi-distance, and it follows from Theorem \ref{thm:expo} (part (i) with $\iota=1$, or part (ii) with $\phi=\varpi_e$) that if \ref{H2} holds either for some $V$-semi-distance $\kappa\geq \varpi_e$ or with $\kappa=\varphi_0$ then    there exist  constants $c'>0$ and $\lambda\in ]0,1[$ such that, for all $\mu_1,\mu_2\in\mathcal{P}_V(S)$, $n\geq 0$, $p\in\NN-\{0\}$, 
\begin{align*}
\WW_{\psi_S,p}\left(\mu_1P_n,\mu_2P_n\right)^p \leq c_p^{-1} D_{\varpi_e}\left(\mu_1P_n,\mu_2P_n\right)  \leq c_p^{-1}\, c' ~\lambda^n\|\mu_1-\mu_2\|_V.
\end{align*}
Since $p\mapsto\WW_{\psi_S,p}$ is non-decreasing, this also yields exponential contraction estimates for $\WW_{\psi_S,p}$ for all real $p\geq 1$.

\subsection{Continuous time Markov semigroups}\label{sub:continuous}

In this subsection we show how the results presented above can be used to obtain exponential contraction estimates for a continuous time Markov semigroup  $(Q_t)_{t\in[0,\infty[}$, which satisfies the semigroup property $Q_{s+t}=Q_sQ_t$ for all $s,t\geq 0$.

Let $\Ga$ denote  the generator of the semigroup,  defined on some domain $\Da(\Ga)$, and assume that there exist an l.s.c.~function $V\in \Da(\Ga)$ with $V\geq 1/2$ and some parameters $a_0, a_1>0$  such that 
\begin{equation}\label{V-L-ex}
 \Ga(V)\leq -a_0~ V+a_1.
\end{equation}
In this context,  for all $h>0$ we have the following rather well known estimate (see for instance~\cite{ado-24})
\begin{equation}\label{V-L-ex-s}
 Q_h(V)\leq \epsilon_h V+c_h\quad\mbox{\rm with}\quad
\epsilon_{h}=(1+a_0h)^{-1}<1\quad \mbox{\rm and}\quad c_{h}=a_1 h.
\end{equation}
 
We now let $h\in]0,1]$ be fixed and $P=Q_h$. Then,   by \eqref{V-L-ex-s},    condition \ref{H1} is met for  the function   $V$ such that \eqref{V-L-ex} holds. We now assume that, for $P=Q_h$, condition \ref{H2} holds for some semi-distance $\kappa$.  Noting that  $(P_n)_{n\geq 0}=(Q_{nh})_{n\geq 0}$, that is that $(Q_{nh})_{n\geq 0}$ is the discrete-time semigroup of Markov $V$-operators associated with the Markov $V$-operator $P=Q_h$, it follows from Theorem \ref{thm:expo} that there exist constants $c_h>0$ and $\lambda_h\in]0,1[$ such that
\begin{align}\label{eq:V_diff}
\beta_{\varphi_V,\kappa}(Q_{nh})\leq c_h \lambda_h^n,\quad\forall n\geq 0.
\end{align}

To extend the contraction estimate \eqref{eq:V_diff} obtained for the discrete time Markov semigroup $(Q_{nh})_{n\geq 1}$ to the continuous time Markov semigroup $(Q_t)_{t\in[0,\infty[}$ remark first  that 
\begin{align}\label{eq:NT}
Q_t=Q_{h\{t/h\}}Q_{h \lfloor t/h\rfloor}=Q_{h\{t/h\}}P_{\lfloor t/h\rfloor},\quad\forall t>0
\end{align}
 where $\lfloor s\rfloor$ and $\{s\}:=s-\lfloor s\rfloor<1$ stand for the integer part and the fractional part of $s>0$.   In addition,  remark that since for all $t>0$ we have $\{ t/h\}h< h\leq 1$, it follows from \eqref{V-L-ex-s} that
$$
\Vert Q_{\{ t/h\}h}(V)/V\Vert\leq 
\sup_{0\leq s\leq 1}\Vert Q_s(V)/V\Vert<\infty,\quad\forall t>0
$$
implying that there exists a constant $\bar c>0$ for which we have
 \begin{align}\label{eq:VQ}
 \|\mu_1 Q_{h\{t/h\}}-\mu_2 Q_{h\{t/h\}}\|_V\leq \bar c~\|\mu_1-\mu_2\|_V\quad\forall t>0,\quad \forall\mu_1,\mu_2\in\mathcal{P}_V(S).
 \end{align}
By combining \eqref{eq:V_diff}-\eqref{eq:VQ}, we obtain the following  contraction estimate  for the continuous time Markov semigroup $(Q_t)_{t\in[0,\infty[}$ 
 \begin{align}\label{eq:T}
 \beta_{\varphi_V,\kappa}(Q_t)\leq \bar c\,c_h~ \lambda_h^{\lfloor t/h\rfloor},\quad\forall t>0.
 \end{align}
 Since $\lambda_h^{\lfloor t/h\rfloor}\leq \lambda_h^{t/h-1}$, we obtain, for all $t>0$, the exponential contraction estimate 
 \begin{align}\label{eq:TT}
  \beta_{\varphi_V,\kappa}(Q_t)\leq  \bar c_{h}~e^{-\varsigma_h t}\quad \mbox{\rm with}\quad \bar c_{h}:=\bar c\,c_h~ \lambda_h^{-1}\quad\mbox{\rm and}\quad \varsigma_h:=-\frac{1}{h}\log{\lambda_h}>0.
 \end{align}
Moreover, if for $P=Q_h$ condition \ref{H2} holds with $\kappa=\varphi_0$ then, by using part (ii) of Theorem \ref{thm:expo} as well as \eqref{eq:NT}-\eqref{eq:VQ}, we readily obtain that   for all $V$-semi-distance $\phi$   there exist   constants   $\bar c'_{h}\geq 0$ and $\varsigma'_h>0$ such that
 \begin{align}\label{eq:T2}
 \beta_{\varphi_V,\phi}(Q_t)\leq \bar c'_{h}~e^{-\varsigma'_h t},\quad\forall t>0. 
 \end{align}

We note that a contraction  result for a continuous time Markov semigroup   $(Q_t)_{t\in]0,\infty[}$ is obtained in  \cite[Theorem 4.8]{hairer-mattingly-scheutzow}. Assuming the existence of a Lyapunov function $V$, this theorem
shows, for a symmetric semi-distance $\kappa$ bounded by one, that
$Q_t$ is a strict contraction, for $t$ sufficiently large, for the
Wasserstein-type distance associated with the cost
$\kappa^{1/2}\varphi_1^{1/2}$; in the notation of the present
article, this can be written as $\beta_{\kappa_{\iota,\rho}}(Q_t)<1$
with $\iota=1/2$ and $\rho=1$. However, this result further assumes
a global contraction condition for the cost $\kappa$ at some time
$t'>0$ which, by Lemma \ref{lemma:ced}, can be written as
$\beta_{\kappa}(Q_{t'})<1$.

 In the context of $\RR^n$-valued diffusions, an exponential contraction estimate is obtained in \cite{eberle}, assuming   the existence of a Lyapunov function $V$ satisfying the geometric drift condition (\ref{V-L-ex}) and some additional key growth conditions (see Assumption 2.3 in~\cite{eberle}). Specifically, Theorem 2.1 in \cite{eberle} shows that
$$D_{\phi}(\mu_1Q_t,\mu_2Q_t)\leq \lambda^t D_{\phi}(\mu_1,\mu_2)$$ for
some constant $\lambda\in]0,1[$, that is
$\beta_{\phi}(Q_t)\leq \lambda^t$ in the notation of the present
article, where $$\phi(x,y)=f(\|x-y\|)+\varphi_\rho(x,y)$$ for some judiciously chosen bounded, non-decreasing, concave and continuous function $f$ and some constant $\rho>0$.  The proof of this result combines diffusion reflection coupling techniques with a general It\^o--Tanaka type formula and martingale and local time techniques to estimate the time spent by the processes on different level sets. 
  By contrast,   for the same diffusion model as the one considered in \cite{eberle}, and without the additional growth conditions of \cite{eberle} (but assuming that the transitions $Q_h$ admit a continuous and positive density), we show in Section \ref{sub:diffusion} that the above calculations directly yield exponential contraction estimates. We however stress that our   estimates do  not establish that each Markov operator   $Q_t$ is a   contraction, unlike the result obtained  in  \cite{eberle} (again, see Section \ref{sub:diffusion} for details).

\section{Applications\label{sec:applications}}

In this section we   illustrate  the scope of the proposed framework through six representative families of examples, arranged roughly in order of increasing complexity.
 These examples demonstrate how the same Lyapunov--Kantorovich methodology yields quantitative contraction estimates across a diverse collection of Markov models encompassing both classical and contemporary applications. Specifically, we consider (i) Markov transitions on state spaces
with boundary behavior, including bounded domains with entrance
boundaries (Sections~\ref{sec-illustrations-interval}
and~\ref{bd-sec}); (ii) products of a Markov kernel and its
adjoint, including two-block Gibbs samplers
(Section~\ref{sub:GIbbs}); (iii) iterated random functions
(Section~\ref{sub:IRF}); (iv) diffusion processes, including
overdamped Langevin diffusions with potentials that are convex at
infinity (Section~\ref{sub:diffusion}); (v) continuous time
processes on domains with a boundary, including jump-diffusions on
compact manifolds, Jacobi diffusions on the open unit interval and
Riccati matrix diffusions (Section~\ref{sub:boundary-processes});
and (vi) time-varying semigroups, including nonlinear Feynman--Kac
evolutions and Sinkhorn bridges (Section~\ref{time-varying}).

\subsection{Some illustrative  examples}\label{sec-illustrations-interval}

\subsubsection{Markov chains on the unit open interval with bounded densities} 
We let $S=]0,1[$, $\psi_e(x,y)=|x-y|$ be the Euclidean metric and, for all $\iota\in ]0,1]$ we let
\begin{equation}\label{Ref-upsi}
\psi_\iota(x,y):=\left\vert
\frac{1}{x}-\frac{1}{y}\right\vert^{\iota}+ \left\vert
\frac{1}{1-x}-\frac{1}{1-y}\right\vert^{\iota},
\end{equation}
and we define the continuous function
\begin{align}\label{eq:W_def}
W_\iota(x):=\frac{1}{x^{\iota}}+\frac{1}{(1-x)^{\iota}}.
\end{align}
 We note that  $(S,\psi_\iota)$ is a complete and separable metric space for all $\iota\in]0,1]$ while the space $(S,\psi_e)$ is separable but not complete.  

 We now let    $\nu(dy)=g(y)dy\in\mathcal{P}(S)$ and $\mu(dy)=h(y)dy\in\mathcal{P}(S)$ be  probability measures with bounded densities, i.e.~such that $\|g\|+\|h\|<\infty$, and we assume first that
\begin{equation}\label{ex-min-x}
\begin{array}{l}
P(x,dy)=x\,\mu(dy)+(1-x) \nu(dy).
\end{array}
\end{equation}
Let $\iota\in]0,1[$. Noting that $W_\iota\in \Ba_{\infty}(S)$
while $$\Vert P(W_\iota)\Vert\leq c:=2(\Vert g\Vert+\Vert h\Vert)/(1-\iota)<\infty,$$
it follows that \eqref{vanish-0} holds with $\Theta=c/W_\iota$ and
thus, by \eqref{eq:Lap}, condition \ref{H1} holds with
$V=W_\iota$. 
We further  assume that the two probability distributions $\nu$ and $\mu$ have disjoint supports. In this situation, for all $x_1, x_2\in S$ we have
$$
\Vert \delta_{x_1} P-\delta_{x_2}P\Vert_{\rm tv}=|x_1-x_2|~\Vert \mu-\nu\Vert_{\rm tv}=|x_1-x_2| 
$$
implying that $\beta_{\varphi_0}(P)=1$. However, noting that,
\begin{align*}
W_{\iota}\leq r\implies P(x,dy)\geq (1-\epsilon(r)) \nu(dy),\quad  \epsilon(r):=\sup_{W_{\iota}(x)\leq r}x\in ]0,1[
\end{align*} 
it follows that the local minorization condition \eqref{loc-min}
holds for all $r_0\geq \inf_S W_\iota$ and with $\nu_r=\nu$. Hence,
condition \ref{H2} is met with the discrete metric $\kappa=\varphi_0$,
although $P$ is not a strict contraction in total variation. By part
(ii) of Theorem \ref{thm:expo}, for all $V$-semi-distance $\phi$,
the Kantorovich semi-distance $D_\phi(\mu_1 P_n,\mu_2 P_n)$ converges
to zero exponentially fast as $n\rightarrow\infty$. In particular,
since the complete metric $\psi_S=\psi_\iota$ is a $V$-semi-distance
(by subadditivity of $t\mapsto t^\iota$, we have
$\psi_\iota\leq\varpi_{W_\iota}$), we obtain the following
exponential contraction estimate for the Wasserstein 1-distance:
 \begin{align}\label{eq:Wass_S2}
\WW_{\psi_S,1}\left(\mu_1P_n,\mu_2P_n\right) \leq c\lambda^n\|\mu_1-\mu_2\|_V,\quad\forall n\geq 0,\quad\forall\mu_1,\mu_2\in\mathcal{P}_V(S).
\end{align}

Following an example proposed in \cite[Section 2.1]{df-99}, we
now assume that
$$
P(x,dy):=
\frac{1}{2x}~1_{]0,x]}(y)\,dy+\frac{1}{2(1-x)}~1_{]x,1[}(y)\,dy 
$$
so that the Markov transition $P$ has an unbounded density. In this context, we let $\iota\in]0,1/2[$ and   show  in Appendix \ref{RFI-arcsin-proof} that  there exist some parameters $ \epsilon\in [0,1[$ and $c\geq 0$ such that
\begin{equation}\label{RFI-arcsin}
P(W_\iota)(x)\leq \epsilon\,W_\iota(x)+c 
\end{equation}
and thus such that condition \ref{H1} holds with $V=W_\iota
$. Moreover, for all $x\in S$ we have
$$P(x,dy)\geq \frac{1}{2}~1_{]0,1[}(y)~dy,$$ so that the minorization
condition \eqref{loc-min} holds on the whole state space, and thus
\ref{H2} is met with $\kappa=\varphi_0$. Therefore, since the
complete metric $\psi_S=\psi_{\iota}$ is a $V$-semi-distance, the
exponential contraction estimate \eqref{eq:Wass_S2} for the
Wasserstein 1-distance holds.

\subsubsection{A Markov chain on the half-line}

Let $S=]0,\infty[$, $\delta\in]0,1[$ and $\gamma>0$, and let $P$
be the Markov transition on $S$ given by
$$
P(x,dy):=\frac{\delta}{2}~
\frac{1}{x}~1_{]0,x]}(y)~dy+\frac{\delta}{2}~g(y-x)~dy+(1-\delta)\nu_{\gamma}(dy).
$$
In the above display, $g$ and $\nu_{\gamma}$ stand for the
half-Gaussian density and the Weibull distribution with shape
parameter $1+\gamma$, defined by
$$
g(z)=1_{]0,\infty[}(z)~\sqrt{\frac{2}{\pi}}~e^{-z^2/2}\quad \mbox{\rm and}\quad
\nu_{\gamma}(dz):=1_{]0,\infty[}(z)~(1+\gamma)~z^{\gamma}~e^{-z^{1+\gamma}}~dz.
$$
We remark that, as in the previous example, the   density of $P$ is unbounded and the space  $S$ equipped with the Euclidean metric $\psi_e(x,y)=|x-y|$ is separable but not complete. It can however be shown that $(S,\psi_S)$ is a complete and separable metric space, where the    metric $\psi_S$ is defined, for some $\iota\in ]0,1[$, by
\begin{align}\label{Ref-upsi2}
\psi_S(x,y):=\left\vert
\frac{1}{x}-\frac{1}{y}\right\vert^{\iota}+ \left\vert
x-y\right\vert.
\end{align}
We now let $W(x):=x^{-\iota}+x$ with $\iota$ as in
\eqref{Ref-upsi2}, and remark that $W\in \Ba_{\infty}(S)$. In
Appendix \ref{RFI-arcsin-proof} we show that, if
$\iota<2(1-\delta)/(2-\delta)$, there exist some parameters
$\epsilon\in ]0,1[$ and $c\geq 0$ such that
\begin{equation}\label{RFI-0-2}
P(W)\leq \epsilon\, W+c 
\end{equation}
and therefore condition \ref{H1} holds with $V=W$. In addition,
since  for all $x\in S$ we have
$$P(x,dy)\geq (1-\delta)\nu_{\gamma}(dy),$$
the minorization condition \eqref{loc-min} holds on the whole state
space, and thus \ref{H2} is met with $\kappa=\varphi_0$. Therefore, since the complete metric $\psi_S$ defined in
\eqref{Ref-upsi2} is a $V$-semi-distance, it follows from part (ii)
of Theorem \ref{thm:expo} that the exponential contraction estimate
\eqref{eq:Wass_S2} for the Wasserstein 1-distance holds.

\subsection{Markov chains on domains with a boundary}\label{bd-sec} 

In this subsection we consider  a Markov transition $P(x,dy)=p(x,y)dy$ with a uniformly bounded, strictly positive and continuous density $p(x,y)$
on an open and connected set $S\subset \RR^m$ having a non-empty boundary $\partial S$. We    define the continuous function  
$$
d_m(x,\partial S):=\inf{\left\{\Vert x-y\Vert~:~y\in \partial S\right\}}
$$
 and, for some $ \iota\in]0,1[$, we equip $S$ with the metric
\begin{equation}\label{Ref-upsi-bord}
\psi_S(x,y):=\left\vert
\frac{1}{d_m(x,\partial S)}-\frac{1}{d_m(y,\partial S)}\right\vert^{\iota}+\|x-y\|^\iota.
\end{equation}
It can be shown that $(S,\psi_S)$ is a complete and separable metric space. 
\subsubsection{Bounded domains with a locally Lipschitz boundary\label{sub:BB}}

 Assume that $S$ is bounded and has a locally Lipschitz boundary $\partial S$, and let $$W(x)=d_m(x,\partial S)^{-\iota}+\|x\|^{\iota}.$$ In this scenario, we clearly have $W\in \Ba_{\infty}(S)$ and,  since the function $W$ is
 integrable over the set $S$ (cf.~Lemma 5.2 in~\cite{ado-24}), under the above assumptions on $P$ we have $\|P(W)\|<\infty$. Consequently, \eqref{vanish-0} is met with $\Theta=\Vert P(W)\Vert/W$
and it follows from \eqref{eq:Lap} that \ref{H1} holds with $V=W$. On the other hand,  since $P$ is assumed to have a strictly positive and continuous density,  condition \ref{H2} holds with $\kappa=\varphi_0$ by Lemma \ref{lemma:minor}. By part (ii) of Theorem \ref{thm:expo}, it follows that  for all $V$-semi-distance $\phi$ the Kantorovich semi-distance $D_\phi(\mu_1 P_n,\mu_2 P_n)$ converges to zero exponentially fast as $n\rightarrow\infty$. In particular, the metric $\psi_S$ defined in \eqref{Ref-upsi-bord}
is a $V$-semi-distance (again by subadditivity of $t\mapsto t^\iota$),
and thus part (ii) of Theorem \ref{thm:expo} gives the following
exponential contraction estimate for the Wasserstein 1-distance:
 \begin{align}\label{eq:Wass_S}
\WW_{\psi_S,1}\left(\mu_1P_n,\mu_2P_n\right) \leq c \lambda^n\|\mu_1-\mu_2\|_V,\quad\forall n\geq 0,\quad\forall\mu_1,\mu_2\in\mathcal{P}_V(S).
\end{align}

\subsubsection{Domains with a locally Lipschitz bounded boundary}
Assume that $S$ is not necessarily bounded but has a locally Lipschitz bounded boundary $\partial S$. Assume also that there exists a function $W_1\in \Ba_{\infty}(S)$  such that $\sup_{x\in S} \|x\|^\iota/W_1(x)<\infty$ and such that $P(W_1)\leq \epsilon_1 W_1+c_1$  for some parameters $\epsilon_1\in [0,1[$ and $c_1> 0$. In this scenario, condition \ref{H1} holds with $V=W_1$ but, for this definition of $V$, the metric $\psi_S$ defined in \eqref{Ref-upsi-bord} may not be a $V$-semi-distance. To solve this problem we let $$W_2(x):=d_m(x,\partial S)^{-\iota}\quad\mbox{\rm and}\quad  c_2:=\Vert P(W_2)\Vert \leq 1+\Vert P(W_2 1_{W_2\geq 1})\Vert.$$  Using the fact that $$\{x\in S:W_2(x)\geq 1\}=\{x\in S :d_m(x,\partial S)\leq 1\}$$ with $\partial S$ a bounded set, and recalling that the density $p(x,y)$ is assumed to be uniformly bounded, by using a similar argument as in the proof of Lemma 5.2 in~\cite{ado-24} we can check that $\Vert P(W_2 1_{W_2\geq 1})\Vert<\infty$. This implies that $c_2<\infty$ and thus, since
$$
P(W_1+W_2)\leq \epsilon_1 W_1+c_1+c_2\leq \epsilon_1\,(W_1+W_2)+c_1+c_2,
$$
condition \ref{H1} also holds with $V=W_1+W_2\in\Ba_\infty(S)$.  For this definition of $V$ the metric $\psi_S$  defined in \eqref{Ref-upsi-bord} is a $V$-semi-distance  and, as in Section  \ref{sub:BB}, we can apply Lemma \ref{lemma:minor} to conclude that \ref{H2}
holds with $\kappa=\varphi_0$. Therefore, by part (ii) of Theorem
\ref{thm:expo}, the exponential Wasserstein contraction estimate
\eqref{eq:Wass_S} holds.

\subsection{Block-Gibbs samplers\label{sub:GIbbs}}

\subsubsection{Conjugate transformations}
We consider a locally compact Polish space $S$  endowed with some locally bounded positive measure  $\nu$, and we let $(\nu_g,\nu_h)$ be Boltzmann-Gibbs probability distributions on $S$
defined by 
\begin{equation}\label{ref-intro-UV}
\nu_g(dx):=\exp{(-g(x))}~\nu(dx)
\quad \mbox{\rm and}\quad
\nu_h(dy):=\exp{(-h(y))}~\nu(dy) 
\end{equation}
for some normalized potential functions $g:x\in S\mapsto g(x)\in \RR$ and 
$h:y\in S\mapsto h(y)\in \RR$. In addition, we let $M(y,dx)$ be a Markov transition   of the form
\begin{equation}\label{ref-intro-Kcchi}
M(y,dx)=\exp{(-m(y,x))}~\nu(dx)
\end{equation}
 for some normalized  potential transition $m:(y,x)\in S^2\mapsto m(y,x)\in \RR$.
 
In this context, we have the duality formula
\begin{equation}\label{ref-duality-def}
\pi_{h}(d(x,y)):=\nu_h(dy)~M(y,dx)=(\nu_h M)(dx)~M^{\star}_{\nu_h}(x,dy)
\end{equation}
where the conjugate (a.k.a.~adjoint or backward) transition $M^{\star}_{\nu_h}$ is given by the Markov transition 
\begin{equation}\label{def-K-gibbs}
\displaystyle M^{\star}_{\nu_h}(x,dy)=K(x,dy):=
\frac{d\nu}{d\nu_h M}(x)~\nu_h(dy)~\exp{(-m(y,x))}\Longrightarrow \nu_h MK=\nu_h.
\end{equation}
To go one step further in our discussion we let  $\pi_g$  be the probability measure on $S^2$ defined by
$$
\pi_g(d(x,y)):=\nu_g(dx)~K(x,dy)=
(\nu_g K)(dy)~K^{\star}_{\nu_g}(y,dx)
$$
where the conjugate transition $ K^{\star}_{\nu_g}$ is given by the Markov transition
\begin{equation}\label{def-L-gibbs}
 K^{\star}_{\nu_g}(y,dx)=L(y,dx):=\frac{1}{M\left(\frac{d\nu_g}{d\nu_hM}\right)(y)}~M(y,dx)~\frac{d\nu_g}{d\nu_hM}(x)\Longrightarrow
 \nu_g KL=\nu_g.
\end{equation}

We stress that a large class of forward-backward-type samplers,  including proximal samplers~\cite{cai,chewi-chen,chewi-phd,jiaming,jiaojiao,lee,mou}, Glauber dynamics~\cite{zchen,glauber}, Sinkhorn Gibbs-loop processes~\cite{adpm-25,adpm-24,dm-25} and two-block Gibbs samplers~\cite{damlen1999gibbs,gelfand,gelfand-2,vono2019split}
with a prescribed target measure,  have a two-stage 
transition $P=KL$ of the product form of a Markov transition $K$ and its adjoint $L$; see also ~\cite{anari,ychen,geman,guan,rendell2020global,vono2020asymptotically} and references therein. For instance, the two-block Gibbs sampler (or data augmentation
algorithm) targeting the joint distribution $\pi_g(d(x,y))$ alternates
the two conditional updates: starting from $x$, it first draws $y$
from the conditional distribution $K(x,dy)$ of the second component
given the first one, and then draws a new state $x'$ from the
conditional distribution $L(y,dx')$ of the first component given the
second one. The $x$-component of this sampler is therefore a Markov
chain with transition $P=KL$ and invariant distribution $\nu_g$, while
its $y$-component is a Markov chain with transition $LK$ and invariant
distribution $\nu_gK$. 

The stability analysis of these Markov chains is traditionally based on spectral gaps  or relative entropies and log-Sobolev inequalities~\cite{bakry,achen,dplm-03,diaconis,ding,jerrum}, see also the recent article~\cite{goyal} and references therein.
The stability of Gibbs samplers in particular classes of statistical conjugate models has also been studied using Lyapunov and Wasserstein approaches, respectively in the articles ~\cite{chan,hobert,rao,roman} and~\cite{hobert-15}. As underlined in~\cite{chan} the design of appropriate Lyapunov functions depends on some understanding of the dynamics of the Gibbs sampler which is often implicitly defined. The next subsection provides a simple method for designing Lyapunov functions for transitions $(K,L)$ of the form \eqref{def-K-gibbs} and \eqref{def-L-gibbs}.

\begin{examp}[Bivariate Gaussian model]\label{bi-gauss-ex}
Let $S=\RR^m$, let $\nu$ be the Lebesgue measure on $S$ and let $\mbox{\rm gauss}_{\theta,\sigma}$ denote the Gaussian density with mean $\theta$ and positive definite covariance matrix $\sigma$. For some vector $a\in\RR^m$, some $(m\times m)$-matrix $B$ and some positive definite matrices $\tau$ and $\sigma$, we set
$$
M(y,dx):=\mbox{\rm gauss}_{a+B y,\tau}(x)~dx\quad\mbox{\rm and}\quad
\nu_h(dy):=\mbox{\rm gauss}_{\theta,\sigma}(y)~dy.
$$
Equivalently, we have
$$
m(y,x)=\frac{1}{2}~(x-(a+B y))^{\top}\tau^{-1}
 (x-(a+B y))+\log{\sqrt{\mbox{\rm det}(2\pi\tau)}}.
$$
Note that
$$
(\nu_hM)(dx)=\mbox{\rm gauss}_{\theta_0,\sigma_0}(x)dx
\quad\mbox{\rm with}\quad
\theta_0:=a+B \theta
\quad\mbox{\rm and}\quad
\sigma_0:=\tau+B\sigma B^{\top},
$$
so that $M(y,dx)=\mbox{\rm gauss}_{\theta_0+B (y-\theta),\tau}(x)dx$. On the other hand, by the matrix inversion lemma we have
$$
\sigma_0^{-1}=\tau^{-1}-\tau^{-1}B~\tau_1~B^{\top}\tau^{-1}
\quad\mbox{\rm with}\quad
\tau_1^{-1}:=\sigma^{-1}+B^{\top}\tau^{-1}B,
$$
that is $\tau_1=\left(I-\sigma B^{\top}\sigma_0^{-1}B \right)\sigma$.
By the Gaussian regression formula, we have
$$
K(x,dy)=\mbox{\rm gauss}_{\theta+B_1(x- \theta_0),\tau_1}(y)~dy
\quad\mbox{\rm with}\quad
B_1:=\sigma B^{\top}~\sigma_0^{-1}.
$$
Let $\nu_g(dx):=\mbox{\rm gauss}_{\overline{\theta},\overline{\sigma}}(x)dx$ for some mean $\overline{\theta}$ and positive definite covariance matrix $\overline{\sigma}$. Then $(\nu_gK)(dy)=\mbox{\rm gauss}_{\theta_1,\sigma_1}(y)dy$, with the parameters
$$
\theta_1:=\theta+B_1(\overline{\theta}- \theta_0)
\quad\mbox{\rm and}\quad
\sigma_1:=B_1\overline{\sigma}B_1^{\top}+\tau_1.
$$
In this context, the target probability measure $\pi_g$ is the Gaussian distribution
$$
\pi_g(d(x,y))=\nu_g(dx)~K(x,dy)=
(\nu_gK)(dy)~ K^{\star}_{\nu_g}(y,dx)
$$
and, again by the regression formula, the conjugate transition is given by
$$
 K^{\star}_{\nu_g}(y,dx)=L(y,dx)=\mbox{\rm gauss}_{\overline{\theta}+B_2(y-\theta_1),\tau_2}(x)\,dx
$$
with the parameters
$$
B_2:=\overline{\sigma}B_1^{\top}\sigma_1^{-1}
\quad \mbox{\rm and}\quad
\tau_2^{-1}:=\overline{\sigma}^{-1}+B_1^{\top}\tau_1^{-1}B_1.
$$
\end{examp}

\subsubsection{Lyapunov functions}\label{lyap-sec-gibbs}

Let us assume that the functions $(g,h,m)$ are locally bounded and bounded below by real numbers
$(g^-,h^-,m^-)$
and that the two functions  $(g,h)$ have compact sub-level sets.

Consider the functions
\begin{equation}\label{UV-delta}
g_{\delta}:=\delta g-m_h\quad\mbox{\rm and}\quad
h_{\delta}:=\delta h-m_g
\end{equation}
indexed by $\delta\in ]0,1[$ with the integrated functions $(m_g,m_h)$ defined by
\begin{equation}\label{int-costs}
m_h(x):=\int\nu_h(dz)~ m(z,x)\quad\text{and}\quad
m_g(y):=\int\nu_g(dz)~ m(y,z).
\end{equation}
Note that
$
m_h\wedge m_g\geq m^-
$.
We further assume there is some $\delta\in ]0,1/2[$ such that the functions
$
(g_{\delta},h_{\delta})
$ grow at infinity (i.e. locally bounded with compact sub-level sets)  and such that
\begin{equation}\label{huv-prim}
\nu_{(1-2\delta)g}\left(
1\right)\vee \nu_{(1-\delta)h}\left(
1\right)<\infty.
\end{equation}
Note that $
(g_{\delta},h_{\delta})
$ are bounded below by real numbers $
(g^-_{\delta},h^-_{\delta})
$. 
Various  examples satisfying the above regularity conditions are discussed in~\cite{adpm-25}, slightly weaker conditions are also discussed in Section 5.6.1 in~\cite{dj-sinkhorn}. Intuitively, the growth condition at infinity ensures that the log-likelihoods $(g,h)$
dominate at infinity the (integrated) transition-potentials $(m_h,m_g)$. In terms of potential
energy conversion, the potentials $(g,h)$ should be large enough to compensate high energy
transition losses.

 \begin{examp}[Bivariate Gaussian model, continued]\label{bi-gauss-ex-2}
Consider the Gaussian model discussed in Example~\ref{bi-gauss-ex}. In this case, we have
\begin{align*}
g(x)&=\frac{1}{2}(x-\overline{\theta})^{\top}\overline{\sigma}^{-1}(x-\overline{\theta})+\log{\sqrt{\mbox{\rm det}(2\pi\overline{\sigma})}},\\
h(y)&=\frac{1}{2}(y-\theta)^{\top}\sigma^{-1}(y-\theta)+\log{\sqrt{\mbox{\rm det}(2\pi\sigma)}},
\end{align*}
so that condition (\ref{huv-prim}) is satisfied for all $\delta\in ]0,1/2[$, while $m^-=\log{\sqrt{\mbox{\rm det}(2\pi\tau)}}$. In addition, the integrated functions defined in \eqref{int-costs} are given by
\begin{align*}
m_h(x)&=\frac{1}{2}(x-\theta_0)^{\top}\tau^{-1}(x-\theta_0)+c_h,\\
m_g(y)&=\frac{1}{2}(\overline{\theta}-a-By)^{\top}\tau^{-1}(\overline{\theta}-a-By)+c_g,
\end{align*}
for some constants $c_h,c_g\in\RR$. Consequently, $g_{\delta}$ and $h_\delta$ are quadratic functions with Hessian matrices
$$
\nabla^2g_\delta=\delta\,\overline{\sigma}^{-1}-\tau^{-1}\quad\mbox{\rm and}\quad
\nabla^2h_\delta=\delta\,\sigma^{-1}-B^{\top}\tau^{-1}B,
$$
and thus $(g_{\delta},h_{\delta})$ grow at infinity as soon as these two matrices are positive definite, which is the case if
$$
\frac{\delta}{\Vert \overline{\sigma}\Vert_2}>\Vert \tau^{-1}\Vert_2
\quad\mbox{\rm and}\quad
\frac{\delta}{\Vert \sigma\Vert_2}>\Vert B\Vert_2^2\,\Vert \tau^{-1}\Vert_2 .
$$
For a given $\delta\in]0,1/2[$, the above conditions are met as soon as the covariance matrix $\tau$ is sufficiently large; for instance, when $\tau= t~I$ they are met as soon as $t>(\Vert \overline{\sigma}\Vert_2\vee \Vert B\Vert_2^2\Vert \sigma\Vert_2)/\delta$.
 More refined estimates are provided in Section 5.6.1 in~\cite{dj-sinkhorn}, see also  Proposition 5.3 in~\cite{adpm-25}.
  \end{examp}

\begin{examp}
Let $S=\RR^m$, let $\nu$ be the Lebesgue measure on $S$, and assume
that $(g,h)$ and $(m_{g},m_h)$ are continuous functions on $S$.
In addition, assume that there exist  some parameters $a_1,b_1\in\RR$,  $a_2,b_2> 0$ and $q\geq p>0$ such that
$$
m(x,y)\leq a_1+a_2~(\Vert x\Vert^{p}+\Vert y\Vert^{p})\quad \mbox{and}\quad
  g(x)\wedge h(x)\geq b_1+b_2~\Vert x\Vert^{q}.
  $$
  In this case \eqref{huv-prim} holds and $
(g_{\delta},h_{\delta})
$ are locally bounded with compact sub-level sets
 for all $\delta\in ]0,1/2[$  as soon as $q>p$, and for some $\delta\in ]0,1/2[$ as soon as $b_2/2>a_2$ and $q=p$.
  \end{examp}

\begin{lem}\label{lem-lyap-gibbs}
The Markov transitions $(K,L)$ defined in \eqref{def-K-gibbs} and \eqref{def-L-gibbs} 
satisfy 
\begin{eqnarray}
\exp{(-\delta g)}~
K\left(\exp{(\delta h)}\right)&\leq& c_{\delta,h}
~\exp{\left(-g_{\delta}\right)}\label{h2g}\\
\exp{(-\delta h)}~
L\left(\exp{(\delta g)}\right)&\leq &c_{\delta,g}~\exp{\left(- h_{\delta}\right)}
 \label{g2h}
\end{eqnarray}
with the parameters
$$
c_{\delta,h}:=\exp{(-m^-)}~\nu_{(1-\delta)h}(1)\quad \mbox{and}\quad
 c_{\delta,g}:=\nu_{(1-2\delta)g}(1)~\exp{(-m^--g_{\delta}^--\nu_g(g))}.
 $$
\end{lem}
\begin{proof}
Using Jensen's inequality we check that
$$
\frac{d\nu_h M}{d\nu}(x)=\int\nu_h(dz) \exp{(-m(z,x))}\geq \exp{\left(-m_h(x)\right)}
$$
 yielding  the estimate
\begin{align}\label{eq:LL}
\exp{(-\delta g(x))}~
K(x,dy)~\exp{(\delta h(y))}\leq c_0
~\exp{\left(-g_{\delta}(x)\right)}~\nu_{(1-\delta)h}(dy)
\end{align}
with $c_0:=\exp{(-m^-)}$. We  end the proof of (\ref{h2g}) integrating both sides of \eqref{eq:LL}.

To check (\ref{g2h}), note that
$$
M\left(\frac{d\nu_g}{d\nu_hM}\right)(y)=
\int~\nu_g(dx)~\exp{\left(g(x)-m(y,x)+\log{\frac{d\nu_g}{d\nu_hM}(x)}\right)}.
$$
Then, using Jensen's inequality, we check that
$$
M\left(\frac{d\nu_g}{d\nu_hM}\right)(y)\geq 
\exp{\left(\nu_g(g)-m_g(y)+\nu_g\left(\log{\frac{d\nu_g}{d\nu_hM}}\right)\right)}\geq \exp{\left(\nu_g(g)-m_g(y)\right)}.
$$
This yields the estimate
$$
L(y,dx)\leq \exp{\left(-\nu_g(g)+m_g(y)\right)}~e^{-m(y,x)}~\nu_g(dx)~\frac{d\nu}{d\nu_hM}(x)
$$
from which we check that
\begin{align}\label{eq:LLL}
e^{-\delta h(y)}
L(y,dx)e^{\delta g(x)}\leq c_1~\exp{\left(- h_{\delta}(y)\right)}~\nu_{(1-2\delta)g}(dx)
\end{align}
with $c_1:=\exp{(-m^--g_{\delta}^--\nu_g(g))}$, where we used the fact that $e^{-m(y,x)}\leq e^{-m^-}$. We
 end the proof of (\ref{g2h}) integrating both sides of \eqref{eq:LLL}.
\end{proof}

Lemma~\ref{lem-lyap-gibbs} ensures that 
transitions $(K,L)$ defined in \eqref{def-K-gibbs} and \eqref{def-L-gibbs} satisfy the drift conditions in \eqref{Lyap-KL} with $(W,V)=(\exp{(\delta h)},\exp{(\delta g)})$. To check this claim, given some $\epsilon\in ]0,1[$, consider some $r>0$ such that
$$
(c_{\delta,h}\vee  c_{\delta,g})
\exp{\left(-r\right)}\leq \epsilon.
$$
Then, by (\ref{h2g}) and  (\ref{g2h}), the drift conditions in \eqref{Lyap-KL} are satisfied with $\epsilon_0=\epsilon$, $c_0=c_{1,\delta}(r)\vee c_{2,\delta}(r) $ and the parameters
$$
c_{1,\delta}(r):=c_{\delta,h}~
\exp{\left(\sup_{g_{\delta}(x)\leq r}|\delta g(x)-g_{\delta}(x)|\right)}~\mbox{\rm and}~
c_{2,\delta}(r):=c_{\delta,g}~
\exp{\left(\sup_{h_{\delta}(x)\leq r}|\delta h(x)-h_{\delta}(x)|\right)}.
$$

\subsubsection{Local minorizations}

For all $r>1$, consider the level sets
$$
C_{V}(r):=\{V\leq r\}\quad \mbox{\rm and}\quad
C_{W}(r):=\{W\leq r\}.
$$

Since $S$ can be exhausted by each of the two increasing families of compact sub-level sets
$C_{V}(r)$ and $C_{W}(r)$, there exists some $r_0>1$ such that $\nu(C_{V}(r))\wedge \nu(C_W(r))>0$  for $r=r_0$ and thus for all $r\geq r_0$. This ensures that 
$
\nu_h(C_W(r))\wedge \nu_g(C_V(r))>0
$ for all $r\geq r_0$, and we set
$$
\nu_{g,r}(dx):=\frac{\nu_g(dx)~1_{C_V(r)}(x)}{\nu_g(C_V(r))}\quad \mbox{\rm and}\quad
\nu_{h,r}(dy):=\frac{\nu_h(dy)~1_{C_W(r)}(y)}{\nu_h(C_W(r))},\quad\forall r\geq r_0.
$$

\begin{lem}\label{lem:min-gibbs}
For all $r\geq r_0$ there exists some $\alpha(r)\in ]0,1[$ such that
for all $x\in C_{V}(r)$ and $y\in C_{W}(r)$ we have
\begin{equation}\label{min-loc-KL}
K(x,dy)~\geq \alpha(r)~\nu_{h,r}(dy)\quad \mbox{and}\quad
L(y,dx)\geq~\alpha(r)~\nu_{g,r}(dx).
\end{equation}
\end{lem}
\begin{proof}
Observe that, for all $x\in C_V(r)$, we have
$$
\begin{array}{l}
\displaystyle ~K(x,dy)\geq \exp{(m^-)}~
\nu_h(dy)~\exp{(-m(y,x))}~1_{C_W(r)}(y)\geq 
 \exp{(a(r))}~\nu_h(dy)~1_{C_W(r)}(y)
\end{array}$$
with the parameter
$$
a(r):=m^--\sup_{(x,y)\in (C_{V}(r)\times C_{W}(r))}m(y,x),
$$
which   implies that  
$$
K(x,dy)~\geq \alpha_h(r)~\nu_{h,r}(dy)
\quad \mbox{\rm with}\quad
\alpha_h(r):= \exp{(a(r))}~\nu_h(C_W(r)),\quad\forall  x\in C_V(r).
$$
In the same vein, we have
$$
M(y,dx)~\frac{d\nu_g}{d\nu_hM}(x)=\frac{\nu_g(dx)~e^{-m(y,x)}}{\int \nu_h(dy)~e^{-m(y,x)}}\geq e^{m^-}~\nu_g(dx)~e^{-m(y,x)}
$$
while, by Jensen's inequality,
$$
\begin{array}{rcl}
\displaystyle
~M(y,dx)~\frac{d\nu_g}{d\nu_hM}(x)&\leq& e^{-m^-}~\nu_g(dx)~ e^{m_h(x)} \\
&=&e^{-m^-}~\nu_{(1-\delta)g}(dx)~ e^{-g_{\delta}(x)}\leq 
\exp{(-m^--g_{\delta}^-)}~\nu_{(1-\delta)g}(dx).
\end{array}$$
For all $y\in C_W(r)$, this implies that
$$
\begin{array}{l}
\displaystyle 
L(y,dx)\geq~\alpha_g(r)~\nu_{g,r}(dx)
\end{array}$$
with the parameters
$$
\alpha_g(r):= \nu_g(C_V(r))~\frac{\exp{(b(r))}}{\nu_{(1-\delta)g}(1)}
\quad \mbox{\rm and}\quad
b(r):=2m^-+g^-_{\delta}-\sup_{(x,y)\in (C_{V}(r)\times C_{W}(r))}m(y,x).
$$
We conclude that the two inequalities in \eqref{min-loc-KL} are satisfied with $\alpha(r):=\alpha_g(r)\wedge \alpha_h(r)\wedge (1/2)$. This ends the proof of the lemma.
\end{proof}

The local minorization estimates \eqref{min-loc-KL}  ensure that
for all $r\geq r_0$, $(x_1,x_2)\in C_{V}(r)^2$ and $(y_1,y_2)\in  C_{W}(r)^2$
we have
$$
\Vert \delta_{x_1} K-\delta_{x_2} K\Vert_{\rm tv}\vee \Vert \delta_{y_1} L-\delta_{y_2} L\Vert_{\rm tv}\leq 1-\alpha(r)
$$
from which we conclude that condition \eqref{loc-contracKL} is met.
Recalling that \eqref{Lyap-KL} is also met, it follows from the
discussion in Section \ref{sub:two_op} that, for $P=KL$, the
conclusions of part (ii) of Theorem~\ref{th-intro-basic} hold; in
particular, there exists a $\rho>0$ such that
$\beta_{\varphi_\rho}(P)<1$, where the semi-distance
$\varphi_{\rho}$ is defined in (\ref{vero}) and
(\ref{def-phi-rho}). By symmetry, the minorization of $L$ in
\eqref{min-loc-KL} yields the same conclusion for the reversed
product $LK$, with the roles of $V$ and $W$ exchanged.

\subsection{Iterated random functions}\label{sub:IRF}

In this subsection we let $\Za$ be a measurable space equipped with
some probability distribution $\mu$, $(S,\Vert\cdot\Vert)$ be a
finite dimensional normed vector space, and we assume that 
\begin{equation}\label{FZ-def}
P(x,dy)=\PP\left(F_Z(x)\in dy\right),\quad Z\sim \mu
\end{equation}
for some function
$F:(x,z)\in (S\times\Za)\mapsto F_z(x)\in S$. In this scenario the Markov semigroup $(P_n)_{n\geq 0}$ takes the form
$$
P_n(x,dy)=\PP\left(X_n(x)\in dy\right)\quad\mbox{\rm with}\quad
X_n(x)=F_{Z_n}(X_{n-1}(x))=\left(F_{Z_n}\circ\ldots\circ F_{Z_1}\right)(x)
$$
where the $Z_n$'s  are independent copies of $Z$.  Markov chains are often  constructed by iterating random functions, and in this framework it is known that the semigroup is stable as soon as
the random functions exhibit a certain average contraction, 
 see for instance~\cite{df-99,meyn-tweedie,steinsaltz}.

\subsubsection{Some general local contraction results\label{sub:locRF}}

Assume that the function $x\mapsto F_z(x)$ is smooth for all
$z\in\Za$ and denote by $\nabla_1F_z(x)$ its Jacobian matrix, and assume that  the function $\Vert x-y\Vert^p$ is a $V$-semi-distance for some $p>0$. In this context condition  \ref{H2} is met with the semi-distance  
$\kappa(x,y):=\Vert x-y\Vert^p$   as soon as,  for all $r$ large enough, there exists some $\alpha_p(r)\in ]0,1[$ such that
\begin{equation}\label{wcond-p2}
\varpi_V (x,y)\leq r\Longrightarrow \EE\left(\Vert F_Z(x)-F_Z(y)\Vert^p\right)\leq (1-\alpha_p(r))  \Vert x-y\Vert^p.
\end{equation}
Notice that, by the mean value theorem, \eqref{wcond-p2} holds with
$p=1$ if
$$\sup_{ x\in \mbox{\rm co}\left(\{V\leq r\}\right)}\EE\left(\Vert \nabla_1F_Z(x)\Vert\right)\leq 1-\alpha_1(r)$$
for all $r$ large enough, where $\mbox{\rm co}\left(\{V\leq r\}\right)$ stands for the convex hull of  the compact sub-level sets $\{V\leq r\}$, assuming that $V\in  \Ba_{\infty}(S)$. It is also worth noting that if $F_z=F+z$ for some function $F:x\in S\mapsto F(x)\in S$ then condition \eqref{wcond-p2} depends only on $F$. In this case,  \ref{H2} is met with $\kappa(x,y):=\Vert x-y\Vert^p$ for all distributions $\mu$ of $Z$.

Alternatively, if $V\in  \Ba_{\infty}(S)$ and $P(x,dy)=p(x,y)\,\nu(dy)$ for some continuous and strictly  positive transition density $p(x,y)$  then, by Lemma \ref{lemma:minor}, \ref{H2} holds for the semi-distance $\kappa=\varphi_0$. 

\subsubsection{A general approach to construct Lyapunov functions}\label{lyap-ref-if-intro}

We now assume that $F_Z =F +G_{Z}$ for some functions $F:x\in S\mapsto F(x)\in S$ and
$G:(x,z)\in (S\times\Za)\mapsto G_z(x)\in S$. In addition, we assume that there exist  an $x_0\in S$ as well as some parameters $ 0< \lambda(x_0)<1$ and
$c(x_0)\geq 0$ such that  
\begin{equation}\label{hypFx}
\Vert F(x)-F(x_0)\Vert\leq \lambda(x_0)\,
\Vert x-x_0\Vert+c(x_0),\quad\forall x\in S.
\end{equation}
To make the connection with the geometric drift condition imposed in \ref{H1},
 condition  \eqref{hypFx}  ensures that the norm is a Lyapunov function of the deterministic part of the transition.
 
Assume first that there exists a $\delta>0$ such that $$\sup_{x\in S}
\EE\big(e^{\delta \Vert G_Z(x)\Vert}\big)     <\infty.$$ Then, noting that
$\Vert F(x)\Vert\leq \lambda_0~
\Vert x\Vert+c_{\lambda}$ with
$\lambda_0:=\lambda(x_0)$ and   $
c_{\lambda}:=c(x_0)+\Vert F(x_0)\Vert+
\lambda(x_0)\Vert x_0\Vert$,
we have
$$
\EE\left(e^{\delta \Vert F(x)+G_Z(x)\Vert}\right)\leq 
a_1~
e^{\delta  \lambda_0
\Vert x\Vert}\quad \mbox{\rm with}\quad a_1:=e^{\delta  c_{\lambda}}~\sup_{x\in S}
\EE\left(e^{\delta \Vert G_Z(x)\Vert}\right).
$$
We conclude that $P(W)/W\leq \Theta\in \Ba_0(S)$  with the functions
$$
W(x):=\exp{(\delta\Vert x\Vert)}
\quad \mbox{\rm and}\quad
 \Theta(x):= a_1~\exp{\left(-\delta (1-\lambda_0)~
\Vert x\Vert\right)}
$$
and thus,  by \eqref{eq:Lap}, condition \ref{H1} holds with $V=W\in \Ba_{\infty}(S)$.

Assume now that  $$\sup_{x\in S}\EE\left(\Vert G_Z(x)\Vert^p\right)<\infty$$ for some $p\geq 1$. In this scenario, letting $$c_{\lambda,p} =c_{\lambda}+\sup_{x\in S}\EE\left(\Vert G_Z(x)\Vert^p\right)^{1/p},$$ we observe that
\begin{align*}
\EE\left(\Vert F(x)+G_Z(x)\Vert^p\right)&\leq \left(\EE\left(\Vert G_Z(x)\Vert^p\right)^{1/p}+\Vert F(x)\Vert\right)^{p}\leq \left(\lambda_0~
\Vert x\Vert+c_{\lambda,p}\right)^{p}.
\end{align*}
Then, by choosing $r>0$ sufficiently large so that
$$
\frac{c_{\lambda,p}}{\lambda_0 r}\leq 1-\lambda_0\Longrightarrow
\lambda_0\,\left(1+\frac{c_{\lambda,p}}{\lambda_0 r}\right)\leq \lambda_1:=
1-(1-\lambda_0)^2
$$
we obtain that
\begin{align*}
\EE\left(\Vert F(x)+G_Z(x)\Vert^p\right)&\leq  
\lambda_0^p~
\Vert x\Vert^p~\left(1+c_{\lambda,p}/(\lambda_0 r)\right)^{p}
1_{\Vert x\Vert\geq r}+
\left(\lambda_0\,r+c_{\lambda,p}\right)^{p}\\
&\leq  \lambda_1^p~\Vert x\Vert^p+c_{\lambda,p}(r) 
\end{align*}
where $c_{\lambda,p}(r):=\left(\lambda_0\,r+c_{\lambda,p}\right)^{p}$. Using the above calculations we conclude that the function $W(x):=\frac{1}{2}+\Vert x\Vert^p$ is such that $$P(W)\leq \lambda_1^p~W+(1/2+c_{\lambda,p}(r)),$$ showing that \ref{H1} holds with $V=W\in\Ba_{\infty}(S)$.

\subsubsection{Example 1: Random iterated functions with local dissipative conditions}
 Consider the functions $(F,F_z,G_z)$ defined in Section \ref{lyap-ref-if-intro} with $\Za=\RR^m=S$ and 
$$
 F(x)=x+b(x)~h\quad \mbox{\rm and}\quad G_z(x)=z
$$
for some parameter $h>0$ and for a smooth function $b(x)$ with Lipschitz constant $\mbox{\rm lip}(b)>0$.

The following regularity conditions are discussed in the series of articles~\cite{bao,bao-hao,huang-21,monmarche-25,luo}:

$(B_1)$ There exist $r\geq 0$, $\lambda >0$ and $c\geq 0$ such that for all $x\not=y$ we have
\begin{equation}\label{cond-conv}
\frac{(x-y)^{\top}(b(x)-b(y))}{\Vert x-y\Vert^2}\leq -\lambda~1_{\Vert x\Vert\vee \Vert y\Vert\geq r}+ c~1_{\Vert x\Vert\vee \Vert y\Vert< r}
\end{equation}

$(B_2)$ There exist $r\geq 0$, $\lambda >0$ and $c\geq 0$ such that for all $x\not=y$ we have
\begin{equation}\label{cond-conv-b2}
\frac{(x-y)^{\top} (b(x)-b(y))}{\Vert x-y\Vert^2}\leq -\lambda~1_{\Vert x- y\Vert\geq r}+ c~1_{\Vert x- y\Vert< r}
\end{equation}

The case $r=0$ is equivalent to the uniform dissipative condition.
When $Z$ is a Gaussian random variable with covariance $\EE(ZZ^{\top})=hI$ the model coincides with the  transition kernel of the discretized diffusion process, under a contractivity at infinity condition (on the drift (\ref{cond-conv}))  discussed in~\cite{monmarche-25}. 
 We point out that more general $\alpha$-stable random variables $Z$ are discussed in~\cite{huang-21} and  that condition (\ref{cond-conv-b2}) coincides with the  local dissipative drift-type condition discussed in the articles~\cite{bdurmus-19,durmus-24,huang-21,luo}.

By applying (\ref{cond-conv}) or (\ref{cond-conv-b2}) with $y=0$  we have
\begin{equation*} 
\Vert x\Vert\geq r\Longrightarrow
x^{\top} (b(x)-b(0))\leq -\lambda~\Vert x\Vert^2.
\end{equation*}
Note that the above implies that $\lambda\leq \mbox{\rm lip}(b)$. We now let $h\in]0,\lambda/\mbox{\rm lip}(b)^2[$, so that
$$
\lambda_h^2:=1-2\lambda h+\mbox{\rm lip}(b)^2h^2~\in~]0,1[.
$$
For all $x\in S$ such that $\Vert x\Vert\geq r$ we then have
\begin{align*}
\Vert F(x)-F(0)\Vert^2&= \Vert x\Vert^2+2h~x^{\top}(b(x)-b(0))+h^2~\Vert b(x)-b(0)\Vert^2\leq  \lambda_h^2~\Vert x\Vert^2,
\end{align*}
while, for $\Vert x\Vert< r$, we have $$\Vert F(x)-F(0)\Vert\leq (1+h\,\mbox{\rm lip}(b))\,r.$$ Therefore,
$$
\Vert F(x)-F(0)\Vert\leq 
\lambda(0)~\Vert x\Vert+c(0),\quad\forall x\in S,
$$
with $\lambda(0):=\lambda_h\in]0,1[$ and
$c(0):=(1+h\,\mbox{\rm lip}(b))\,r$.
This shows that condition  \eqref{hypFx}  is met with $x_0=0$ and   the results of Section~\ref{lyap-ref-if-intro} can be used to obtain a function $V\in \Ba_{\infty}(S)$  such that \ref{H1} holds. 

Specifically, if  $Z$ is a Gaussian random variable it follows from the calculations in Section \ref{lyap-ref-if-intro} that \ref{H1} is met when $V=\exp(\delta\|x\|)$ for some $\delta>0$. If we further assume that the function $b$ is continuous, it follows from the discussion of Section \ref{sub:locRF} that condition \ref{H2} is met with $\kappa=\varphi_0$. In this context, the calculations in Section \ref{sub:Wasser} show
that, for all $p\geq 1$ and with $\psi_S(x,y)=\|x-y\|$, the
following exponential contraction estimate for the Wasserstein
$p$-distance holds, for some constants $c>0$ and $\lambda_p\in]0,1[$
depending on $p$: For  any $n\geq 0$ and $\mu_1,\mu_2\in\mathcal{P}_V(S)$ we have
 \begin{align*}
\WW_{\psi_S,p}\left(\mu_1P_n,\mu_2P_n\right) \leq c \lambda_p^n\|\mu_1-\mu_2\|^{1/p}_V.
\end{align*}
Under the conditions (\ref{cond-conv}) or (\ref{cond-conv-b2}),
Wasserstein contraction estimates for discretized diffusions with
explicit rates, often in the form of strict contractions, are
obtained in~\cite{bao,bao-hao,huang-21,monmarche-25,luo} by means of
coupling methods. The above estimate should therefore be seen as an
illustration of how such results can be recovered, in a weaker form,
from conditions \ref{H1} and \ref{H2} with a short argument, rather
than as an improvement of these results.

\subsubsection{Example 2: A fixed point theorem\label{sub:fp}}

Let  $F: S \rightarrow S$ be a
mapping such that \eqref{hypFx} holds for some $x_0\in S$ and parameters $\lambda(x_0)\in]0,1[$ and $c(x_0)\geq 0$. Next, we let $\delta>0$, $W(x)=\exp{(\delta\Vert x\Vert)}$  and assume that there is an $r_0>0$ such that, for all $r\geq r_0$,  there exists an $\alpha(r)\in ]0,1[$ such that
\begin{align}\label{eq:locF}
 W(x)+W(y)\leq r\implies \|F(x)-F(y)\|\leq (1-\alpha(r))\|x-y\|.
\end{align}
Condition \eqref{eq:locF} implies that $F$ is  uniformly locally contractive (\cite{ciesielski-18}, Definition 2.1),  and by a theorem of Edelstein~\cite{edelstein-61}  this local contraction property implies that $F$ has a unique fixed point $y^*\in S$ (see also Remark 2.2 and Theorem 3.5 in~\cite{ciesielski-18}). Below we use Proposition \ref{cor:invariant} to show this result, and Theorem \ref{thm:expo} to prove that the sequence $(y_n)_{n\geq 1}$ defined for some $y_0\in S$ by the recursion $y_n=F(y_{n-1})$, $n\geq 1$,  converges exponentially fast to $y^*$.

To this aim we consider the Markov kernel
$P(x,dy)=\delta_{F(x)}(dy)$ and the cost
function $\kappa(x,y)=\|x-y\|$. Following
the discussion of
Section~\ref{lyap-ref-if-intro}, condition
\ref{H1} is met with $V=W$. In addition,
since $\delta\,\|x\|\leq W(x)$, we have
$\kappa(x,y)\leq (W(x)+W(y))/\delta$, so
that $\kappa$ is a $V$-semi-distance, and
since
$$
D_{\kappa}(\delta_xP,\delta_yP)
=\|F(x)-F(y)\|,
$$
condition \eqref{eq:locF} is exactly
condition \ref{H2} for this cost function.
Therefore, Proposition~\ref{cor:invariant},
applied with $p=1$ and
$\psi_S(x,y)=\kappa(x,y)=\|x-y\|$,
ensures that $P$ has a unique invariant
distribution $\pi\in\mathcal{P}_V(S)$.

Moreover, its proof shows that,
for all $y_0\in S$, the sequence $(\delta_{y_0}P_n)_{n\geq 0}=(\delta_{y_n})_{n\geq 0}$
converges to $\pi$ in the Wasserstein 1-distance, and is a Cauchy
sequence for this distance. Since
$$\WW_{\psi_S,1}(\delta_{y_n},\delta_{y_m})=\|y_n-y_m\|$$ and $S$ is
complete, the sequence $(y_n)_{n\geq 0}$ converges to some $y^*\in S$,
so that $\pi=\delta_{y^*}$. Since $\pi P=\pi$, we have
$\delta_{F(y^*)}=\delta_{y^*}$, that is $F(y^*)=y^*$. Finally, any
fixed point $y$ of $F$ yields an invariant distribution $\delta_y$,
and thus, by uniqueness, $y^*$ is the unique fixed point of $F$. Moreover, by using part (i) of Theorem \ref{thm:expo} with $\iota=1$ and $(\mu_1,\mu_2)=(\delta_{y_0},\delta_{y^*})$, it follows that there exist constants $c>0$ and $\lambda\in]0,1[$ such that
\begin{align*}
\|y_n-y^*\|\leq c~\lambda^n~ \varphi_V(y_0,y^*),\quad\forall n\geq 0.
\end{align*}

\subsection{Diffusion processes}\label{sub:diffusion}

\subsubsection{A general diffusion model}\label{sect-gen-diff}
 
We let $S=\RR^m$, $\psi_S(x,y)=\|x-y\|$ and we consider below the continuous time Markov semigroup $(Q_t)_{t\in]0,\infty[}$ of a diffusion flow $y\in S\mapsto X_{t}(y) $   defined  by
\begin{equation}\label{ref-X-diff-stable}
dX_{t}(y)=b(X_{t}(y))\,dt+\tau(X_t(y))\, dB_t\quad \mbox{\rm with}\quad X_0(y)=y.
\end{equation}
In the above display, $B_t$ is a $m_1$-dimensional Brownian motion starting at the origin   and $b$ is a differentiable drift function from $\RR^m$ into itself satisfying, for some constants $a>0$ and $c\geq 0$, 
\begin{equation}\label{xpb}
 x^{\top}b(x)\leq -a\,\Vert x\Vert^2+c,\quad\forall x\in S.
\end{equation}
Moreover,  $\tau $ stands for some $(m\times m_1)$-matrix valued continuous function such that, for all $x\in S$, the matrix $\tau(x)\tau(x)^{\top}$ is invertible and the map
 $x\mapsto \Vert \tau(x)\Vert/\Vert x\Vert$ vanishes at infinity. In what follows, for all $x\in S$ we  let  
$$
\Vert \tau(x)\Vert_F^2:=\tr(\tau(x)\tau(x)^{\top})\leq m \Vert \tau(x)\Vert_2^2\quad\text{  with } 
 \Vert \tau(x)\Vert_2^2=\lambda_{\max}(\tau(x)\tau(x)^{\top}) 
$$
and for all   $p\geq 1$ we let $W^{p}(x):=\Vert x\Vert^{p}$.

Under the above assumptions on the functions $b$ and $\tau$, by using It\^o calculus one can check that  the generator $\Ga$  of the semi-group $(Q_t)_{t\in]0,\infty[}$ is such that
\begin{eqnarray*}
\Ga(W^{2p})(x)&\leq& p~W^{2p}(x)~\left(-2a+\frac{2c+\Vert \tau(x)\Vert_F^2}{\Vert x\Vert^2}+2(p-1)\frac{\Vert \tau(x)\Vert_2^2}{\Vert x\Vert^2}\right)
\end{eqnarray*}
for all $p\geq 1$ with the constants $a$ and $c$ as in (\ref{xpb}). Choosing $r>0$ sufficiently large so that, for a given $p\geq 1$,
\begin{eqnarray*}
\Vert x\Vert >r&\Longrightarrow&
2c+\Vert \tau(x)\Vert_F^2+2(p-1)\Vert \tau(x)\Vert_2^2\leq a~\Vert x\Vert^2\\
&\Longrightarrow&\Ga(W^{2p})(x)\leq -ap~W^{2p}(x),
\end{eqnarray*}
we obtain that, for all $p\geq 2$, there exist  constants $a_{0,p}>0$ and $a_{1,p}\geq 0$ such that
\begin{align}\label{eq:L_Diff}
\Ga(V_{p})\leq -a_{0,p}~ V_{p}+a_{1,p}~~\text{ with }~~  V_{p}:=\frac{1}{2}+W^{p}
\end{align}
and thus  $V_p$ satisfies \eqref{V-L-ex}.

We now let $h\in]0,1[$, $p\geq 2$ and $P=Q_h$. Then,   following the discussion of Section \ref{sub:continuous}, the result in \eqref{eq:L_Diff} implies that for $P=Q_h$ condition \ref{H1} is met for $V=(1/2)+\|x\|^{p}$. We now assume that for $P=Q_h$ condition \ref{H2} is met with either $\kappa=\varphi_0$ or with a $V$-semi-distance $\kappa$ such that $\kappa(x,y)\geq \psi_S(x,y)^p$.   In this context, it follows from \eqref{eq:TT} (when $\kappa\geq\psi_S^p$) or from \eqref{eq:T2} applied with $\phi=\psi_S^p$ (when $\kappa=\varphi_0$. Recall from \eqref{cond-intro-p-norms} that $\psi_S^p$ is a $V$-semi-distance) that there exist constants $c'_p>0$ and $\lambda_p\in]0,1[$ such that for all
$t\geq 0$ and $\mu_1,\mu_2\in\mathcal{P}_V(S)$ we have
\begin{align*}
\WW_{\psi_S,p}\left(\mu_1 Q_t,\mu_2 Q_t\right)^p=D_{\psi_S^p}\left(\mu_1 Q_t,\mu_2 Q_t\right) \leq c'_p~ \lambda_p^{t}~\|\mu_1-\mu_2\|_V.
\end{align*}
 In particular, we observe that if the functions $b$ and $\tau$ are sufficiently regular for $Q_h$ to admit a continuous and strictly positive transition density, then, for $P=Q_h$, the local minorization condition  \eqref{loc-min}  is met by Lemma~\ref{lemma:minor} since $V\in \Ba_{\infty}(S)$, in which case \ref{H2}  indeed holds with $\kappa=\varphi_0$.  

\subsubsection{The overdamped Langevin diffusion model}\label{langevin-sec}

An important class of diffusion processes arising in statistical
physics and applied probability is the overdamped Langevin
diffusion flow $X_t(y)$, defined as in (\ref{ref-X-diff-stable})
with $(m_1,\tau(x))=(m,\sigma I)$ for some $\sigma>0$ and the drift
function $b(x)=-\gamma\nabla U(x)$ associated with some parameter
$\gamma>0$ and smooth potential function $U$ satisfying the
following condition
\begin{equation}\label{ex-lg-cg-v2}
\nabla^2 U(x)\geq
\lambda~1_{\Vert x\Vert \geq  r}~I-c
~1_{\Vert x\Vert< r}~I\quad
\mbox{for some $\lambda>0$, $c\geq 0$ and $r\geq 0$.}
\end{equation}
The local Hessian condition (\ref{ex-lg-cg-v2}) is often termed
``strong convexity outside a ball''. To avoid confusion, we
underline that this condition strongly differs from the strong
convexity of $U$ on the non convex domain
$\{x\in S:\Vert x\Vert\geq r\}$ discussed in~\cite{ma,peters,yan}.

The following lemma is a radial version of the
construction in Section~3.3 of~\cite{monmarche-AHL},
with explicit bounds on the gradient and on the
oscillation of $\widetilde{U}$. The boundedness of
$\nabla\widetilde{U}$, which is not needed in the
argument of~\cite[Section~3.3]{monmarche-AHL}, plays
a key role in the comparison estimate \eqref{uXoX}
below. The proof of the lemma is provided in
Appendix~\ref{app:Utilde}.

\begin{lem}\label{lem:Utilde}
Assume that condition (\ref{ex-lg-cg-v2}) holds and set
$\overline{r}:=r+1$. Then there exists a smooth and bounded
function $\widetilde{U}$ on $S$ such that
$$
\nabla^2(U+\widetilde{U})\geq a~I\quad\mbox{\rm with}\quad
a:=\lambda/2.
$$
In addition, $\nabla\widetilde{U}$ vanishes outside the ball of
radius $2\overline{r}(1+c/\lambda)+1$, and we have
$$
\Vert \nabla\widetilde{U}\Vert_\infty\leq
\left(c+\frac{\lambda}{2}\right)\overline{r}
\quad\mbox{and}\quad
\mbox{\rm osc}(\widetilde{U})\leq
\frac{\overline{r}^{\,2}\,(2c+\lambda)(c+\lambda)}{2\lambda}.
$$
\end{lem}

Let $\widetilde{U}$ be the function given by
Lemma~\ref{lem:Utilde}, and let $(\widetilde{Q}_t)_{t\in]0,\infty[}$
be the continuous time Markov semigroup of the stochastic flow
$\widetilde{X}_t(y)$ defined as $X_t(y)$ by replacing the drift
function $b(x)=-\gamma\nabla U(x)$ by
$$
\widetilde{b}:=-\gamma\nabla (U+\widetilde{U})\quad
\Longrightarrow\quad
\nabla \widetilde{b}(x)\leq -\gamma a~I.
$$
Following the proof of Proposition 2.16 in~\cite{ado-24}, for all
$\delta>0$ and $h>0$ there exist some constants $\delta_h>0$ and
$c'_h>0$ such that
\begin{align}\label{eq:VV1}
W(x):=\exp(\delta\|x\|)\Longrightarrow
\sup_{0\leq h'\leq 1}\Vert \widetilde{Q}_{h'}(W)/W\Vert<\infty
\text{ and }
\widetilde{Q}_{h}(W)/W\leq c'_h/W^{\delta_h}.
\end{align}
On the other hand, using the stochastic interpolation formula
(cf.~Theorem 1.2 in~\cite{dss-19}) and the boundedness of
$b-\widetilde{b}=\gamma\nabla\widetilde{U}$, we check the almost
sure uniform estimate
\begin{equation}\label{uXoX}
\Vert  X_t(x)-\widetilde{X}_t(x)\Vert\leq c_0:=
\frac{\Vert\nabla\widetilde{U}\Vert_\infty}{a},
\quad\forall t\geq 0,\quad\forall x\in S.
\end{equation}
By combining \eqref{eq:VV1} and \eqref{uXoX}, we obtain that
$$
Q_{h}(W)/W\leq
e^{\delta c_0}~\widetilde{Q}_h(W)/W\leq
c'_h~e^{\delta c_0}~/W^{\delta_h}.
$$
Recalling that $W\in \mathcal{B}_\infty(S)$, we conclude that
condition (\ref{ex-lg-cg-v2}) ensures the existence of constants
$c_h>0$ and $\epsilon_h\in]0,1[$ such that
\begin{equation}\label{lyap-ref-3}
\begin{array}{l}
\sup_{0\leq h'\leq 1}\Vert Q_{h'}(W)/W\Vert<\infty
\quad\text{ and } \quad
Q_{h}(W)\leq  \epsilon_h W+c_h
\end{array}
\end{equation}
and thus, for $P=Q_h$, condition \ref{H1} is met with $V=W$.
In addition, since $U$ is smooth and the diffusion coefficient is
constant, $Q_h$ admits a continuous and strictly positive
transition density. Hence, by Lemma~\ref{lemma:minor}, the local
minorization condition \eqref{loc-min} is met for $P=Q_h$, since
$V\in\mathcal{B}_\infty(S)$, and thus for $P=Q_h$ condition
\ref{H2} holds with $\kappa=\varphi_0$. Since in addition we have
$\sup_{0\leq h'\leq 1}\Vert Q_{h'}(V)/V\Vert<\infty$ by
\eqref{lyap-ref-3}, the estimate \eqref{eq:T2} holds for all
$V$-semi-distance $\phi$, as shown in Section \ref{sub:continuous}.
In particular, applying \eqref{eq:T2} with the $V$-semi-distance
$\phi=\varpi_e$ (see \eqref{cond-intro-exp}) and following the
discussion in Section \ref{sub:Wasser}, with $\psi_S$ the
Euclidean distance on $S$, there exist constants $c_1>0$ and
$\lambda_1\in]0,1[$ such that
\begin{equation}\label{ref-cv-diff-contract}
\WW_{\psi_S,p}\left(\mu_1 Q_t,\mu_2 Q_t\right)^p
\leq c_p^{-1}\,c_1~ \lambda_1^t~\|\mu_1-\mu_2\|_V,
\end{equation}
for all $t\geq 0$, $p\in\NN-\{0\}$ and
$\mu_1,\mu_2\in\mathcal{P}_V(S)$, where for all $p\geq 1$ the
constant $c_p$ is as defined in \eqref{cond-intro-exp}.

Exponential contraction of overdamped Langevin diffusions under the
condition (\ref{ex-lg-cg-v2}) has been extensively studied with
coupling techniques. Reflection couplings~\cite{eberle-ptrf,eberle}
yield a strict contraction of $Q_t$, for all $t>0$, with respect to a
carefully designed Wasserstein-type distance, with explicit rates,
while~\cite{monmarche-AHL} establishes a strict $2$-Wasserstein
contraction with explicit and dimension-sharp rates under a high
diffusivity assumption; see also the recent preprint~\cite{mon-26}
for $L^p$-Wasserstein contraction without global dissipativity.
Compared with these works, the estimate (\ref{ref-cv-diff-contract})
holds for all $\sigma>0$ and any integer $p\geq 1$, and follows from a
short argument based only on conditions \ref{H1} and \ref{H2}.
However, the constants $c_1$ and $\lambda_1$ are not explicit, and
are expected to be far from optimal since they rely on a local
minorization condition, and $Q_t$ is not shown to be a strict
contraction. On the other hand, since \ref{H2} holds with
$\kappa=\varphi_0$, the estimate \eqref{eq:T2} applied with
$\phi=\varphi_V$ shows that $\beta_{\varphi_V}(Q_t)\leq c\,e^{-\varsigma t}$
for some constants $c,\varsigma>0$, which is equivalent to the
$V$-uniform ergodicity of $Q_t$ (see Remark~\ref{rmk:V-unif}), that
is, to a spectral gap property on the weighted space $\Ba_V(S)$
(but not in $L^2(\pi)$).

\subsection{Processes on domains with a boundary}\label{sub:boundary-processes}

Following Section~\ref{sub:continuous}, Markov transitions of continuous time processes on domains with a boundary
can be treated as Markov chains on domains with a boundary (see Section~\ref{bd-sec}).

\subsubsection{Jump-diffusion on compact manifolds }\label{compact-diffusions-sec-end}

The transition semigroup $Q_t(x,dy)$ of a jump-type elliptic diffusions on compact manifolds with bounded jump rate
satisfies  the uniform minorization condition (\ref{global-min}) for all time horizons $t>0$; see for instance the pioneering work of Aronson~\cite{aronson}, Nash~\cite{nash} and Varopoulos~\cite{varopoulos} on Gaussian estimates for heat
kernels on manifolds. 

For instance, consider the Brownian motion starting at $X_0(x)=x\in S=[0,1]$ reflected at the boundaries $\partial S=\{0,1\}$. The transition density $Q_t(x,dy)=q_t(x,y)dy$ of this reflected process coincides with the Neumann heat semigroup, and the density $q_t(x,y)$ is uniformly bounded and uniformly bounded away from zero on the compact interval $[0,1]$.  In this context, for all time steps $h>0$ the uniform minorization condition (\ref{global-min}) is satisfied for the Markov transition
$P=Q_{h}$ and the uniform probability measure $\nu(dx)=1_{[0,1]}(x)dx$. Hence the total variation contraction estimate (\ref{ref-tv-contract}) holds, and thus $\beta_{\varphi_0}(Q_{h})<1$. Noting that the compact unit interval is complete with respect to the Euclidean metric $\psi_e(x,y):=|x-y|$, and that $\psi_e(x,y)\leq \varphi_0(x,y)$ for all $x,y\in[0,1]$, we obtain the Wasserstein exponential contraction estimate 
$$
\WW_{\psi_e,1}(\mu_1 Q_{nh},\mu_2Q_{nh})\leq \|\mu_1 Q_{nh}-\mu_2Q_{nh}\|_{\rm tv}\leq \beta_{\varphi_0}(Q_{h})^n\,\|\mu_1-\mu_2\|_{\rm tv},\quad\forall n\geq 0.
$$

From a different perspective, following Example 2.13 in~\cite{ado-24}, the Neumann heat semigroup $q_t(x,y)$
is symmetric and admits an explicit
spectral decomposition. Since its first non-trivial eigenvalue is explicitly known, we underline that this spectral decomposition provides more refined convergence to equilibrium than the total variation contraction estimates discussed above.

\subsubsection{Jacobi diffusion (a.k.a. Wright-Fisher diffusion)}\label{sub:jacobi}

The Jacobi diffusion $X_t$ 
 with entrance boundaries on the open interval $S=]0,1[$ is defined as in (\ref{ref-X-diff-stable}) by choosing a linear drift
$b(x)=\theta_0-(\theta_0+\theta_1)x$ and a diffusion function $\tau(x)=\sigma~\sqrt{x(1-x)}$ for some $\sigma>0$ and some parameters $\theta_0,\theta_1$ such that
$2\theta_0/\sigma^2>1$ and $2\theta_1/\sigma^2>1$. For all $t>0$ we have
$$
\PP(X_t\in dy~|~X_0=x)=
Q_t(x,dy)=q_t(x,y)~dy
$$
for a uniformly bounded transition density
$\sup_{(x,y)\in S^2}q_t(x,y)<\infty$. In contrast with the reflected diffusion discussed in Section~\ref{compact-diffusions-sec-end}, the transition density $q_t(x,y)$ of the Jacobi process is not uniformly positive across the entire interval $x,y\in ]0,1[$.
Following the arguments provided in Section~\ref{sec-illustrations-interval} and  Section~\ref{bd-sec}, the function $W_{\iota}$ defined in (\ref{eq:W_def}) is a Lyapunov function and, for $\iota\in]0,1[$, since $\Vert Q_h(W_{\iota})\Vert<\infty$, condition \ref{H1} holds with $V=W_\iota
$ and $P=Q_h$ for all time steps $h>0$.

The uniform boundedness of the transition density can be checked using the 
Gaussian bounds for Jacobi heat kernels presented in~\cite{coulhon,kerk-petrus}.
Similar two-sided Gaussian estimates for heat kernels on balls and simplex are also provided in~\cite{kerk-petrus-2}. These estimates also ensure that the  local minorization condition
 \eqref{loc-min} is satisfied on the compact sub-level sets of $W_{\iota}$ and  thus \ref{H2} is met with the  functions $(V,\kappa)=(W_{\iota},\varphi_0)$.
 
As $t\rightarrow\infty$, the Jacobi diffusion converges to its unique invariant measure, which is a Beta distribution.
Arguing as in Section~\ref{sub:continuous}, it also  follows from part (ii) of Theorem \ref{thm:expo} that the exponential contraction estimate \eqref{eq:Wass_S2}, with $P_n$ replaced by $Q_{nh}$, holds for 
the Wasserstein distance $\WW_{\psi_S,1}$ associated with  the complete metric $\psi_S=\psi_\iota$  defined in (\ref{Ref-upsi}). Similar arguments apply to the reflected diffusion on $[0,1]$ discussed in Section~\ref{compact-diffusions-sec-end}. In this scenario, 
 the reflected Brownian motion converges to its unique invariant measure which is the uniform distribution on the compact interval. We underline that the complete metric $\psi_S=\psi_\iota$  differs from the more conventional but non complete (and thus non equivalent) intrinsic arcsin-metric associated with the Jacobi carr\'e-du-champ defined by
 \begin{equation}\label{arcsin-metric}
 \psi_{\rm in}(x,y)=\left\vert \int_x^y\frac{1}{\tau(z)}~dz\right\vert=\frac{2}{\sigma}~\vert
\arcsin(\sqrt{x})-\arcsin(\sqrt{y}) \vert.
 \end{equation} 

 The Bakry--\'Emery curvature condition yields gradient semigroup estimates with respect to the intrinsic metric. By Kuwada's duality~\cite{kuwada}, these gradient estimates yield contraction estimates for the
Wasserstein distance computed with the intrinsic metric~\cite{bakry,kuwada, otto}. Also note that the arcsin metric transforms the diffusion to have constant noise profile. Indeed, using the Lamperti transform
\begin{eqnarray*}
Y_t:=2\arcsin(\sqrt{X_t})\in ]0,\pi[&\Longleftrightarrow&
X_t^{1/2}=\sin(Y_t/2)\in ]0,1[\\
&\Longrightarrow& \sqrt{X_t(1-X_t)}=\frac{1}{2}~\sin(Y_t)\quad\mbox{\rm and}\quad
X_t=(1-\cos(Y_t))/2
\end{eqnarray*}
we check that
\begin{eqnarray*}
dY_t
&=&\left(\frac{\theta_0-\theta_1}{\sin(Y_t)}+\left((\theta_0+\theta_1)-\frac{\sigma^2}{2}\right)~\cot(Y_t)\right)~dt+\sigma dB_t.
\end{eqnarray*}

This Lamperti transform is an elegant tool to compute directly the gradient of the drift of the diffusion with constant noise but the state space under the Euclidean or the arcsin metric remains incomplete. The stability analysis of the process based on coupling is deduced from a global contraction argument (or a gradient estimate) on the entire open unit interval. 
 Instead of demanding global contraction, the Lyapunov approach discussed above 
  splits the space into an elliptic core (a central compact interval where noise is strong and paths mix easily) and the outer boundary (where the Lyapunov drift forcefully pushes paths back to the core). This framework allows one to bypass the global contraction requirement: we only need a local contraction condition inside some compact and a Lyapunov drift condition  near the boundaries.

To get one step further in our discussion, we also recall that the semigroup density $q_t(x,y)$ of the Jacobi diffusion has an explicit spectral expansion in terms of Jacobi polynomials.
These expansions can also be used to prove the uniform boundedness of the transition density. More interestingly, since the first non-trivial eigenvalue is explicitly known, these spectral decompositions provide refined estimates for the convergence to equilibrium, much sharper than those derived from a Poincar\'e inequality alone. These spectral estimates and the semigroup gradient estimates  differ from the
 Wasserstein exponential decays w.r.t. the complete metric $\psi_S=\psi_\iota$ discussed above.

\subsubsection{Stochastic Riccati diffusion  }
Let $\overline{S}$ and $S$ be the spaces of $(m\times m)$ positive semi-definite and positive definite matrices, respectively, and let $\lambda_1(x)\geq \ldots\geq \lambda_m(x)$ denote the ordered eigenvalues of $x\in \overline{S}$.
Let $\Wa_t$ denote an $(m\times m)$-matrix with independent Brownian entries. Also let $A$ be an $(m\times m)$-matrix with real entries and let 
$R\in S$ and $\Sigma\in \overline{S}$. We associate with these objects the Riccati-type matrix  diffusion
$$
	dX_t \,=\, (AX_t + X_tA^{\prime} + R - X_t\Sigma X_t)\,dt \,+\, \frac{\epsilon}{2}\left[X_t^{1/2}d\Wa_t\,R^{1/2} + R^{1/2}\,d\Wa_t^{\prime}\,X_t^{1/2}\right].
$$
Whenever $\epsilon\leq  2/\sqrt{m+1}$, the diffusion $X_t$ starting in $S$ has a unique strong solution that never hits the boundary $\partial S=\overline{S}-S$. In addition, 
the transition semigroup $Q_t$ of $X_t$ is strongly Feller and admits a smooth positive density with respect to the Lebesgue measure on $S$; in particular, it is irreducible.
Furthermore, when 
$$\epsilon^2(1+m)/2\leq \lambda_{m}(R)/\lambda_{1}(R)$$ then the function $V(x)=\tr( x)+\tr (x^{-1})$, with the trace operator $\tr(\cdot)$, is a Lyapunov function with compact level subsets and we have $\Vert Q_h(V)\Vert<\infty$ for all time steps $h>0$. For the one-dimensional case we refer to~\cite{bishop-19}, and for a detailed proof of the above assertion for more general classes of Riccati matrix valued diffusions we refer to~\cite{bishop-19-2} (see for instance the uniform moments estimates stated in Theorem 2.2 in~\cite{bishop-19-2}).

 We conclude that, on the state space $S$, $P=Q_h$ satisfies condition \ref{H1} and, by Lemma~\ref{lemma:minor}, the local minorization condition \eqref{loc-min} is met on the compact sub-level sets of $V$. Therefore, the transition $P=Q_h$ satisfies condition \ref{H2} with $\kappa=\varphi_0$. This yields the weighted norm contraction estimates (\ref{re-3.v2}) and (\ref{res-H0}).

 Recalling that $\tr(x^2)\leq \tr(x)^2$ we check that 
$$
\Vert x\Vert_F:=\tr(x^2)^{\frac{1}{2}}\leq \tr(x)\quad \mbox{\rm and therefore}\quad
\Vert x-y\Vert_F\leq  \tr(x)+ \tr(y),
$$
and similarly for $x^{-1}$ and $y^{-1}$. This implies that $V(x)+V(y)\geq \psi(x,y)$ with the complete metric $\psi$ on $S$ defined by
$$
 \psi(x,y):=\Vert x-y\Vert_F+\Vert x^{-1}-y^{-1}\Vert_F.
$$
Part (ii) of Theorem \ref{thm:expo} now gives exponential decays for the Wasserstein $1$-distance $\WW_{\psi,1}$.

\subsection{Time-varying semigroups}\label{time-varying}
Consider a collection of Markov $V$-operators $K_n$ indexed by $n\in\NN-\{0\}$. Denote by $P_n$ the time-varying semigroup defined for all $n\geq 1$ by
$$
P_n:=P_{n-1}K_n=K_1\ldots K_n.
$$
Assume that the transitions $K_n$ satisfy conditions \ref{H1} and \ref{H2} for some parameters $(\epsilon,c,r_0,\alpha)$ that do not depend on $n$. In this case, since the constants in Theorem~\ref{th-intro-basic} depend only on $(\alpha,\epsilon,c)$, the transitions $K_n$ satisfy the strict contraction estimates stated in Theorem~\ref{th-intro-basic} uniformly in $n$. For instance, when \ref{H1} holds and \ref{H2} is met for the discrete metric $\kappa=\varphi_0$, there exist some $\rho>0$ and some contraction parameter $0<\delta<1$ such that, for all $n\geq 1$,
$$
 \beta_{\varphi_{\rho}}(K_n)\leq 
1-\delta\quad \mbox{\rm and therefore}\quad
 \beta_{\varphi_{\rho}}(P_n)\leq 
(1-\delta)^n.
$$
Following (\ref{ref-tv-contract}), we consider below the uniform minorization condition
$$
\forall x\in S\quad\forall n\geq 1\qquad
\delta_xK_n\geq \jmath\,\nu
\quad\mbox{\rm
for some $\jmath>0$ and some  $\nu\in\mathcal{P}(S)$. }
$$
In this scenario, for all $n\geq 1$ we have the uniform contraction estimate
$$
\beta_{\varphi_0}(K_n)\leq 1-\jmath.
$$

  Time varying Markov semigroups arise in a variety of areas in applied mathematics as well as in signal processing,
physics and molecular chemistry; see for instance~\cite{ado-24,dm-04,dg-01,dhj,saloff-zuniga,saloff-zuniga-2,saloff-zuniga-22,saloff-zuniga-3} and references therein. 
The stability analysis of this class of semigroups is a broad and notoriously difficult subject, 
far more challenging than their time homogeneous counterparts; this is because the Markov transitions may 
drastically change during the semigroup evolution.
The articles~\cite{douc,saloff-zuniga,saloff-zuniga-2,saloff-zuniga-22} provide several useful probabilistic and analytical  tools, including coupling methods, spectral analysis as well as functional inequalities such as Nash or log-Sobolev inequalities. 

In the following, we examine two important classes of time varying Markov transitions arising in physics, nonlinear filtering and entropic optimal transport. The first class of models is defined in terms of Feynman-Kac formulae on path spaces as well as 
nonlinear  and time varying evolution equations. The second class is defined in terms of 
Sinkhorn bridges and nonlinear conjugate transformations. For a more thorough discussion on these two classes of models we refer the reader to the book~\cite{dm-04} and the recent review article~\cite{dj-sinkhorn}.
In both situations, the application of the refined analytical and spectral techniques discussed in~\cite{douc,saloff-zuniga,saloff-zuniga-2,saloff-zuniga-22} is rather challenging. 
 The stability analysis of Feynman-Kac models and positive semigroups using quantitative Dobrushin contraction coefficients was initiated in~\cite{dg-01,dm-mic-00,dm-mic-03} and has been further developed more recently in~\cite{ado-24,dhj}. The application of these contraction techniques to Sinkhorn semigroups has recently been initiated in~\cite{dj-sinkhorn}.

\subsubsection{Feynman-Kac measures}
Consider a collection of Markov transitions $M_n$ indexed by
$n\geq 1$ and a probability measure $\eta_0$ on $S$. The
distribution of the random paths of a Markov chain with initial
distribution $\eta_0$ and Markov transitions $M_p$ on some time
horizon $n\geq 1$ is given by the probability measure
$$
\PP_n(d(x_0,\ldots,x_n)):=\eta_0(dx_0)
\prod_{0\leq p<n}M_{p+1}(x_{p},dx_{p+1}).
$$
The Feynman-Kac measure associated with a collection of bounded
positive potential functions $G_n$, $n\geq 0$, and the collection
of Markov transitions $M_n$ is defined by the path space weighted
distribution
$$
\QQ_n(d(x_0,\ldots,x_n))=\frac{1}{\Za_n}~
\PP_n(d(x_0,\ldots,x_n))~\prod_{0\leq p< n}G_p(x_p)
$$
with some normalizing constant $\Za_n$. The terminal time marginal
probability measures $\eta_n$ of $\QQ_n$ satisfy, for all
$n\geq 0$, the nonlinear equation
$$
\eta_{n+1}=\Phi_{n+1}(\eta_{n}):=\Psi_{G_n}(\eta_{n})M_{n+1}
\quad\mbox{\rm with}\quad
\Psi_{G_n}(\mu)(dx):=\frac{G_n(x)}{\mu(G_n)}~\mu(dx).
$$
This yields, for all $0\leq p<n$, the evolution equation
$$
\eta_n=\Phi_{p,n}(\eta_p)
\quad\mbox{\rm  with the composition semigroup}\quad
\Phi_{p,n}=\Phi_{n}\circ \ldots\circ\Phi_{p+1}.
$$
Consider the positive integral operators $Q_n$ defined by
$$
Q_n(x_{n-1},dx_n):=G_{n-1}(x_{n-1})M_n(x_{n-1},dx_n).
$$
For $p=n$ we set $G_{n,n}=1$ and $Q_{n,n}=Id$, the identity
integral operator, and for all $0\leq p< n$ we set
$$
\begin{array}{l}
Q_{p,n}:=Q_{p+1}Q_{p+1,n}\quad\mbox{\rm and}\quad
G_{p,n}(x_p):=Q_{p,n}(1)(x_p)\\
\\
\Longrightarrow G_{p,n}=G_p~M_{p+1}(G_{p+1,n})
\Longleftrightarrow
G_p=G_{p,n}/M_{p+1}(G_{p+1,n}).
\end{array}
$$
In this notation, we have
$$
\QQ_n(d(x_0,\ldots,x_n))=\eta^{(n)}_{0}(dx_0)
\prod_{0\leq p<n}\Ka^{(n)}_{p+1}(x_{p},dx_{p+1})
$$
with the initial distribution and the Markov transitions
$$
\eta_{0}^{(n)}(dx)=\frac{\eta_0(dx)G_{0,n}(x)}{\eta_0(G_{0,n})}
\quad\mbox{\rm and}\quad
\Ka^{(n)}_{p}(x,dy)=\frac{ M_{p}(x,dy)G_{p,n}(y)}{M_{p}(G_{p,n})(x)},
\quad 1\leq p\leq n.
$$
Thus, the probability measure $\eta_n$ can also be interpreted as
the terminal distribution of a time-varying Markov chain, that is
we have
$$
\eta_n=\Phi_{0,n}(\eta_0)
=\eta_{0}^{(n)}\Ka^{(n)}_{1}\ldots \Ka^{(n)}_{n},
$$
where we stress that the transitions $\Ka^{(n)}_{p}$ do not depend
on the initial distribution $\eta_0$.
Assume that for all $n\geq 1$ and $x\in S$ we have
$$
\jmath~\nu~\leq \delta_xM_n\leq \jmath^{-1}~\nu
$$
for some $\jmath\in]0,1]$ and some probability measure
$\nu\in\mathcal{P}(S)$. In this situation, we have for all
$x\in S$ the global minorization property
$$
\delta_x\Ka^{(n)}_{p}\geq \jmath^2~\nu^{(n)}_{p}
\quad\mbox{\rm with the probability measure}\quad
\nu^{(n)}_{p}(dy):=\frac{ \nu(dy)G_{p,n}(y)}{\nu(G_{p,n})}.
$$
Therefore, by (\ref{ref-tv-contract}), we have
$\beta_{\varphi_{0}}(\Ka^{(n)}_{p})\leq 1-\jmath^2$,
from which we check the exponential decay
$$
\begin{array}{l}
\Vert\Phi_{0,n}(\eta_0)-\Phi_{0,n}(\mu_0)\Vert_{\rm tv}\\
\quad=\Vert \eta_{0}^{(n)}\Ka^{(n)}_{1}\ldots \Ka^{(n)}_{n}
-\mu_{0}^{(n)}\Ka^{(n)}_{1}\ldots \Ka^{(n)}_{n}\Vert_{\rm tv}
\leq (1-\jmath^2)^n
\end{array}
$$
for all $\eta_0,\mu_0\in\mathcal{P}(S)$, where $\mu_{0}^{(n)}$ is
defined as $\eta_{0}^{(n)}$ with $\eta_0$ replaced by $\mu_0$.

\subsubsection{Sinkhorn semigroups}
Consider some probability measures $\mu$ and $\eta$ and some reference Markov transition $K_{0}$ on a locally compact Polish space $S$. Given $K_{2n}$ at some rank $n\geq 0$ we define 
 $K_{2n+1}$   by the conjugate formulae
\begin{equation}\label{mu-fp}
\begin{array}{l}
\mu(dx)K_{2n}(x,dy)\\
\\
\displaystyle=
(\mu K_{2n})(dy)~K_{2n+1}(y,dx)\quad \Longrightarrow \quad  \mu K_{2n}K_{2n+1}=\mu
\end{array}
\end{equation}
Equivalently, in terms of the conjugate  transition defined in (\ref{ref-duality-def}) we  have
$
K_{2n+1}=(K_{2n})^{\star}_{\mu}$.
Observe that 
$$
\PP_{2n}(d(x,y)):=\mu(dx) K_{2n}(x,dy)\Longrightarrow\PP_{2n}\in \Pi(\mu,\pi_{2n})\quad \mbox{\rm with}\quad
\pi_{2n}:=\mu K_{2n}.
$$
In the same vein, given $K_{2n+1}$ we define 
$K_{2(n+1)}$  by the conjugate formulae
\begin{equation}\label{eta-fp}
\begin{array}{l}
\eta(dy) K_{2n+1}(y,dx)\\
\\
\displaystyle=
(\eta K_{2n+1})(dx)~ K_{2(n+1)}(x,dy)\quad  \Longrightarrow \quad  \eta K_{2n+1}  K_{2(n+1)}= \eta.
\end{array}
\end{equation}
Equivalently, in terms of the conjugate  transition (\ref{ref-duality-def}) we  have
$
K_{2(n+1)}=(K_{2n+1})^{\star}_{\eta}$.

Arguing as above, observe that 
$$
\PP_{2n+1}(d(x,y)):=\eta(dy) K_{2n+1}(y,dx)\Longrightarrow\PP_{2n+1}\in \Pi(\pi_{2n+1},\eta)\quad \mbox{\rm with}\quad
\pi_{2n+1}:=\eta K_{2n+1}.
$$
The fixed point equations (\ref{mu-fp}) and (\ref{eta-fp})  yield the time varying evolution equations
 $$
 \begin{array}{l}
 \pi_{2n} \left(K_{2n+1} K_{2(n+1)}\right)= \pi_{2(n+1)}
 \quad \mbox{\rm and}\quad
\pi_{2n-1}( K_{2n}K_{2n+1} )= \pi_{2n+1} 
 \end{array}$$
 When $n\rightarrow\infty$, the marginal measures $(\pi_{2n+1},\pi_{2n})$
converge towards the target marginals $(\mu,\eta)$ and Sinkhorn bridges $\PP_{n}$ converge towards the Schr\"odinger bridge~\cite{dj-sinkhorn}. 
 
We further assume that $S$  is endowed with some locally bounded positive measure  $\nu$ and set
$$
(\mu,\eta)=(\nu_g,\nu_h)
\quad \mbox{\rm and}\quad K_0(x,dy)=M(x,dy)=\exp{(-m(x,y))}~\nu(dy)
$$
with $(\nu_g,\nu_h,m)$ as in (\ref{ref-intro-UV}) and (\ref{ref-intro-Kcchi}). We further assume that the functions $(g,h,m)$ satisfy the regularity conditions discussed in Section~\ref{lyap-sec-gibbs} for some $\delta>0$ and set 
$$(W,V):=(\exp{(\delta h)},\exp{(\delta g)}).
$$
By \cite[Lemma 5.2]{dj-sinkhorn} there exist
some constants $\epsilon\in ]0,1[$ and $c\geq 0$ such that
$$
K_{2n}(W)\leq \epsilon~V+c\quad \mbox{\rm and}\quad
K_{2n+1}(V)\leq \epsilon~W+c;
$$
note that
$
\sup_{n\geq 1}(\pi_{2n}(W)\vee
 \pi_{2n+1}(V))<\infty
$.
In addition (see again~\cite{dj-sinkhorn}), for all $r$ large enough
there exists some $\alpha(r)\in]0,1[$ such that, for all
$(x_1,x_2)\in \{V\leq r\}^2$ and $(y_1,y_2)\in \{W\leq r\}^2$,
we have
$$
\Vert \delta_{x_1} K_{2n}-\delta_{x_2} K_{2n}\Vert_{\rm tv}\vee \Vert \delta_{y_1} K_{2n+1}-\delta_{y_2} K_{2n+1}\Vert_{\rm tv}\leq 1-\alpha(r)
$$
Noting that the two functions $V$ and $W$ are l.s.c.~and bounded below away from zero, it follows from the discussion    in Section \ref{sub:two_op} (applied to the product $K_{2n}K_{2n+1}$ and, with the roles of $V$ and $W$ exchanged, to the product $K_{2n+1}K_{2(n+1)}$) that
there exist some constants $\rho>0$ and $\lambda\in ]0,1[$
such that we have the uniform contraction estimate
$$
 \beta_{\varphi_{\rho}}(K_{2n}K_{2n+1})\vee \beta_{\bar{\varphi}_{\rho}}(K_{2n+1}K_{2(n+1)})\leq \lambda,\quad\forall n\geq 1,
$$
with the semi-distance $\varphi_{\rho}$ defined in \eqref{vero} and \eqref{def-phi-rho}, and with the semi-distance $\bar{\varphi}_\rho$ defined as $\varphi_{\rho}$ with $V$ replaced by $W$. Since $\mu K_{2n}K_{2n+1}=\mu$ and $\eta K_{2n+1}K_{2(n+1)}=\eta$, by using \eqref{kr2} and the above uniform contraction estimate we obtain that
$$
\Vert \pi_{2n+1}-\mu\Vert_{V_{\rho}}\leq  \lambda ~\Vert \pi_{2n-1}-\mu\Vert_{V_{\rho}}
\quad\mbox{\rm and}\quad
\Vert \pi_{2(n+1)}-\eta\Vert_{W_{\rho}}\leq  \lambda~\Vert \pi_{2n}-\eta\Vert_{W_{\rho}},
$$
with the function $V_\rho$ defined in \eqref{vero} and the function $W_\rho$ defined as $V_\rho$ with $V$ replaced by $W$. In particular, the marginals $(\pi_{2n+1},\pi_{2n})$ converge exponentially fast to $(\mu,\eta)$ as $n\rightarrow\infty$, in the weighted total variation norms $(\Vert\cdot\Vert_{V_\rho},\Vert\cdot\Vert_{W_\rho})$.

\addtocontents{toc}{\protect\setcounter{tocdepth}{-1}}
\section*{Acknowledgments}
The authors would like to thank the Associate Editor and the three
anonymous referees for their careful reading of the manuscript and
for their insightful comments and suggestions, which led to a
substantial improvement of the presentation and of the content of
this article.

\addtocontents{toc}{\protect\setcounter{tocdepth}{2}}

\appendix
\addcontentsline{toc}{section}{Appendix}
\addtocontents{toc}{\protect\setcounter{tocdepth}{0}}

\section{Proof of the quantitative contraction estimates\label{sec:quantitative}}

\subsection{Arbitrary drift constants}

In Theorem~\ref{th1-intro} (see Section~\ref{sub:quantitative}), there is no loss of generality in assuming that  \ref{H1} is met with $c=1/2$ (and that $V\geq 1/2$).  To prove this claim, assume that \ref{H1} and   \ref{H2} are  met, let $a:=\epsilon/(2c)$ and let $(\bar r_0,\bar\alpha)$ be as in \eqref{eq:bar-alpha}. For all $r> \bar r_0\vee r_{\epsilon}$, we let
\begin{align*}
\bar\rho_\epsilon(r):=\frac{1}{1+\epsilon}~\frac{\bar\alpha(r)}{2r}\quad\mbox{\rm and}\quad
\bar{\rho}_{\epsilon,c}(r):=a\,\frac{\bar\rho_\epsilon(r)}{\bar\rho_\epsilon(r)+1}
\end{align*}
and prove the following result:
\begin{cor}\label{cor:strict}
Assume that $\inf_{x\in S}V(x)<1/2$ or that \ref{H1} holds for some $c>1/2$. Then,
\begin{itemize}
\item if \ref{H2} is  met for some $V$-semi-distance  $\kappa$,  for all $r>r_{\epsilon}\vee \bar r_0$  and $\iota\in [1/2,1[$ the inequality stated in \eqref{re-upsilon} holds with $(\rho_\epsilon(r),\alpha(r))$ replaced by $(\bar{\rho}_{\epsilon,c}(r),\bar\alpha(r))$.

 \item if \ref{H2} is met for the discrete metric $\kappa=\varphi_0$, for all $r>r_{\epsilon}\vee \bar r_0$ the inequality stated in \eqref{re-3} holds with $(\rho_\epsilon(r),\alpha(r))$ replaced by $(\bar{\rho}_{\epsilon,c}(r),\bar\alpha(r))$. Moreover, for all $V$-semi-distance  $\phi$  such that  $$\|\phi/\varphi_{1/2+aV}\|\leq (1+\bar\rho_\epsilon(r))\bar\rho_\epsilon(r)\big(1-\beta_{\varphi_{\bar{\rho}_{\epsilon,c}(r)}}(P)\big),$$ the inequality stated  in \eqref{re-3-cor} holds with $\rho_\epsilon(r)$ replaced by $\bar{\rho}_{\epsilon,c}(r)$ and with the $V$-semi-distance $\phi$ replaced by the $V$-semi-distance $\phi/(1+\bar\rho_\epsilon(r))$.
\end{itemize}
\end{cor}

\begin{proof}

If \ref{H1} holds then, for the same parameter $\epsilon\in]0,1[$, the drift condition stated in \ref{H1} also holds with   $c=1/2$ and with the function $V$ replaced by the function
\begin{align*}
\overline{V}:=1/2 +\big( \epsilon /(2c)\big)V\geq 1/2.
\end{align*}
By Lemma~\ref{lemma:barV}, $\phi$ is a $V$-semi-distance if and only if $\phi$ is a $\overline{V}$-semi-distance, and condition \ref{H2} is met when $(V,r_0,\alpha)$ is replaced by $(\bar{V},\bar r_0,\bar\alpha)$. Consequently,  the conclusions of Theorem \ref{th1-intro} hold with $(V,r_0,\alpha,\rho_\epsilon(r))$ replaced by $(\bar{V},\bar r_0,\bar\alpha,\bar\rho_\epsilon(r))$.

To proceed further, let $(\phi,\kappa)$ be two $V$-semi-distances, and let $\overline{\rho}\geq 0$ be arbitrary. Then, 
\begin{align*}
\displaystyle \overline{V}_{\overline{\rho}}:=\frac{1}{2}+\overline{\rho}~ \overline{V}=(\overline{\rho}+1)~\left(\frac{1}{2}+ \rho~V\right)=(\overline{\rho}+1)~V_{\rho}\quad \text{ with }\,\, \rho=a~\frac{\overline{\rho}}{\overline{\rho}+1}.
\end{align*}
Letting $\overline{\phi}(x,y)=(\overline{\rho}+1)\phi(x,y)$, this implies that
$$
 \varphi_{\overline{V}_{\overline{\rho}}}=(\overline{\rho}+1)~\varphi_{V_{\rho}}\quad \mbox{\rm and}\quad \phi_{\overline{V}_{\overline{\rho}}}:= (\overline{\rho}+1)\left(\overline{\phi}+\varphi_{V_{\rho}}\right)=(\overline{\rho}+1)~\overline{\phi}_{V_{\rho}} $$
 and thus
\begin{align*}
  \beta_{\varphi_{ \overline{V}_{\overline{\rho}}}}(P)=\beta_{\varphi_{V_ {\rho}}}(P)\qquad
    \beta_{\phi_{\overline{V}_{\overline{\rho}}}}(P)=    \beta_{\overline{\phi}_{V_{\rho}}}(P)
      \quad\mbox{\rm and}\quad
    \beta_{\kappa_{\iota,\overline{V}_{\overline{\rho}}}}(P)=
  \beta_{\kappa_{\iota,V_{\rho}}}(P)
\end{align*}
where the last assertion uses the fact that
$$
\begin{array}{l}
\displaystyle
\varpi_{\overline{V}_{\overline{\rho}}}(x,y)=
(\overline{\rho}+1)~(V_{\rho}(x)+V_{\rho}(y))=
(\overline{\rho}+1)~ \varpi_{V_{\rho}}(x,y)
\\
 \\
 \Longrightarrow\quad
\kappa_{\iota,\overline{V}_{\overline{\rho}}}=(\overline{\rho}+1)^{1-\iota}~\kappa^{\iota}\times \varpi_{V_{\rho}}^{1-\iota}=(\overline{\rho}+1)^{1-\iota}~\kappa_{\iota,V_{\rho}}.
\end{array}$$

Applying these identities with $\overline{\rho}=\bar\rho_\epsilon(r)$, for which $\rho=\bar\rho_{\epsilon,c}(r)$, the results of the corollary follow.
\end{proof}

\subsection{Proof of Theorem~\ref{th1-intro}\label{p-th1-intro}}
 
The contraction estimate for the  weighted discrete metric  $\varphi_{\rho}$ stated in  \eqref{re-3}  is known, but for the sake of completeness a detailed proof    is provided below in Section \ref{p-re-3}. Under the additional assumption that $V\geq 1/2$, the estimate  \eqref{re-3-cor}  is a consequence of Lemma~\ref{klem} combined with  \eqref{re-3}. To see this let $r>r_\epsilon\vee r_0$ so that $\beta:= \beta_{\varphi_{\rho_{\epsilon}(r)}}(P)\in[0,1[$ by  \eqref{re-3}, and let $\phi$ be a $V$-semi-distance  such that $\|\phi/\varphi_V\|\leq (1-\beta)\rho_\epsilon(r) $. Then,  since $V\geq 1/2$ by assumption, by  using \eqref{VVr} for the third inequality we have
$$
 \phi \leq \|\phi/\varphi_V\|\varphi_{V} \leq   (1-\beta)\rho_\epsilon(r)\, \varphi_{V}\leq   (1-\beta)\varphi_{\rho_\epsilon(r)} 
$$
and \eqref{re-3-cor}   follows by applying \eqref{Tex} with $\iota=\|\phi/\varphi_{\rho_\epsilon(r)}\|$, $\varphi=\phi/\iota$ and $\psi=\varphi_{\rho_\epsilon(r)}$.  The proof of \eqref{re-upsilon} is  presented below in Section \ref{p-re-upsilon}.

\subsubsection{Proof of \eqref{re-3}\label{p-re-3}}
 
 Let $x,y\in S$ be such that $x\neq y$ and $r>r_{\epsilon}\vee r_0$ be fixed in what follows. Then,  under the assumption that \ref{H2} holds with the semi-distance $\kappa=\varphi_0$, we have
\begin{align}\label{eq:TV}
 \Vert \delta_x P-\delta_y P\Vert_{\rm tv}\leq  1-\alpha(r)<1.
\end{align}
In what follows the functions $(\varpi_{\rho},\varphi_\rho)$ indexed by $\rho>0$ are as defined in (\ref{vero}) and (\ref{def-phi-rho}).

Assume first that $\varpi_V(x,y)\geq r$. In this scenario, for all $\rho>0$  we have
\begin{equation}\label{eq:twoParts}
\begin{split}
\ell_{\varphi_{\rho}}(P)(x,y) \leq \ell_{\varpi_{\rho}}(P)(x,y)&\leq \frac{1 +\rho P(V)(x)+\rho P(V)(y)}{1+\rho \varpi_V(x,y)}\\
&\leq \frac{1+\rho  \big(1+\epsilon   \varpi_V(x,y)\big)}{1+\rho    \varpi_V(x,y)}\\
&\leq  \epsilon+\frac{1+\rho-\epsilon}{1+\rho r}= 1-(1-\epsilon)~ \frac{\rho r}{1+\rho r}~\left(1-\frac{r_{\epsilon}}{r}\right)
\end{split} 
\end{equation}
where the first inequality holds since $\varphi_\rho\leq\varpi_\rho$ with equality outside the diagonal, the second one holds since $\pi(\varpi_\rho)=1+\rho P(V)(x)+\rho P(V)(y)$ for all $\pi\in\Pi(\delta_xP,\delta_yP)$, the third one uses the assumption that \ref{H1} holds with $c=1/2$, and the last one uses the assumption $\varpi_V(x,y)\geq r$. Simple calculations show that 
\begin{align*}
(1-\epsilon)\frac{r\rho_{\epsilon}(r)}{1+r\rho_{\epsilon}(r)}\left(1-\frac{r_{\epsilon}}{r}\right)&=~\frac{\alpha(r)}{2}~\frac{1-\epsilon}{(1+\epsilon)+\alpha(r)/2}~\left(1-\frac{r_{\epsilon}}{r}\right) \geq~\frac{\alpha(r)}{2}~\delta_{\epsilon}(r)
\end{align*}
and thus we have the implication
\begin{align}\label{eq:kappa1}
\varpi_V(x,y)\geq r\implies \ell_{\varphi_{\rho_\epsilon(r)}}(P)(x,y)\leq \ell_{\varpi_{\rho_\epsilon(r)}}(P)(x,y)\leq 1-\delta_\epsilon(r)\frac{\alpha(r)}{2}.
\end{align}
Conversely, if $\varpi_V(x,y)\leq r$ then, letting $\pi\in\Pi(\delta_xP,\delta_yP)$ be an optimal coupling for the total variation distance, so that $\pi(\varphi_0)\leq 1-\alpha(r)$ by \eqref{eq:TV}, and following similar calculations as in \eqref{eq:twoParts},
\begin{align*}
 \ell_{\varphi_{\rho}}(P)(x,y)& \leq \frac{1-\alpha(r)+\rho  \big(1+\epsilon   \varpi_V(x,y)\big)}{1+\rho    \varpi_V(x,y)}\leq 1-\alpha(r)+\rho+r\epsilon\rho,\quad\forall \rho>0 
\end{align*}
from which we readily obtain that 
\begin{align}\label{eq:kappa2}
\varpi_V(x,y)\leq r\implies  \ell_{\varphi_{\rho_\epsilon(r)}}(P)(x,y)\leq 1-\frac{\alpha(r)}{2}\leq 1-\delta_\epsilon(r)\frac{\alpha(r)}{2}.
\end{align}
Combining \eqref{eq:kappa1} and \eqref{eq:kappa2} shows that
\begin{align*}
\ell_{\varphi_{\rho_\epsilon(r)}}(P)(x,y) \leq 1-\delta_\epsilon(r)\frac{\alpha(r)}{2}
\end{align*}
and the result in \eqref{re-3} follows from Lemma~\ref{lemma:ced}.\cqfd

\subsubsection{Proof of \eqref{re-upsilon}\label{p-re-upsilon}}

In what follows the functions $(\varpi_{\rho},\varphi_\rho)$ indexed by $\rho>0$ are as defined in (\ref{vero}) and (\ref{def-phi-rho}), and we extend the  definition of the function $\ell_{\phi}(P):(x,y)\in S^2\rightarrow [0,\infty]$ defined in \eqref{eq:ell} for a $V$-semi-distance $\phi$ to any l.s.c.~function $\phi: S^2\rightarrow [0,\infty)$.

We start with some preliminary calculations. We let $\rho>0$,  $\iota\in]0,1[$ and   $x,y\in S$ be such that $x\neq y$. Next,   we let  $\pi_{x,y}\in \Pi(\delta_{x}P,\delta_{y}P)$ be such that   $\pi_{x,y}(\kappa)=D_{\kappa}(\delta_{x}P ,\delta_{y}P)$, so that  $\ell_{\kappa}(P)(x,y)=\pi_{x,y}(\kappa)/\kappa(x,y)$. In addition,   noting  that
\begin{align}\label{eq:eq}
\pi_{x,y}\left(\varpi_\rho\right)=1+\rho\,P(V)(x)+\rho\,P(V)(y)=D_{\varpi_\rho}(\delta_{x}P,\delta_{y}P)
\end{align}
it follows that
$$
\ell_{\varpi_\rho}(P)(x,y)=\frac{\pi_{x,y}\left(\varpi_\rho\right)}{\varpi_\rho(x,y)}.
$$
Using the above expressions for $\ell_{\kappa}(P)(x,y)$ and  for $\ell_{\varpi_\rho}(P)(x,y)$, as well as H\"older's inequality, we obtain that
\begin{equation}\label{eq:zeta}
\begin{split}
\ell_{\kappa_{\iota,\rho}}(P)(x,y)\leq \frac{\pi_{x,y}(\kappa_{\iota,\rho})}{\varpi_{\rho}(x,y)^{1-\iota}\kappa(x,y)^{\iota}}&\leq 
\ell_{\varpi_\rho}(P)(x,y)^{1-\iota}~\ell_{\kappa}(P)(x,y)^{\iota}\\
&\leq \ell_{\varpi_\rho}(P)(x,y)^{1-\iota} 
\end{split}
\end{equation}
where the last inequality uses the fact that $\beta_\kappa(P)\leq 1$ under \ref{H2} (since any pair $(x,y)\in S^2$ satisfies $\varpi_V(x,y)\leq r$ for $r$ large enough).

To prove \eqref{re-upsilon} let $x,y\in S$ be such that $x\neq y$ and  $r>r_{\epsilon}\vee r_0$ be fixed in what follows. Then,  
\begin{equation}\label{eq:e1}
\begin{split}
 \varpi_V(x,y)\geq  r\Longrightarrow 
\ell_{\kappa_{\iota,\rho_\epsilon(r)}}(P)(x,y)\leq \ell_{\varpi_{\rho_\epsilon(r)}}(P)(x,y)^{1-\iota}\leq  \displaystyle
\Big( 1-\delta_\epsilon(r)\frac{\alpha(r)}{2}\Big)^{1-\iota}
\end{split}
\end{equation}
where the first  inequality holds by \eqref{eq:zeta}  while the second inequality holds by \eqref{eq:kappa1}.

Conversely, if $\varpi_V(x,y)\leq  r$ we have, under the assumption that \ref{H1} holds with $c=1/2$,
$$
\ell_{\varpi_{\rho_\epsilon(r)}}(P)(x,y)\leq  1+\rho_\epsilon(r) (1+\epsilon\,r) \leq  1+\alpha(r) 
$$
where the second inequality holds since $\rho_\epsilon(r)\leq  \alpha(r)/(1+\epsilon r)$. Consequently,
\begin{equation}\label{eq:e2}
\begin{split}
\varpi_V(x,y)\leq  r\Longrightarrow \ell_{\kappa_{\iota,\rho_\epsilon(r)}}(P)(x,y)&\leq \ell_{\kappa}(P)(x,y)^{\iota}~\ell_{\varpi_{\rho_\epsilon(r)}}(P)(x,y)^{1-\iota} \\
&\leq (1-\alpha(r))^{\iota}~(1+\alpha(r))^{1-\iota}\leq (1-\alpha(r)^2)^{1-\iota}
\end{split}
\end{equation}
where the first  inequality holds by \eqref{eq:zeta}  and
where the last inequality comes from the fact that, for all $1/2\leq \iota< 1$, we have
$$
(1-\alpha(r))^{\iota}~(1+\alpha(r))^{1-\iota}=
(1-\alpha(r))^{\iota-(1-\iota)}~(1-\alpha(r)^2)^{1-\iota}\leq (1-\alpha(r)^2)^{1-\iota}.
$$
Combining \eqref{eq:e1} and \eqref{eq:e2} shows that
$$
\ell_{\kappa_{\iota,\rho_\epsilon(r)}}(P)(x,y)\leq \left(1-\alpha(r)\left(\alpha(r) \wedge  \frac{\delta_\epsilon(r)}{2}\right)\right)^{1-\iota} 
$$
and the result in \eqref{re-upsilon} follows from Lemma~\ref{lemma:ced}.\cqfd

\section{Additional Proofs\label{sec:proof}}

\subsection{Proof of  Lemma \ref{lemma:ced}}\label{ced-proof}
 
\begin{proof}
We clearly have the estimate
$$
\beta_{\psi,\phi}(P)\geq \ell_{\psi,\phi}^{\star}(P):=\sup_{(x,y)\in S^2~:~\psi(x,y)>0}\ell_{\psi,\phi}(P)(x,y)
$$
and thus, recalling that since $\psi$ is a semi-distance we have $\psi(x,y)>0$ if and only if $x\neq y$, to prove the lemma it remains to show that
\begin{align}\label{eq:toShow}
\beta_{\psi,\phi}(P)\leq \ell_{\psi,\phi}^{\star}(P).
\end{align}

To this aim, following \cite{Villani}, page 67, we remark that since  $\phi$ is l.s.c.~we have $\phi(x,y)=\lim_{k\rightarrow\infty}\phi_k(x,y)$ for all $x,y\in S$,  where for all $k\geq 1$ the function $\phi_k$ is defined by
\begin{align}\label{eq:psi_k}
\phi_k(x,y)=\inf_{x',y'\in S}\Big\{\min\Big(\phi(x',y'),k\Big)+k\Big(\psi_S(x,x')+\psi_S(y,y')\Big)\Big\},\quad x,y\in S.
\end{align}
It is easy to verify that  $(\phi_k)_{k\geq 1}$ is   a non-decreasing sequence of bounded and   continuous functions such that $0\leq \phi_k(x,y)\leq\min(\phi(x,y),k)$ for all $k\geq 1$ and all $x,y\in S$. 

As a first step towards proving \eqref{eq:toShow} we show that
\begin{align}\label{eq:D_lim}
D_\phi(\mu_1,\mu_2)= \limsup_{k\rightarrow\infty}D_{\phi_k}(\mu_1,\mu_2),\quad\forall\mu_1,\mu_2\in\mathcal{P}_V(S).
\end{align}
To do so we let $\mu_1,\mu_2\in\mathcal{P}_V(S)$ be arbitrary and   remark that the set $\Pi(\mu_1,\mu_2)$ is compact  w.r.t.~the topology of the weak convergence; this follows from Lemma 4.4 in \cite{Villani} and Prokhorov's theorem. Next,   for all $k\geq 1$ we let $\pi_k\in\Pi(\mu_1,\mu_2)$ be such that $D_{\phi_k}(\mu_1,\mu_2)=\pi_k(\phi_k)$; remark that such a measure $\pi_k$ exists since $\phi_k$ is a $V$-semi-distance   and $S$ is Polish. Since the set $\Pi(\mu_1,\mu_2)$ is compact, the  sequence $(\pi_k)_{k\geq 1}$ admits a subsequence $(\pi_{k_i})_{i\geq 1}$ that converges weakly to some measure $\tilde{\pi}\in\Pi(\mu_1,\mu_2)$. Then, noting that $\phi_l\leq \phi_k$ for all $k\geq l\geq 1$, and using the fact that the functions $(\phi_l)_{l\geq 1}$ are continuous and bounded, for all fixed $l\geq 1$ we have
\begin{align}\label{eq:limsup}
\tilde{\pi}(\phi_l)=\lim_{i\rightarrow\infty}  \pi_{k_i}(\phi_l)\leq \limsup_{i\rightarrow\infty}\pi_{k_i}(\phi_{k_i})= \limsup_{i\rightarrow\infty}D_{\phi_{k_i}}(\mu_1,\mu_2)\leq  \limsup_{k\rightarrow\infty}D_{\phi_k}(\mu_1,\mu_2).
\end{align}
On the other hand, we have
\begin{align}\label{eq:D_up}
 \limsup_{k\rightarrow\infty}D_{\phi_k}(\mu_1,\mu_2)\leq D_\phi(\mu_1,\mu_2)\leq   \tilde{\pi}(\phi)=\lim_{l\rightarrow\infty}\tilde{\pi}(\phi_l)
\end{align}
where the first inequality holds since $\phi_k\leq \phi$ for all $k\geq 1$, the second inequality holds since $\tilde{\pi}\in\Pi(\mu_1,\mu_2)$, while the equality holds by the monotone convergence theorem. Then, combining \eqref{eq:limsup} and \eqref{eq:D_up} shows \eqref{eq:D_lim}.

We now let $k\geq 1$ be arbitrary. Then, since the $V$-semi-distance $\phi_k$ is continuous it follows from Corollary 5.22 in \cite{Villani} that there exists a Markov transition $Q_k$ on $S^2$ such that, for all $x,y\in S$, the measure $\pi^{(k)}_{x,y}:=Q_k( (x,y),d(u,v))$ belongs to $\Pi(\delta_x P,\delta_y P)$ and satisfies $\pi^{(k)}_{x,y}(\phi_k)=D_{\phi_k}(\delta_x P,\delta_y P)$.

We now show that under the assumption of the lemma we have
\begin{equation}\label{pced2}
 D_{\phi_k}(\delta_{x}P,\delta_{y}P)= \int~\phi_k(u,v)Q_k((x,y),d(u,v))\leq~\ell_{\psi,\phi}^{\star}(P)  \psi(x,y),\quad\forall x,y\in S.
\end{equation}
Since $\phi_k\leq \phi$, the inequality in \eqref{pced2} clearly holds for all $x,y\in S$ such that  $\psi(x,y)>0$. Let $x,y\in S$ be such that $\psi(x,y)=0$. Then, since $\psi$ is a semi-distance, we have $x=y$. On the other hand, since $\phi$ is also a semi-distance,  it follows from \eqref{eq:psi_k} that $\phi_k$ is  a semi-distance, and thus we have the implication  $x=y\implies  D_{\phi_k}(\delta_{x}P,\delta_{y}P)=0$. This shows that $\psi(x,y)=0\implies D_{\phi_k}(\delta_{x}P,\delta_{y}P)=0$, and thus the inequality in \eqref{pced2} also holds when $\psi(x,y)=0$. This concludes the proof of \eqref{pced2}.

To proceed further we remark that since $\psi$ is l.s.c.~there exists a measure $\pi\in \Pi(\mu_1,\mu_2)$ such that $D_{\psi}(\mu_1,\mu_2)=\pi(\psi)$, and we also remark that 
\begin{align}\label{eq:piQ}
(\pi Q_k)(d(u,v)):=\int~\pi(d(x,y))Q_k((x,y),d(u,v))\in 
\Pi(\mu_1P,\mu_2P).
\end{align}
Consequently, by using first \eqref{eq:piQ} and then by integrating both sides of (\ref{pced2}) using the measure $\pi$ we have just defined, we obtain that
\begin{align}\label{eq:piQ3}
D_{\phi_k}(\mu_1P,\mu_2P)&\leq    (\pi Q_k)(\phi_k) \leq \ell_{\psi,\phi}^{\star}(P) \pi(\psi) =\ell_{\psi,\phi}^{\star}(P)D_{\psi}(\mu_1,\mu_2)
\end{align}
showing that if $D_{\psi}(\mu_1,\mu_2)>0$ then
\begin{align*}
\frac{D_{\phi_k}(\mu_1P,\mu_2P)}{D_{\psi}(\mu_1,\mu_2)}\leq \ell_{\psi,\phi}^{\star}(P).
\end{align*}
Together with \eqref{eq:D_lim}, this shows \eqref{eq:toShow} and the proof of the lemma is complete.

\end{proof}

\begin{rmk}
Under the additional assumption that $\phi$ is continuous, by Corollary 5.22 in \cite{Villani} there exists a Markov transition $Q$ on $S^2$ such that
\begin{equation}\label{pced3}
D_{\phi}(\delta_{x}P,\delta_{y}P)=~\int~\phi(u,v)~Q((x,y),d(u,v)),\quad\forall x,y\in S.
\end{equation}
In this case, \eqref{eq:toShow} directly follows by applying \eqref{eq:piQ3} with $Q_k$ replaced by $Q$. Assuming only that $\phi$ is l.s.c.,  \eqref{eq:D_lim} implies that the mapping $(x,y)\mapsto D_{\phi}(\delta_{x}P,\delta_{y}P) $ is measurable but it is unclear whether or not there exists  a Markov kernel $Q$ such that \eqref{pced3} holds.
\end{rmk}

\subsection{Proof of Lemma~\ref{klem}}\label{klem-proof}
 
The first statement in \eqref{SJ} is trivial. To prove the second statement in \eqref{SJ} let $x,y\in S$  be such that $\varphi(x,y)>0$ and $\pi_{x,y}\in\Pi(\delta_xP,\delta_yP)$ be such that $D_{\varphi}(\delta_{x}P,\delta_{y}P)=\pi_{x,y}(\varphi)$. Then, by Lemma \ref{lemma:ced}, we have
$$
D_{\varphi}(\delta_{x}P,\delta_{y}P)=\pi_{x,y}(\varphi)\leq \beta_{\varphi}(P)~
\varphi(x,y).
$$
By Jensen's inequality, for all $\iota\in]0,1[$  we have
$$
D_{\varphi^{\iota}}(\delta_{x}P,\delta_{y}P)\leq \pi_{x,y}(\varphi^{\iota})\leq 
\pi_{x,y}(\varphi)^{\iota}=D_{\varphi}(\delta_{x}P,\delta_{y}P)^{\iota}
\leq \beta_{\varphi}(P)^{\iota}~
\varphi(x,y)^{\iota}
$$
implying that $$D_{\varphi^{\iota}}(\delta_{x}P,\delta_{y}P)\leq \beta_{\varphi}(P)^{\iota}~
\varphi(x,y)^{\iota}$$ which, by Lemma \ref{lemma:ced}, implies that we have  $\beta_{\varphi^{\iota}}(P)\leq 
\beta_{\varphi}(P)^{\iota}$. This ends the proof of  \eqref{SJ}.

 To show \eqref{compCost} let   $\psi$ be a $V$-semi-distance    such that  $a\,\varphi \leq \psi\leq b\,\varphi$ for some constants $0<a\leq b$. When $\beta_{\psi}(P)=\infty$ the estimate  \eqref{compCost} is immediate. When
$\beta_{\psi}(P)<\infty$ for all $\mu_1,\mu_2\in\mathcal{P}_V(S)$ we have
$$
a\,D_{\varphi }(\mu_1P,\mu_2P)\leq
D_{\psi}(\mu_1P,\mu_2P)\leq 
\beta_{\psi}(P)~D_{\psi}(\mu_1,\mu_2)\leq 
b\,\beta_{\psi}(P)~D_{\varphi}(\mu_1,\mu_2)
$$
which implies that $$D_{\varphi}(\mu_1P,\mu_2P)\leq
(b/a)\,\beta_{\psi}(P)~D_{\varphi}(\mu_1,\mu_2).$$ The result in \eqref{compCost} follows.
 
  To show \eqref{pre-Tex} we let   $\psi$ be a $V$-semi-distance  such that $\varphi \leq \psi$ and such that $\beta_{\psi}(P)<1$. Next,  we let $x,y\in S$ be such that $\psi(x,y)>0$ and $\pi_{x,y}\in\Pi(\delta_xP,\delta_yP)$ be such that $$D_{\psi}(\delta_{x}P,\delta_{y}P)=\pi_{x,y}(\psi).$$ Then,  
 \begin{align}\label{eq:nu_12}
D_{\varphi }(\delta_{x}P,\delta_{y}P)\leq
\pi_{x,y}(\varphi )
\leq \pi_{x,y}(\psi)=
D_{\psi}(\delta_{x}P,\delta_{y}P)\leq \beta_{\psi}(P)\,\psi(x,y)
\end{align}
where the last inequality holds by Lemma \ref{lemma:ced}. Using \eqref{eq:nu_12},  it follows that for all $\iota>0$ we have
$$
\frac{D_{\iota\varphi +\psi}(\delta_{x}P,\delta_{y}P)}{\iota\varphi(x,y)+\psi(x,y)}\leq
\frac{\pi_{x,y}(\iota\varphi +\psi)}{\iota\varphi(x,y)+\psi(x,y)}\leq (1+\iota)~
 \beta_{\psi}(P)~\frac{\psi(x,y)}{\iota\varphi(x,y)+\psi(x,y)} 
$$
 and thus, for all $\delta\in ]0,1[$, by Lemma \ref{lemma:ced} we have the implication:
\begin{align}\label{pre-pre-Tex}
\beta_{\psi}(P)<\frac{\delta}{1+\iota}\Longrightarrow
\beta_{\iota\varphi+\psi}(P)\leq 
(1+\iota)~
 \beta_{\psi}(P)\leq \delta 
\end{align}
showing \eqref{pre-Tex}.

Lastly, noting that   for $\iota\leq  1-\beta_{\psi}(P) $ we have $(1+\iota)\beta_{\psi}(P)\leq  1-(1-\beta_{\psi}(P))^2 $, it follows from  \eqref{pre-pre-Tex} that
\begin{eqnarray*}
\beta_{\iota\varphi +\psi}(P)\leq  1-(1-\beta_{\psi}(P))^2
\end{eqnarray*}
showing \eqref{Tex}. This ends the proof of the lemma.\cqfd

\subsection{Proof of  Lemma \ref{lemma:minor}}\label{lemminor-proof}

For all $r\geq 1$ let $A_r=\{x\in S: V(x)\leq r\}$. Then, since $V\in \Ba_{\infty}(S)$ the set $A_r$ is compact for all $r\geq 1$ and we have $S=\cup_{n\in\mathbb{N}}A_n$. In this context, there exists an $n'\in\mathbb{N}$ such that $\eta(A_{n'})>0$, since otherwise $\eta(S)=0$. In addition, since the density $p$ is assumed to be strictly positive and continuous,  for all $r\geq n'$ we have $$a_r:=\inf_{(x,y)\in A_r\times A_{n'}}p(x,y)>0.$$ Therefore, for all $r\geq n'$,
\begin{align*}
P(x, dy)\geq P(x, dy)1_{A_{n'}}(y) \geq a_r\, \eta(d y)1_{A_{n'}}(y)=b_r~ q(d y),\quad\forall x\in A_r
\end{align*}
with $b_r=a_r~\eta(A_{n'})>0$ and $q(dy)=\eta(d y)1_{A_{n'}}(y)/\eta(A_{n'})$. This shows that \eqref{loc-min} holds with $r_0=n'$, $\alpha(r)=(1/2)\wedge b_r$ and $\nu_r=q$. The proof of the lemma is complete.\cqfd

\subsection{Proof of   \eqref{cond-intro-p-norms} and  \eqref{cond-intro-exp}}\label{cond-intro-exp-proof}

For functions of the form $V(x)=1/2+\Vert x\Vert^p$ for some $p>0$ we have
\begin{equation}\label{p2w}
\frac{1}{2^{(p-1)_+}}~\Vert x-y\Vert^p\leq \Vert x\Vert^p+\Vert y\Vert^p\leq \varpi_V(x,y)\end{equation}
This ends the proof of   (\ref{cond-intro-p-norms}). For
functions of the form
$V(x)=e^{\delta\Vert x\Vert}$, using the triangle inequality and the inequality $2\sqrt{uv}\leq u+v$, we have
\begin{equation}\label{ref-V-pw}
\varpi_{e}(x,y)\leq2\left(e^{\delta(\Vert x\Vert+\Vert y\Vert)/2}-1\right)\leq  2~e^{\delta\Vert x\Vert/2}~e^{\delta\Vert y\Vert/2}\leq \varpi_V(x,y).
\end{equation}
On the other hand, for all integers $p\geq 1$, using the inequality $e^s-1\geq s^p/p!$, valid for all $s\geq 0$, we have
$$
\varpi_e(x,y)\geq \frac{2}{p!}\left(\frac{\delta\Vert x-y\Vert}{2}\right)^p=\frac{\delta^p}{2^{p-1}p!}\,\Vert x-y\Vert^p.
$$
This ends the proof of (\ref{cond-intro-exp}).\cqfd

\subsection{Proof of (\ref{RFI-arcsin}) and (\ref{RFI-0-2})}\label{RFI-arcsin-proof}

Let $V=W_\iota$. Note that $V(x)\geq V(1/2)=2^{1+\iota}>0$ and that $V\in \Ba_\infty(S)$.
For all $\iota\in ]0,1/2[$ there exists some finite constant $c>0$ such that
$$
P(V)(x)\leq \epsilon_{\iota}~V(x)+c\quad\mbox{\rm with}\quad
\epsilon_{\iota}:=\frac{1}{2(1-\iota)}<1
$$
We check this claim using the formula
\begin{eqnarray*}
2(1-\iota)~P(V)(x)&=&\frac{1}{x}\left(x^{1-\iota}+\left(1-(1-x)^{1-\iota}\right)\right)+\frac{1}{1-x}\left((1-x)^{1-\iota}+\left(1-x^{1-\iota}\right)\right)\\
&=&V(x)+f(x)+f(1-x)
\end{eqnarray*}
with the uniformly bounded function
$$
\begin{array}{l}
\displaystyle f(x):=\frac{1}{x}\left(1-(1-x)^{1-\iota}\right)=\frac{1}{(1-x)^{\iota}}+
\frac{1}{x}~\left(1-\frac{1}{(1-x)^{\iota}}\right)\leq \frac{1}{(1-x)^{\iota}}\wedge 
\frac{1}{x}
\\
\\
\displaystyle\Longrightarrow\forall \delta\in ]0,1[\qquad
0\leq f(x)\leq \frac{1}{\delta^{\iota}}~ 1_{
\delta\leq 1-x<1}+\frac{1}{1-\delta}~
1_{1-\delta<x<1}
\end{array}
$$
This ends the proof of (\ref{RFI-arcsin}).

To prove (\ref{RFI-0-2}), 
consider the parameters $\iota\in ]0,1[$ such that
$$
 0<
\frac{\iota}{2}<\frac{1-\delta}{1+(1-\delta)}~\left(\leq \frac{1}{2}\right),
$$
Observe that
$$
\frac{1}{x}~\int_{0}^x
~\left(\frac{1}{y^{\iota}}+y\right)~dy=\frac{1}{1-\iota}~\frac{1}{x^{\iota}}+\frac{x}{2}
$$
We also check that
$$
\int_x^{\infty}\left(\frac{1}{y^{\iota}}+y\right)\sqrt{\frac{2}{\pi}}~e^{-(y-x)^2/2}dy\leq 
\frac{1}{x^{\iota}}+x+c_0\quad\mbox{\rm with}\quad c_0:=\int_0^{\infty}z\sqrt{\frac{2}{\pi}}~e^{-z^2/2}dz
$$
and thus
\begin{eqnarray*}
P(V)(x)&\leq& \frac{\delta}{2}\left(\frac{1}{1-\iota}~\frac{1}{x^{\iota}}+\frac{x}{2}
+\frac{1}{x^{\iota}}+x+c_0\right)+(1-\delta)~\nu_\gamma(V)\\
&\leq &\frac{\delta}{2}~\left(1+\frac{1}{1-\iota}\right)~\frac{1}{x^{\iota}}+
 \frac{\delta}{2}~\left(1+\frac{1}{2}\right)~x+c\quad \mbox{\rm with}\quad
 c:=\frac{\delta}{2}~c_0+(1-\delta)~\nu_\gamma(V)\\
 &\leq &\epsilon~V(x)+c\quad
 \mbox{\rm with}\quad \epsilon:=\frac{\delta}{2}~\left(1+\frac{1}{1-\iota}\right)<1
\end{eqnarray*}
where $\nu_\gamma(V)<\infty$ since $\iota<1$. This ends the proof of (\ref{RFI-0-2}).\cqfd

\subsection{Proof of Lemma~\ref{lem:Utilde}}\label{app:Utilde}

Let $\theta$ be the function on $[0,\infty[$
such that $\theta(0)=0$ and
$$
\partial\theta(u)=\min\left(\left(c+\frac{\lambda}{2}\right)u~,~
\left((c+\lambda)\,\overline{r}-\frac{\lambda}{2}\,u\right)_+\right),
\quad u\geq 0,
$$
where $\partial\theta$ and $\partial^2\theta$ denote the first and
second derivatives of $\theta$.
The function $\partial\theta$ is continuous and piecewise
affine: it has slope $c+\lambda/2$ on $[0,\overline{r}]$,
slope $-\lambda/2$ on $[\overline{r},R_\star]$,
with $$R_\star:=2\overline{r}(1+c/\lambda),$$ and vanishes on
$[R_\star,\infty[$. Hence the function
$\overline{U}(x):=\theta(\Vert x\Vert)$ has a Lipschitz gradient,
and its Hessian, which exists almost everywhere, has the radial
eigenvalue $\partial^2\theta(\Vert x\Vert)$ and the tangential eigenvalues
$\partial\theta(\Vert x\Vert)/\Vert x\Vert\geq 0$. Therefore,
$$
\nabla^2\overline{U}(x)\geq \left(c+\frac{\lambda}{2}\right)I
\quad\mbox{\rm if}\quad \Vert x\Vert<\overline{r}
\quad\mbox{\rm and}\quad
\nabla^2\overline{U}(x)\geq -\frac{\lambda}{2}~I
\quad\mbox{\rm for all $x\in S$}.
$$
We set $\widetilde{U}:=\overline{U}\star\varrho$, where $\varrho$
is a smooth probability density supported in the unit ball.
Since $\nabla\overline{U}$ is Lipschitz, the function
$\overline{U}$ has bounded second order weak derivatives, which
coincide almost everywhere with its Hessian
$\nabla^2\overline{U}$. Differentiation commutes with
convolution, so that $\widetilde{U}$ is smooth and
$$
\nabla^2\widetilde{U}(x)=\int\nabla^2\overline{U}(x-y)~
\varrho(y)~dy,\quad\forall x\in S,
$$
where the integrand is well defined for almost every $y$.
Hence, for all $x\in S$ and any unit vector $v\in\RR^m$, we have
$$
v^{\top}\nabla^2\widetilde{U}(x)\,v=\int
v^{\top}\nabla^2\overline{U}(x-y)\,v~\varrho(y)~dy.
$$
Since $v^{\top}\nabla^2\overline{U}(z)\,v\geq -\lambda/2$ for
almost every $z\in S$ and $\int\varrho(y)\,dy=1$, we obtain
$\nabla^2\widetilde{U}(x)\geq -(\lambda/2)~I$ for all $x\in S$.
In addition, if $\Vert x\Vert<r$, then for all $y$ in the support
of $\varrho$ we have $\Vert y\Vert\leq 1$, and thus
$\Vert x-y\Vert\leq \Vert x\Vert+\Vert y\Vert<r+1=\overline{r}$,
so that $v^{\top}\nabla^2\overline{U}(x-y)\,v\geq c+\lambda/2$
for almost every such $y$. Integrating with respect to $\varrho$,
we conclude that
$$
\nabla^2\widetilde{U}(x)\geq \left(c+\frac{\lambda}{2}\right)I
\quad\mbox{\rm if}\quad\Vert x\Vert<r,
\quad\mbox{\rm and}\quad
\nabla^2\widetilde{U}(x)\geq -\frac{\lambda}{2}~I
\quad\mbox{\rm for all $x\in S$}.
$$
Combined with (\ref{ex-lg-cg-v2}), this shows that
$\nabla^2(U+\widetilde{U})\geq (\lambda/2)~I$.

Moreover, for all $z\neq 0$ we have
$\nabla\overline{U}(z)=\partial\theta(\Vert z\Vert)~z/\Vert z\Vert$,
and $\nabla\overline{U}(0)=0$. Since $\partial\theta\geq 0$, this yields
$\Vert\nabla\overline{U}(z)\Vert=\partial\theta(\Vert z\Vert)$ for all
$z\in S$. The function $\partial\theta$ is non-decreasing on
$[0,\overline{r}]$, non-increasing on $[\overline{r},\infty[$,
and vanishes on $[R_\star,\infty[$, so that
$$
\sup_{z\in S}\Vert\nabla\overline{U}(z)\Vert=
\partial\theta(\overline{r})=\left(c+\frac{\lambda}{2}\right)\overline{r}
\quad\mbox{\rm and}\quad
\nabla\overline{U}(z)=0\quad\mbox{\rm if}\quad
\Vert z\Vert\geq R_\star.
$$
Since
$\nabla\widetilde{U}(x)=\int\nabla\overline{U}(x-y)~\varrho(y)~dy$,
we have
$\Vert\nabla\widetilde{U}(x)\Vert\leq\int
\Vert\nabla\overline{U}(x-y)\Vert~\varrho(y)~dy\leq
\partial\theta(\overline{r})$ for all $x\in S$. In addition, if
$\Vert x\Vert\geq R_\star+1$, then
$\Vert x-y\Vert\geq\Vert x\Vert-\Vert y\Vert\geq R_\star$ for all
$y$ in the support of $\varrho$, and thus
$\nabla\widetilde{U}(x)=0$.

On the other hand, since $\partial\theta\geq 0$ and $\partial\theta=0$ on
$[R_\star,\infty[$, the function $\theta$ is non-decreasing and
constant on $[R_\star,\infty[$. Thus
$0=\theta(0)\leq\overline{U}(z)\leq\theta(R_\star)$ for all
$z\in S$, and, since $\varrho$ is a probability density,
$0\leq\widetilde{U}(x)\leq\theta(R_\star)$ for all $x\in S$.
Finally, $R_\star-\overline{r}=\overline{r}(2c+\lambda)/\lambda$
and $\partial\theta$ decreases linearly from
$\partial\theta(\overline{r})=(c+\lambda/2)\overline{r}$ to
$\partial\theta(R_\star)=0$ on $[\overline{r},R_\star]$, so that
$$
\theta(R_\star)=\int_0^{\overline{r}}\left(c+\frac{\lambda}{2}
\right)u~du+\frac{1}{2}\left(c+\frac{\lambda}{2}\right)
\overline{r}~(R_\star-\overline{r})
=\frac{\overline{r}^{\,2}\,(2c+\lambda)(c+\lambda)}{2\lambda}.
$$
This ends the proof of the lemma.\cqfd


\begin{thebibliography}{99}

\bibitem{anari}
N. Anari, K. Liu, S. O. Gharan, C. Vinzant. Log-concave polynomials II: high-dimensional walks and an FPRAS for counting bases of a matroid. In Proceedings of the 51st Annual ACM SIGACT Symposium on Theory of Computing (STOC), pp. 1--12 (2019).


\bibitem{adpm-25}
O.D. Akyildiz, P. Del Moral, J. Miguez. On the contraction properties of Sinkhorn semigroups. Foundations of Data Science, vol. 11, pp. 150--206 (2026). doi: 10.3934/fods.2026006. ArXiv preprint arXiv:2503.09887.

\bibitem{adpm-24}
O.D. Akyildiz, P. Del Moral, J. Miguez. Gaussian entropic optimal transport: Schr\"odinger bridges and the Sinkhorn algorithm. Foundations of Data Science, vol. 11, pp. 50--120 (2026). doi: 10.3934/fods.2026004. ArXiv preprint arXiv:2412.18432.
 
 \bibitem{aronson}
 D. G. Aronson, Bounds for the fundamental solution of a parabolic equation, Bulletin of American Math. Soc., vol. 73, pp. 890--896 (1967).

\bibitem{ado-24} 
M. Arnaudon, P. Del Moral, E.  Ouhabaz. A Lyapunov approach to stability of positive semigroups: An overview with illustrations. Stochastic Analysis and Applications, vol. 42, no. 1, pp. 121--200 (2024).


\bibitem{bakry}
D. Bakry, I. Gentil, M. Ledoux, Analysis and geometry of Markov diffusion operators, Grundlehren der Mathematischen Wissenschaften [Fundamental Principles of Mathematical Sciences], vol. 348, Springer, Cham (2014).


\bibitem{bao}
J. Bao, M. B. Majka, J. Wang. Geometric ergodicity of modified Euler schemes for SDEs with super-linearity.  ArXiv preprint arXiv:2412.19377 (2024).

\bibitem{bao-hao}
J. Bao, J. Hao. $L^2$-Wasserstein contraction of modified Euler schemes for SDEs with high diffusivity and applications. ArXiv:2411.01731 (2024).


\bibitem{bishop-19}
A.N. Bishop, P. Del Moral, K. Kamatani and B. Remillard.
On one-dimensional Riccati diffusions.   Annals of Applied Probability. 
volume 29, no. 2, pp. 1127--1187 (2019).

\bibitem{bishop-19-2}
A.N. Bishop and P. Del Moral. On the stability of matrix-valued Riccati diffusions. Electronic Journal of Probability, vol. 24, pp. 1--40  (2019).

 
\bibitem{nawaf-20}
N. Bou-Rabee, A. Eberle, R. Zimmer.
Coupling and convergence for Hamiltonian Monte Carlo.
Ann. Appl. Probab., vol. 30, no. 3, pp. 1209--1250 (2020).
 

 \bibitem{cai}
 X. Cai, J.D. McEwen, M. Pereyra. Proximal nested sampling for high-dimensional Bayesian model selection. Statistics and Computing, vol. 32, no. 5:87 (2022).
 
 
 \bibitem{chan}
 K. S. Chan, Asymptotic behavior of the Gibbs sampler, J. Amer. Statist. Assoc. vol. 88, pp. 320--326  (1993).
 
 \bibitem{achen}
A.Y.  Chen, K. Sridharan. Optimization, Isoperimetric Inequalities, and Sampling via Lyapunov Potentials. ArXiv preprint arXiv:2410.02979 (2024).

\bibitem{ychen}
Y. Chen, R. Eldan. Localization schemes: A framework for proving mixing bounds for Markov chains. In 2022 IEEE 63rd Annual Symposium on Foundations of Computer Science (FOCS), pp. 110--122. IEEE (2022).

\bibitem{zchen}
Z. Chen, K. Liu, E. Vigoda. Optimal mixing of Glauber dynamics: Entropy factorization via high-dimensional expansion. In Proceedings of the 53rd Annual ACM SIGACT Symposium on Theory of Computing (STOC), pp. 1537--1550 (2021).

\bibitem{chewi-chen}
Y. Chen, S. Chewi, A. Salim, A. Wibisono. Improved analysis for a proximal algorithm for sampling. In Conference on Learning Theory, PMLR, pp. 2984--3014 (2022).	

 
\bibitem{chewi-phd}
S. Chewi. An optimization perspective on log-concave sampling and beyond (Doctoral dissertation, Massachusetts Institute of Technology) (2023).



\bibitem{ciesielski-18}
K.C. Ciesielski, J. Jasinski. Fixed point theorems for maps with local and pointwise contraction properties. Canadian Journal of Mathematics, vol. 70, no. 3, pp. 538--594 (2018).

\bibitem{coulhon}
T. Coulhon, G. Kerkyacharian, and P. Petrushev, Heat kernel generated frames in the setting
of Dirichlet spaces, J. Fourier Anal. Appl., vol. 18, pp. 995--1066 (2012).


\bibitem{damlen1999gibbs}
P. Damlen, J. Wakefield, S. Walker.
Gibbs sampling for Bayesian non-conjugate and hierarchical models by using auxiliary variables. Journal of the Royal Statistical Society: Series B (Statistical Methodology),
vol. 61, no. 2, pp. 331--344 (1999).


\bibitem{bdurmus-19}
V. De Bortoli, A. Durmus. Convergence of diffusions and their discretizations: from continuous to discrete processes and back. ArXiv preprint arXiv:1904.09808 (2019).	

\bibitem{dg-01}
P. Del Moral, A. Guionnet.
On the stability of interacting processes with applications to filtering and genetic algorithms. Annales de l'Institut Henri Poincar\'e Probab. \& Statist. 
vol. 37, no. 2, pp. 155--194 (2001).

\bibitem{dplm-03}
P. Del Moral, M. Ledoux, L. Miclo. On contraction properties of Markov kernels. Probability theory and related fields, vol.126, no. 3, pp. 395--420  (2003).

  \bibitem{dm-04}
P. Del Moral. Feynman-Kac formulae. Springer New York (2004).

 \bibitem{dm-mic-00}
 P. Del Moral, L. Miclo.
 Branching and Interacting Particle Systems Approximations of Feynman-Kac Formulae with Applications to Non-Linear Filtering. S\'eminaire de Probabilit\'es XXXIV, Lecture Notes in Mathematics, vol. 1729, pp. 1--145 (2000).

 \bibitem{dm-mic-03}
 P. Del Moral, L. Miclo.
Particle approximations of Lyapunov exponents connected to Schr\"odinger operators and Feynman-Kac semigroups.  ESAIM: Probabilit\'es et Statistiques,
vol. 7, pp. 171--208 (2003).

 
\bibitem{dm-25}
P. Del Moral.
A Contraction Theory for Sinkhorn and Schr\"odinger Bridges via Log-Sobolev Inequalities.
Stochastic Analysis and Applications, vol. 44, no. 2, pp. 155--230 (2026).


\bibitem{dhj}
P. Del Moral, E. Horton, A. Jasra. On the stability of positive semigroups. 
The Annals of Applied Probability, vol. 33, no. 6A, pp. 4424--4490 (2023).

\bibitem{dj-sinkhorn}
P. Del Moral, A. Jasra.
New Trends in the Stability of Sinkhorn Semigroups.
arXiv preprint arXiv:2601.12633, January (2026), to appear in the Journal of Theoretical Probability.
  
\bibitem{dm-penev-2017}
P. Del Moral and S. Penev. Stochastic Processes: From Applications to Theory. CRC-Press (2016).

\bibitem{dss-19}
P. Del Moral and S.S. Singh. Backward It{\^ o}-Ventzell and stochastic interpolation formulae. arXiv preprint arXiv:1906.09145  (2019).

\bibitem{df-99}
P. Diaconis, D. Freedman. Iterated random functions. SIAM review, vol. 41, no. 1, pp. 45--76 (1999).	

\bibitem{diaconis}
P. Diaconis, L. Saloff-Coste, Separation cut-offs for birth and death chains, Ann. Appl. Probab., vol. 16,  no. 4, pp. 2098--2122 (2006).

\bibitem{ding}
J. Ding, E. Lubetzky, Y.  Peres, Total variation cutoff in birth-and-death chains, Probab. Theory Related Fields, vol. 146  no. 1--2, pp. 61--85 (2010).

\bibitem{kanto-42}
L.V. Kantorovich. On the translocation of masses. C. R. (Doklady)
Acad. Sci. URSS (N.S.), vol. 37, pp. 199--201 (1942).

\bibitem{rachev-84}
S.T. Rachev. The Monge--Kantorovich problem on mass transfer and
its applications in stochastics. Theory of Probability and Its
Applications, vol. 29, no. 4, pp. 647--676 (1984).

\bibitem{dob-56}
R.L. Dobrushin. Central limit theorem for nonstationary Markov
chains. I. Theory of Probability and Its Applications, vol. 1,
no. 1, pp. 65--80 (1956).

\bibitem{dob-70}
R.L. Dobrushin. Prescribing a system of random variables by
conditional distributions. Theory of Probability and Its
Applications, vol. 15, no. 3, pp. 458--486 (1970).

\bibitem{dob-96} 
R.L. Dobrushin. Perturbation methods of the theory of Gibbsian fields. In Lectures on Probability Theory and Statistics: 
\'Ecole d'\'Et\'e de Probabilit\'es de Saint-Flour XXIV -1994. Lecture Notes in Mathematics 1648, pp. 1--66. Berlin: Springer (1996). 


\bibitem{douc}
R. Douc, E. Moulines, P. Priouret, and P. Soulier. Markov Chains. Springer (2019).

\bibitem{durmus-24}
A. Durmus, A. Eberle, A.  Enfroy, A. Guillin, P.  Monmarch\'e. Discrete sticky couplings of functional autoregressive processes, Ann. Appl. Probab., vol. 34,  pp. 5032--5075 (2024).


\bibitem{eberle-ptrf}
A. Eberle. Reflection couplings and contraction rates for diffusions.  Probability theory and related fields, vol. 166, no. 3  pp. 851--886 (2016).

\bibitem{eberle-cras}
A. Eberle.  Reflection coupling and Wasserstein contractivity without convexity. Comptes Rendus Math\'ematique, vol. 349, no. 19--20, pp. 1101--1104  (2011).

 
\bibitem{eberle}
A. Eberle, A. Guillin,  R. Zimmer. Quantitative Harris-type theorems for diffusions and McKean--Vlasov processes. Transactions of the American Mathematical Society, vol. 371, no. 10, pp. 7135--7173 (2019).

\bibitem{eberle-majka}
A. Eberle, M. B. Majka.
Quantitative contraction rates for Markov chains on general state spaces.
Electron. J. Probab., vol. 24, pp. 1--36 (2019)

\bibitem{edelstein-61}
M. Edelstein. An extension of Banach's contraction principle. Proc. Amer. Math. Soc., vol. 12, pp. 7--10 (1961).

\bibitem{edwards}
D.A. Edwards. On the Kantorovich--Rubinstein theorem. Expositiones Mathematicae, vol. 29 pp. 387--398 (2011).


\bibitem{fort-02}
G. Fort. Computable bounds for V-geometric ergodicity of Markov transition kernels. Rapport de Recherche, Univ. J. Fourier, RR 1047-M. (2002).

\bibitem{gelfand}
A. E. Gelfand. Gibbs sampling.
Journal of the American Statistical Association, vol. 95, no. 452. pp. 1300--1304 (2000).

\bibitem{gelfand-2}
 A. E. Gelfand and A. F. M. Smith, Sampling based approaches to calculating marginal densities, J. Amer. Statist. Assoc., vol. 85  pp. 398--409 (1990).



\bibitem{geman}
S. Geman, D.  Geman. Stochastic Relaxation, Gibbs Distributions and the Bayesian Restoration of Images, IEEE Trans. Pattern Anal. Mach. Intell, vol. 6, pp. 721--741 (1984).


\bibitem{glauber}
R.J. Glauber. Time-dependent statistics of the Ising model. Journal of mathematical physics vol. 4, no. 2,  pp. 294--307 (1963).

\bibitem{guan}
Y. Guan,  K. Balasubramanian, S. Ma. Riemannian Proximal Sampler for High-accuracy Sampling on Manifolds. ArXiv:2502.07265 (2025).	

\bibitem{goldys}
B. Goldys, B. Maslowski. Exponential ergodicity for stochastic Burgers and 2D Navier-Stokes equations. Journal of Functional Analysis, vol. 226, no. 1, pp. 230--255 (2005).

\bibitem{goyal}
A. Goyal, G. Deligiannidis, N. Kantas. Mixing Time Bounds for the Gibbs Sampler under Isoperimetry. ArXiv preprint arXiv:2506.22258 (2025).



\bibitem{hairer-mattingly}
M. Hairer and J. C. Mattingly. Yet another look at Harris' ergodic theorem for Markov chains. Seminar on Stochastic Analysis, Random Fields and Applications VI, 109--117, Progr. Probab., 63, Birkh\"auser/Springer Basel AG, Basel (2011).


\bibitem{hairer-mattingly-scheutzow}
M. Hairer, J. C. Mattingly, M.
 Scheutzow. Asymptotic coupling and a general
form of Harris' theorem with applications to stochastic delay equations. Probab. Theory Related
Fields, vol. 149, pp. 223--259 (2011).
 

\bibitem{hobert}
J.P. Hobert, C. J. Geyer. Geometric ergodicity of Gibbs and block Gibbs samplers for a hierarchical random effects model. Journal of Multivariate Analysis, vol. 67, no. 2, pp.  414--430 (1998).


\bibitem{hobert-15}
J.P. Hobert, K. Khare. Computable upper bounds on the distance to stationarity for Jovanovski and Madras's Gibbs sampler. Annales de la Facult\'e des sciences de Toulouse: Mathématiques, vol. 24. no. 4. (2015).


\bibitem{huang-21}
L.J. Huang, M.B. Majka, J. Wang. Strict Kantorovich contractions for Markov chains and Euler schemes with general noise. Stochastic Processes and their Applications, vol. 151, pp. 307--341  (2022).

\bibitem{jerrum}
M. Jerrum, J. B. Son, P. Tetali, E. Vigoda, Elementary bounds on Poincar\'e and log-Sobolev constants for decomposable Markov chains, Ann. Appl. Probab., vol. 14  no. 4, pp. 1741--1765 (2004).

\bibitem{jiaming}
J. Liang, Y. Chen. A proximal algorithm for sampling. arXiv preprint arXiv:2202.13975 (2022), Transactions on Machine Learning Research (2023).


\bibitem{jiaojiao}
J. Fan, B. Yuan, Y. Chen. Improved dimension dependence of a proximal algorithm for sampling. The Thirty Sixth Annual Conference on Learning Theory. PMLR (2023).


\bibitem{joulin}
A. Joulin, Y. Ollivier. Curvature, concentration and error estimates
for Markov chain Monte Carlo. Ann. Probab., vol. 38, no. 6, pp. 2418--2442 (2010).
 
\bibitem{kellerer}
H.G. Kellerer, Duality theorems for marginal problems, Zeitschrift f\" ur Wahrscheinlichkeitstheorie und Verwandte Gebiete
vol. 67 pp. 399--432 (1984).
 
 \bibitem{kellerer2}
 H.G. Kellerer, Duality theorems and probability metrics. In Proceedings of the Seventh Conference on Probability Theory,
Bra\c{s}ov, 1982, VNU Press, Utrecht, pp. 211--220 (1985).
  
  
  \bibitem{kerk-petrus}
  G. Kerkyacharian, P. Petrushev, H. Xu,
Heat kernel generated by the Jacobi operator,
Advances in Mathematics, vol. 226 , pp. 2623--2666 (2011).


  \bibitem{kerk-petrus-2}
  G. Kerkyacharian, P. Petrushev, H. Xu,
Heat kernel generated by the Jacobi operator,
Studia Mathematica, vol. 250, no. 3, pp. 235--252 (2020).

\bibitem{kuwada}
K. Kuwada.
Duality on gradient estimates and Wasserstein controls,
Journal of Functional Analysis, vol. 258; no. 11, pp. 3758--3774 (2010).

\bibitem{lee}
 Y.T. Lee, R. Shen, K. Tian. Structured log-concave sampling with a restricted Gaussian oracle. Conference on Learning Theory, PMLR (2021).


\bibitem{monmarche-25}
L. Liu, M.B. Majka, P. Monmarch\'e. $L^2$-Wasserstein contraction for Euler schemes of elliptic diffusions and interacting particle systems. Stochastic Processes and their Applications, vol. 179, Paper No. 104504 (2025).

\bibitem{luo}
D. Luo, J. Wang. Exponential convergence in $L^p$-Wasserstein distance for diffusion 
processes without uniformly dissipative drift. Math. Nachr.,  vol. 289, no. 14-15, pp. 1909--1926 (2016).

\bibitem{ma}
Y.-A. Ma, Y. Chen, C. Jin, N. Flammarion, M.I. Jordan. 
Sampling can be faster than optimization. Proceedings of the National Academy of Sciences, vol. 116, no. 42, pp.  20881--20885 (2019).



\bibitem{meyn-tweedie-92}
S.P. Meyn and R. L.  Tweedie. Stability of Markovian processes I: Criteria for discrete-time chains. Advances in Applied Probability, vol. 24, no. 3, pp. 542--574 (1992).

\bibitem{meyn-tweedie-93}
S.P. Meyn and R. L.  Tweedie.  Stability of Markovian processes II: Continuous-time processes and sampled chains. Advances in Applied Probability, vol. 25, no. 3, pp. 487--517 (1993).

\bibitem{meyn-tweedie}
S.P. Meyn and R. L.  Tweedie. A survey of Foster-Lyapunov techniques for general state space Markov processes. In Proceedings of the Workshop on Stochastic Stability and Stochastic Stabilization, Metz, France (1993).	

\bibitem{meyn-tweedie-2}
S.P. Meyn and  R. L.  Tweedie.  Markov Chains and Stochastic Stability. Springer Science \& Business Media (2012).

\bibitem{mon-26}
P. Monmarch\'e. Long-time $L^p$ Wasserstein contraction for diffusion
processes without global dissipativity. ArXiv preprint
arXiv:2603.00773 (2026).

\bibitem{monmarche-AHL}
P. Monmarch\'e. 
Wasserstein contraction and Poincar\'e inequalities for elliptic diffusions with high diffusivity. Annales Henri Lebesgue, vol. 6, pp. 941--973 (2023).


\bibitem{mou}
  W. Mou, N. Flammarion, M. J. Wainwright, P. L. Bartlett. An efficient sampling algorithm for non-smooth composite potentials. arXiv:1910.00551v1  (2019), Journal of Machine Learning Research, vol. 23, no. 233  pp. 1--50 (2022).
  
  
  \bibitem{nash}
  J. Nash, Continuity of solutions of parabolic and elliptic equations, American J. Math. vol. 80,
pp 931--954 (1958).
   
\bibitem{peters}
H. J. M. Peters, P. P. Wakker. Convex functions on non-convex domains. Econ Lett, vol. 22, no. 2, pp. 251--255 (1986).
 
 \bibitem{otto}
 F. Otto and C. Villani.
Generalization of an inequality by Talagrand and links with the logarithmic Sobolev inequality,
Journal of Functional Analysis, vol. 173, pp. 361--400 (2000).


\bibitem{Quin}
Q. Qin, J.P. Hobert. Geometric convergence bounds for Markov chains in Wasserstein distance based on generalized drift and contraction conditions. In Annales de l'Institut Henri Poincar\'e (B) Probabilit\'es et Statistiques, vol. 58, no. 2, pp. 872--889 (2022).


\bibitem{rendell2020global}
L.J. Rendell, A.M.  Johansen, A. Lee, N. Whiteley.
Global consensus Monte Carlo. Journal of Computational and Graphical Statistics, vol. 30, no. 2, pp. 249--259 (2021).

\bibitem{rao}
Y. Rao, V.  Roy. Block Gibbs samplers for logistic mixed models: convergence properties and a comparison with full Gibbs samplers. Electronic Journal of Statistics, vol. 15, no. 2, pp.  5598--5625
(2021).

\bibitem{roman}
J. C. Rom\`an, J.P. Hobert. Geometric ergodicity of Gibbs samplers for Bayesian general linear mixed models with proper priors. Linear Algebra and its Applications, vol. 473, pp. 54--77 (2015).


\bibitem{rudolf-18}
D. Rudolf, N. Schweizer. Perturbation theory for Markov chains via Wasserstein distance. Bernoulli, vol.  24, no. 4A, pp. 2610--2639 (2018).




\bibitem{saloff-zuniga}
L. Saloff-Coste, J. Zuniga. Merging for time inhomogeneous finite Markov chains, Part I: Singular values and stability. Electronic Journal of Probability, vol. 14, no. 49, pp. 1456--1494 (2009).

\bibitem{saloff-zuniga-2}
L. Saloff-Coste, J. Zuniga. Merging for time inhomogeneous finite Markov chains, Part II: Nash and Log-Sobolev inequalities. Annals of Probability, vol. 39, no. 3, pp. 1161--1203 (2011).

\bibitem{saloff-zuniga-22}
L. Saloff-Coste, J. Zuniga.  Merging and stability for time inhomogeneous finite Markov chains. In Surveys in Stochastic Processes, EMS Series of Congress Reports, European Mathematical Society, pp. 127--151 (2011).



\bibitem{saloff-zuniga-3}
L. Saloff-Coste, J. Zuniga.
Convergence of some time inhomogeneous Markov chains via spectral techniques. Stochastic processes and their applications, vol. 117, no. 8,  pp. 961--979 (2007).


\bibitem{steinsaltz}
D. Steinsaltz. Locally contractive iterated function systems. Annals of Probability, vol. 27, no. 4, pp. 1952--1979 (1999).

\bibitem{varopoulos}
N. Th. Varopoulos. Small time gaussian estimates of heat diffusion kernels. II. The theory of
large deviations. Journal of Functional Analysis, vol. 93, pp. 1--33 (1990).

\bibitem{vono2020asymptotically}
M. Vono, N. Dobigeon, P. Chainais.
Asymptotically exact data augmentation: Models, properties, and algorithms. Journal of Computational and Graphical Statistics, vol. 30, no. 2, pp.
 335--348  (2021).
 
 \bibitem{vono2019split}
M. Vono, N. Dobigeon, P. Chainais.
Split-and-augmented Gibbs sampler---Application to large-scale inference problems.
IEEE Transactions on Signal Processing, vol. 67, no. 6, pp.1648--1661
  (2019).

  
\bibitem{yan}
M. Yan. Extension of convex function. J Convex Anal, vol. 21, no. 4, pp. 965--987 (2014).

 
\bibitem{Villani}
C. Villani. Optimal transport: old and new. Berlin: Springer (2008).



\end{thebibliography}
\end{document}